%% file: mo2.tex
\documentclass[10pt,journal,onecolumn]{IEEEtran}
\usepackage{times}
\usepackage{color}
\usepackage{bm}
\usepackage{graphicx}
\usepackage{longtable}
 \usepackage{amsopn}
\usepackage[cmex10]{amsmath}
\usepackage{amsthm, amssymb}
\usepackage{cite}

\usepackage{fixltx2e}
\newtheorem{theorem}{Theorem}
\newtheorem{lemma}{Lemma}
\newtheorem{remark}{Remark}
\newtheorem{definition}{Definition}
\newcommand{\comment}[1]{}        
\newcommand{\comm}[1]{}           

\DeclareMathOperator*{\sinc}{sinc}

\def\Xint#1{\mathchoice
{\XXint\displaystyle\textstyle{#1}}%
{\XXint\textstyle\scriptstyle{#1}}%
{\XXint\scriptstyle\scriptscriptstyle{#1}}%
{\XXint\scriptscriptstyle\scriptscriptstyle{#1}}%
\!\int}
\def\XXint#1#2#3{{\setbox0=\hbox{$#1{#2#3}{\int}$}
\vcenter{\hbox{$#2#3$}}\kern-.5\wd0}}
\def\tildeint{\Xint-}

\begin{document}
\title{On the output of nonlinear systems excited by discrete prolate spheroidal sequences}

\author{Kyle~Q.~Lepage
        and~ShiNung Ching,~\IEEEmembership{Member~IEEE}
\thanks{K. Q. Lepage is with the Department
of Mathematics and Statistics, Boston University, Boston, MA, USA e-mail: (lepage@math.bu.edu).}
\thanks{S. Ching is with the Department of Electrical and Systems Engineering, Washington 
University in St. Louis, St. Louis, MO, USA email (shinung@ese.wustl.edu).}}

\markboth{Preprint submitted to arXiv}%
 {Egapel \MakeLowercase{\textit{et al.}}: Reduced Mean-Square Error 
   Quadratic Inverse Spectrum Estimator}
\maketitle 
 \begin{abstract}
   \input{mimo_abs2}
 \end{abstract}
\section{Introduction}
  \input{i7}
\section{Background \& Preliminaries}\label{bg}
For convenience Table \ref{tbl:sys_funcs} summarizes the key parameters used to establish the theory developed in 
Sections \ref{sect:result}, \ref{sect:ortho}.
\begin{table}

	\centering

		\begin{tabular}{ | l | l | p{5.2cm} |}

			\hline
      Eqn. & Symbol & Description  \\ \hline \hline & & \\
      \eqref{eqn:w} & $W \in \mathbb{R}$ & The half-bandwidth of the DPSS energy-concentrated frequency interval. \\ \hline & & \\
      \eqref{eqn:lambda} & $\lambda_{min}$ & The minimum of the input DPSS in-band energy ratios. \\ \hline & & \\
      \eqref{eqn:key_gamma1} &  $\Gamma^{(Q)}_{m,*}$ & Suprema of the $Q^{\rm th}$ order system response. \\ & & \\ \hline & &  \\
      \eqref{eqn:key_V*} & $V_{M,*}$ & Suprema of the Magnitudes of the input DPSWFs \\ \hline & & \\
      \eqref{eqn:in-band-sys-rem} & $\Gamma^{(Q)'}_{m,*}$ & Suprema of the $Q^{\rm th}$ order system repsonse's in-band Taylor remainder. \\ \hline & & \\
      \eqref{eqn:gamma_**} & $\Gamma_{m,**}^{(Q)}( {\bf 0}, f )$ & Suprema of the $Q^{\rm th}$ order input-DC system responses at output frequency $f$. \\ \hline & & \\
		\eqref{eqn:gamma_**_prime} & $\Gamma^{(1)'}_{m,**}$ & Suprema of the derivative of the $1^{\rm st}$ order responses with respect to frequency. \\ \hline

		\end{tabular}\label{tbl:sys_funcs}

	\caption{Key Parameters}

	\label{tab:KeyParameters}

\end{table}

\subsection{Discrete Prolate Spheroidal Sequences (DPSSs)}
  \input{dpss}
\subsection{Volterra Expansion of Nonlinear MIMO System} \label{sec:content1}
  \input{mimo}

\section{Suppression of Higher-Order MIMO Response to DPSS Input} \label{sec:content2}
  \input{r6}

\section{Orthogonality of MIMO Response to DPSS Input} 
  \input{ortho2EH}

\section{Linear Narrowband Identification}\label{sect:egs}
  \input{nb_sys_id}
  \subsection{Linear Narrowband Identification: Simulation}
  \input{nb_sys_id_sim}

\section{Discussion \& Conclusions}\label{sect:disc}
  \input{disc3}




\appendices
\section{Generalized Frequency Response Function}
  \input{app_gfrf}
\section{Out-of-band Response of the $Q^{\rm th}$ Order Volterra Kernel}
  \input{app_a}
\section{In-band Response of the $Q^{\rm th}$ Order Volterra Kernel}
  \input{app_b3}
 \bibliographystyle{IEEEtran}
 \bibliography{qmt,cs,sc_bib,sc_introneurorefs} 

\end{document}

%% file: mimo_abs2.tex
The discrete prolate spheroidal sequences (DPSSs) -- a set of optimally bandlimited sequences with unique properties --  
are important to applications in both science and engineering. 
In this work, properties of nonlinear system response due to DPSS excitation are reported.
In particular, this output is
shown to be approximately orthogonal after passing through a nonlinear,
multiple-input multiple-output system with memory under quite general conditions.  This work
quantifies these conditions in terms of constraints upon the higher-order generalized transfer
functions characterizing the Volterra expansion of a MIMO system, the Volterra order of the system, 
and the DPSS bandwidth parameter $W$ and time-bandwidth parameter $NW$.  
The approximate system output orthogonality allows
multiple-input, multiple-output parameter identification of edge structure 
in interconnected nonlinear systems using simultaneous, DPSS excitation.
This narrowband method of system identification is particularly appealing when compared to 
classical broadband system excitation in sensitive, neural engineering applications involving
electrical stimulation of multiple brain regions.
{\color{black} Through the comparison of inner-product  and kernel-based narrowband detectors, the 
utility of this work is demonstrated when identifying narrowband system response of 
a third-order Volterra system from noisy observations.}



\comment
{
We show that the 
prolate spheroidal wave functions form an effective excitation in this 
setting due to their frequency concentration and orthogonility.  Specifically, we demonstrate 

To demonstrate the utility of these properties, the theory is applied to 
To demonstrate the efficacy of the excitation, we provide an 
example of parameter identification in a nonlinear, biophysical model of multi-site neural stimulation.

In this work,  
Here, we develop further orthogonality properties of these functions 
for the purpose of multiple-input, multiple-output parameter identification of edge structure 
in interconnected nonlinear systems.  Direct motivation is taken from neural engineering applications 
involving electrical stimulation of multiple brain regions.  The delicate nature of this application makes 
classical identification paradigms involving broadband excitation inappropriate.  We show that the 
prolate spheroidal wave functions form an effective excitation in this 
setting due to their frequency concentration and orthogonility.  Specifically, we demonstrate 
that the functions remain approximately orthogonal, even after tranformation through a 
nonlinear system under quite general conditions.  We quantify these conditions using 
a volterra expansion of the MIMO system.  
The key conditions center on the order of the 
system, the extent of system memory, and the prolate function bandwidth 
parameter $W$.  To demonstrate the efficacy of the excitation, we provide an 
example of paraemter identification in a nonlinear, biophysical model of multi-site neural stimulation.

{{\bf OLD:}
\color{red}
The prolate spheroidal functions -- a set of special functions with unique properties --
are important to a number of areas in both science and engineering.  
In this work, further
orthogonality properties of the prolate spheroidal functions are reported.  In particular,
these functions are shown to be approximately orthogonal after passing through  a nonlinear,
multiple-input multiple-output system with memory under quite general conditions.  This work
quantifies these conditions in terms of constraints upon the higher-order transfer
functions characterizing the volterra expansion of the MIMO system, the order of the system, the
extent of system memory, the prolate function bandwidth parameter $W$, and
the derivative of the spheroidal prolate wave functions.
These effects are detailed and quantified mathematically and provide a general framework upon which
to base application.  The paper ends with an application to two problems in engineering.
}
} 


%% file: i7.tex
\label{intro}
\label{sect:intro}
Characterization of dynamical systems 
is an important aspect of many areas of science and engineering \cite{Astrom1971123,box2005statistics}.  
Many inputs have been employed for system identification.  Perhaps the two
most prominent input categories  are (i) uncorrelated broadband input, or (ii) narrowband input \cite{Haber1990651}.
Broadband inputs distribute input signal energy evenly over the entire frequency domain \cite{lee1965measurement,Astrom1971123}, 
while narrowband inputs focus signal energy into a small frequency interval 
\cite{victor1979nonlinear,worden1997harmonic,bedrosian1971output,evans1992design,schoukens1998parametric,schoukens2005identification}.

Situations exist where the question is simply, ``If a narrowband stimulus is applied as input, is there a narrowband response?'' \cite{worden1997harmonic}. This situation
arises in neuroscience in both the experimental and clinical settings.  In the clinical 
setting 
there is a risk that stimulation will alter human cognitive ability   
\cite{Boggio2005,perlmutter2006deep,Moscrip2006,DeRidder2007,gallentine2009intraoperative,Vanneste2010,Veniero2011}.
In the experimental setting, stimulation may change the neural system under study, rendering subsequent inference invalid \cite{robertson2003studies}, or cause discomfort \cite{mcfadden2011reducing}.  In all cases, it is desirable to minimize the input energy when inferring, or producing, narrowband response.  

In the growing field of network science  it is important to estimate the connectivity between the inputs and outputs of complicated systems 
\cite{Wasserman1994,barabasi1999emergence,amaral2000classes}.  Here knowledge of the network interconnectivity can be informative while
detailed knowledge of the system dynamics is prohibitively difficult to obtain.  
{\color{black} In neuroscience this literature relates network properties to cognitive state, experimental 
condition, and pathology \cite{chouinard2003modulating,Massimini15052007,keeser2011prefrontal,Greicius2003,Matsumoto01102004}.}

When attempting to detect a narrowband connection between system input and output, it is necessary to 
separate the relative influences of the inputs upon any one output.
Generally, the strategies employed have resorted to the excitation of one input channel at a time {\color{black} \cite{matsumoto2004functional,matsumoto2007functional,shafi2012exploration,lepage2013}}.

{\color{black}In this paper, a theoretical characterization of the response of a Volterra MIMO system to DPSS input is 
provided. This characterization enables a strategy based upon multiple, simultaneous stimulation.
The specific contributions are:} 
\begin{enumerate}
\item \label{contrib:1} The discovery that each of the DPSSs, once passed through a nonlinear Volterra MIMO system, can be approximated by a quadratic generalized frequency response
Volterra system representation \cite{schetzen2006volterra} in the frequency domain, and by a linearly-transformed version of the input in the time domain.
\item \label{contrib:3} The verification that the DPSSs, passed through a nonlinear Volterra  MIMO system, remain approximately orthogonal. 
\item  \label{contrib:4}The conditions under which Contributions \ref{contrib:1} and \ref{contrib:3} are valid, along with a quantification of approximation error.
\item  An inner product based identification scheme exploiting Contributions \ref{contrib:1} - \ref{contrib:4} to estimate the 
  linear narrowbannd connectivity of a nonlinear Volterra MIMO system.
\end{enumerate}
Contribution 4  makes use of an inner-product detector to separate the relative influence upon the system output due to multiple 
simultaneous narrowband MIMO system inputs.  This facilitates simultaneous network stimulation and connectivity inference. 

The DPSSs and their continuous time analogues, the prolate spheroidal wave functions, are used in many areas.  Example areas and publications include: time-series analysis \cite{rodbell199915,thomson2000multitaper}, signal processing \cite{thomson1982spectrum,bronez1992performance,he2010canonical,lepage2014reduced,davenport2012compressive}, communication engineering \cite{1433150,sejdic2008channel}, theoretical physics \cite{flammer1957spheroidal,li2004spheroidal}, and
control \cite{li2011ensemble}.  

Orthogonality plays a prominent role in system identification.  In \cite{lee1965measurement} the Wiener G-functionals result from a modified
Gram-Schmidt orthogonalization procedure, facilitating cross-correlation based system identification.  
In \cite{korenberg1988exact}, the kernels are orthogonalized with respect to the observed input.  In \cite{perez_sysid_legendre,Karanam1978,PARASKEVOPOULOS1978, paraskevopoulos1983chebyshev, van2005system} the Volterra kernels are
expanded in terms of an orthogonal basis, allowing the conversion of a system of integrals to a linear-in-parameters algebraic equation.  After conversion, 
parameters can be estimated by least squares.  Standard bases for these expansions are provided by the Laguerre, Walsh, block-pulse, and Chebyshev functions (a generalized basis
is considered in \cite{van2005system}). 



{\color{black}The design of narrowband excitation has received considerable attention: \cite{victor1979nonlinear,worden1997harmonic,bedrosian1971output} 
\cite{evans1992design,schoukens1998parametric,schoukens2005identification,Carassale_volterra2010}.
}
In \cite{victor1979nonlinear,worden1997harmonic,bedrosian1971output,Carassale_volterra2010} full system identification is addressed.
{\color{black} In this work, 
  motivated by 
  the identification of linear narrowband response,
  emphasis is placed upon characterizing Volterra MIMO response to DPSS input.
} 
In \cite{evans1992design,schoukens1998parametric,schoukens2005identification} models are considered where the higher-order nonlinear system response
is small and additive.  This differs from the proposed work in that (i) sinusoidal input with random phases are used as opposed to the optimal in-band energy concentrated
DPSSs and (ii) the {\color{black} nonlinear system response in the proposed work is suppressed by the use of DPSS as input.}

{\color{black} By characterizing the Volterra MIMO response to DPSS input, progress is made towards developing improved methods
  of identifying linear narrowband MIMO Volterra response.
To the best of our knowledge this is the first characterization of the response of a MIMO Volterra system to DPSS input,
and the first use of the DPSSs for the purpose of linear, narrowband system identification of a nonlinear MIMO Volterra system.}
 

The remainder of this paper is organized as follows. Following a review of relevant properties of the DPSS in \ref{sect:dpss}, and nonlinear 
MIMO systems with memory in \ref{sect:mimo}, the main results of the paper are 
presented in \ref{sect:result} \& \ref{sect:ortho}.  In \ref{sect:sim}, DPSS excitation is used, in a comparison with sum of sinusoid and 
uncorrelated input to detect differences between the narrowband response of two $3^{\rm rd}$ order Volterra systems.  
The paper ends with a discussion in \ref{sect:disc}.


%% file: dpss.tex
\label{sect:dpss}
  The zeroth-order discrete prolate spheroidal sequence (DPSS), or Slepian sequence \cite{SlepianV}, $ v_t^{(0)}$, is the infinite, real-valued sequence 
  that is index-limited to $[0,N-1]$ and possesses the maximum fractional in-band energy
  concentration of all such sequences \cite{SlepianV}.\footnote{Index-limited to $[0,N-1]$ means that
   any  
   sequence element  
 outside of the index-set $\left\{0, \ 1, \ \ldots, \ N-1\right\}$ is zero.}
  \begin{equation}
    v_t^{(0)} = \arg \max_{\mathbb{V}_N} \frac{\int_{-W}^{W} \left| V(f)\right|^2 df}{\int_{-\frac{1}{2}}^\frac{1}{2} \left| V(f) \right|^2 df} .
    \label{eqn:w}
  \end{equation}
  where $\mathbb{V}_N$ denotes the space of all infinite, real-valued sequences index-limited to $[0, N-1]$, 
$W<\frac{1}{2}$ denotes the half-bandwidth  (where, without loss of generality, a sampling period of 1 is assumed) and 
$V(f)$ is the discrete Fourier transform of an element of $\mathbb{V}_N$.  
The higher-order 
DPSSs are the maximally in-band energy concentrated sequences that are mutually orthogonal and 
are orthogonal to $ v_t^{(0)}$.\footnote{Note, $v^{(0)} \in \ell^2(\infty)$.  By truncation, $v^{(0)}$ defines $\tilde{v}^{(0)} \in \ell^2(N)$. By
  discrete Fourier transform $v^{(0)}$ defines the zeroth order DPSWF, $V \in C^2(-\frac{1}{2}, \frac{1}{2})$, and $\tilde{V} \in C^2(-W,W)$ is defined by the restriction of $V$
  to $(-W,W)$.  Each of these elements satisfy orthogonality relations with respect to the canonical inner-product associated with the 
Hilbert space to which they belong:  $\ell^2(\infty)$, $\ell^2(N)$, $C^2(-W,W)$ and $C^2(-\frac{1}{2},\frac{1}{2})$.}
  The $ k^{th}$ DPSS, $v_t^{(k)}$, satisfies eigenvalue/eigenvector equations in both the time and frequency domains \cite{Slepian1965,Slepian1961,Slepian1978}.  
  The discrete Fourier transform of the $k^{th}$ DPSS   is  $V_k(f)$, the $k^{th}$ discrete prolate spheroidal wave
  function (DPSWF), 
  \begin{eqnarray}
    V_k(f) &=& \sum_{t = -\infty}^\infty v_t^{(k)} e^{-i 2 \pi f t} \ , \\
           &=& \sum_{t=0}^{N-1} v_t^{(k)} e^{-i 2 \pi f t} \ .
  \end{eqnarray}
  The $k^{th}$ DPSWF 
   satisfies the frequency domain eigenfunction equation 
  \begin{equation}
    \lambda_k V_k(f) = \int_{-W}^W D_N(f-f') \ V_{k}(f') \ df' \ .
    \label{eqn:lambda}
  \end{equation}
    Here $D_N(f)$ is a Dirichlet (`sinc')-type kernel,
    \begin{equation}
       D_N( f ) =  
      \frac{ \sin{ N \pi f }}{\sin{ \pi f }} \  e^{-i \pi f (N-1)}  \ ,
    \end{equation}
    and
  $\lambda_k$ is the eigenvalue associated with the $k^{th}$ DPSS. 
  The $ k^{th}$ eigenvalue is near one (i.e., $\lambda_k\approx1$) for $k$ less than approximately $2NW$ (twice the dimensionless 
      time-bandwidth product).  The eigenvalues, $\lambda_k$, $k = 1, \ 2, \ldots$  monotonically decrease with increasing $k$
       and are equal to the fraction of the DPSS
  energy within the $ (-W, W)$ band of frequencies
  \cite{Thomson:1982,SlepianV}.  That is 
  \begin{equation}\label{eqn:inbandconc0}
  \int_{-W}^W \left| V_{k} \right|^2 \ df' = \lambda_k \ .
  \end{equation}
  Thus the in-band signal energy of $v_t^{(k)}$ is $\lambda_k$.
  The DPSWFs are in-band orthogonal, i.e.,
  \begin{equation}
      \int_{-W}^W V_{k}(f) V^*_{k'}(f) \ df = \lambda_k \delta_{k,k'} \ .
      \label{eqn:relevant}
   \end{equation}
  An additional inequality used to bound the in-band inner-product of two DPSWFs is
  \begin{equation}\label{eqn:inbandconc}
     \int_{-W}^W \left| V_{k'}(f') V_k^*(f')\right| df' \leq \sqrt{ \lambda_{k'} \lambda_k } \leq 1 \ .
  \end{equation}
  The DPSSs are functions of two parameters: the length, $N$, of the signal, and the user specified half-bandwidth parameter, $W$.  
  For convenience, these dependencies are implied. 

%% file: mimo.tex
\label{sect:mimo}
Let $\mathcal{H}$ be a nonlinear, time-invariant system formulated in discrete time with $M$ inputs and $M'$ outputs.  Assume that $\mathcal{H}$ admits a Volterra expansion \cite{Bussgang1974} such that the $m^{th}$ output, $y_{m,t}$, evaluated at time-index $t$, can be expressed as,
\begin{eqnarray}
  \label{eqn:volterra_td0}
  \lefteqn{y_{m,t}  =}&& \hspace{0.5cm} y_{m}^{(o)} + \nonumber \\
    && \sum_{m'=1}^M \sum_{t'=1}^\infty \gamma^{(1)}_{m,m',t'} \ u_{m',t-t'} \ + \nonumber \\
          && \sum_{m_1,m_2=1}^M \sum_{t_1,t_2=1}^\infty \gamma^{(2)}_{m,m_1,m_2,t_1,t_2} \ u_{m_1,t-t_1} u_{m_2,t-t_2} \ + \nonumber \\
          && \hspace{2cm} \ldots + \nonumber \\
          && \sum_{m_1,\ldots,m_Q=1}^M \sum_{t_1,\ldots,t_Q=1}^\infty \gamma^{(Q)}_{m,m_1,\ldots,m_Q,t_1,\ldots, t_Q} \ \prod_{j=1}^Q u_{m_j,t-t_j}  \ . \nonumber \\
\end{eqnarray}
Here, $\gamma^{(a)}_{m,m_1, \ldots, m_a, t_1, \ldots, t_a} \in \mathbb{R}$, $\left| \gamma^{(a)}_{m,m_1, \ldots, m_a, t_1, \ldots, t_a} \right| < \infty$,
is a finite, order $a$ volterra kernel.  It  relates the product of $a$ inputs from channels $m_1$ to $m_a$ 
to the system output, $y_{m,t}$, on channel $m$ at time-index $t$.  When forming this product input channel $    m_j$ is 
evaluated at time-index $t_j$, for $    j = 1, \ldots, \ a$.
The number of input channels $M$ need not equal the
number of output channels $M'$.  The system, $\mathcal{H}$, is completely characterized by the collection of 
Volterra kernels $\gamma^{(j)}_{m,m_1, \ \ldots,\ m_j, \ t'}$.

To facilitate development, (\ref{eqn:volterra_td0}) is re-written in terms of vectors.  Specifically,
\begin{eqnarray}
  \label{eqn:volterra_td}
  \lefteqn{y_{m,t}  =}&& \hspace{0.5cm} y_{m}^{(o)} + \nonumber \\
    && \sum_{m'=1}^M \sum_{t'=1}^\infty \gamma^{(1)}_{m,m',t'} \ u_{m',t-t'} \ + \nonumber \\
          && \sum_{{\bf m}_2={\bf 1}_2}^M \sum_{{\bf t}_2={\bf 1}_2}^\infty \gamma^{(2)}_{m,{\bf m}_2,{\bf t}_2} \ u_{m_1,t-t_1} u_{m_2,t-t_2} \ + \nonumber \\
          && \hspace{2cm} \ldots + \nonumber \\
          && \sum_{{\bf m}_Q = {\bf 1}_Q}^M \sum_{{\bf t}_Q = {\bf 1}_Q }^\infty \gamma^{(Q)}_{m,{\bf m}_Q,{\bf t}_Q} \ \prod_{j=1}^Q u_{m_j,t-t_j}  \ , \nonumber \\
\end{eqnarray}
where ${\bf m}_Q = [ m_1, \ m_2 \ , \ldots \ ,\ m_Q ]^T$, ${\bf t}_Q = [ t_1, \ t_2 \ , \ldots \ t_Q ]^T$, and ${\bf 1}_Q$ is a $Q$ dimension vector
of ones. For convenience, in the following $y_m^{(0)}$ is set equal to zero {\color{black} (see Remark \ref{rem:DC}).}
In Section \ref{sect:result} investigation focuses upon the nature of $y_{m,t}$ when
the input to the system is specified to be the DPSSs.  This investigation is facilitated by the 
frequency domain representation of (\ref{eqn:volterra_td}).  Following the development in \cite{george1959continuous,Lang2007805},
in Appendix \ref{app:gfrf} it is shown that:
\begin{eqnarray}
  \lefteqn{Y_m(f) =} && \hspace{1cm} \sum_{q=1}^Q T_{m,q}(f)  \ , \hspace{5cm} 
  \label{eqn:volterra_fd}
\end{eqnarray}
where
\begin{eqnarray}
  \lefteqn{T_{m,1}(f) =}  && \hspace{1cm} \sum_{m'=1}^M \Gamma^{(1)}_{m,m'}(f) \ U_{m'}(f) \ , \hspace{3.2cm} 
\end{eqnarray}
\begin{eqnarray}
  \lefteqn{T_{m,2}(f) =}\hspace{1cm} &&
    \sum_{m_1,m_2=1}^M \int_{-\frac{1}{2}}^\frac{1}{2} \Gamma^{(2)}_{m,m_1,m_2}(f_1, f-f_1) \ \times \hspace{1.0cm}  \nonumber \\
          && \hspace{2.5cm} U_{m_1}(f_1) U_{m_2}(f-f_1) \ df_1 \ , \nonumber \\
\end{eqnarray}
and
\begin{eqnarray}
  \lefteqn{T_{m,Q}(f) =}&& \nonumber \\
            && \sum_{ {\bf m}_Q = {\bf 1} }^M \int_{-\frac{1}{2}}^\frac{1}{2} 
                        \Gamma^{(Q)}_{m, {\bf m}_Q}\left( {\bf f}_{Q-1}, f - {\bf f}_{Q-1}^T {\bf 1}_{Q-1} \right) \nonumber \\
        && \hspace{.1cm} U_{m_Q}\left( f - {\bf f}_{Q-1}^T {\bf 1}_{Q-1} \right) 
              \prod_{j=1}^{Q-1} U_{m_j}\left( f_j \right)  \ d{\bf f}_{Q-1} \ . \nonumber \\
              \label{eqn:tmqf}
\end{eqnarray}
Here, ${\bf f}_{Q} = \left[ f_1 \ f_2 \ \ldots \ f_Q \right]^T$, and 
the Volterra kernels are specified to be causal.  The discrete Fourier transform is taken over all time.  Specifically: 
\begin{equation}
  U(f)  = \sum_{t=-\infty}^\infty u_t e^{-i 2 \pi f t} \ .
\end{equation}
As in the time domain, the collection of generalized
frequency response functions, $\Gamma^{(j)}_{m,{\bf m}_j}( {\bf f}_{j-1}, f - {\bf f}_{j-1}^T {\bf 1}_{j-1} )$,
completely specify the nonlinear, time-invariant MIMO system $\mathcal{H}$.

%% file: r6.tex
\label{sect:result}
The main results center on investigating the orthogonality of the outputs of the nonlinear system $\mathcal{H}$ when its inputs are set to be the DPSSs.  
Under conditions to be described,
  the higher-order system responses to DPSS inputs are effectively suppressed.  Recall the Volterra expansion of $\mathcal{H}$ (\ref{eqn:volterra_td}) and
set the $k^{th}$ input to $v^{(k)}_t$, the $k^{\rm th}$ order DPSS for $k \leq K$. 
Here, $K$ is chosen such that $v^{(k)}_t$, $k \leq K$ possesses energy concentration within $(-W,W)$ near one.  From \eqref{eqn:inbandconc}, this is equivalent to
specifying that the first $K$ DPSS eigenvalues, $\lambda_k$, are as close to $1$ as possible. Thus, of all DPSSs index-limited to the
interval $[0,N-1]$ with a given time-bandwidth
parameter $NW$, these sequences are the most energy concentrated DPSSs within the frequency interval $(-W,W)$.
When the number of input channels $M$ is greater than $K$, the remaining $M-K$ channel inputs are set to zero.
In this way, each input is specified to be a strongly in-band energy-concentrated DPSS, or is otherwise set
to zero and does not contribute to the output.  Consider the $Q^{\rm th}$ order Volterra kernel frequency response, $T_{m,Q}(f)$, 
due to this DPSS input:
\begin{eqnarray}
  \lefteqn{T_{m,Q}(f) =}&& \nonumber \\
            && \sum_{ {\bf m}_Q = {\bf 1} }^M \int_{-\frac{1}{2}}^\frac{1}{2} 
                        \Gamma^{(Q)}_{m, {\bf m}_Q}\left( {\bf f}_{Q-1}, f - {\bf f}_{Q-1}^T {\bf 1}_{Q-1} \right) \nonumber \\
        && \hspace{.1cm} V_{m_Q}\left( f - {\bf f}_{Q-1}^T {\bf 1}_{Q-1} \right) 
              \prod_{j=1}^{Q-1} V_{m_j}\left( f_j \right)  \ d{\bf f}_{Q-1} \ . \nonumber \\ \label{QorderVolterra}
\end{eqnarray}
Decompose \eqref{QorderVolterra} into in-band and out-of-band components,
\begin{equation}
 T_{m,Q}(f) = T^{(i)}_{m,Q}(f) + T^{(o)}_{m,Q}(f) \ , 
\end{equation}
where
\begin{eqnarray}
  \lefteqn{T^{(i)}_{m,Q}(f) =}&& \nonumber \\
 && \sum_{ {\bf m}_Q = {\bf 1} }^M \int_{-W  }^W
                        \Gamma^{(Q)}_{m, {\bf m}_Q}\left( {\bf f}_{Q-1}, f - {\bf f}_{Q-1}^T {\bf 1}_{Q-1} \right) \nonumber \\
        && \hspace{.1cm} V_{m_Q}\left( f - {\bf f}_{Q-1}^T {\bf 1}_{Q-1} \right) 
              \prod_{j=1}^{Q-1} V_{m_j}\left( f_j \right)  \ d{\bf f}_{Q-1} \ , \nonumber \\
\end{eqnarray}
and
\begin{eqnarray}
  \lefteqn{T^{(o)}_{m,Q}(f) =}&& \nonumber \\
 && \sum_{ {\bf m}_Q = {\bf 1} }^M \tildeint_{-\frac{1}{2} }^\frac{1}{2}
                        \Gamma^{(Q)}_{m, {\bf m}_Q}\left( {\bf f}_{Q-1}, f - {\bf f}_{Q-1}^T {\bf 1}_{Q-1} \right) \nonumber \\
        && \hspace{.1cm} V_{m_Q}\left( f - {\bf f}_{Q-1}^T {\bf 1}_{Q-1} \right) 
              \prod_{j=1}^{Q-1} V_{m_j}\left( f_j \right)  \ d{\bf f}_{Q-1} \ . \nonumber \\
\end{eqnarray}
The symbol $\tildeint_{-\frac{1}{2}}^\frac{1}{2} \ dx$ is used to specify the sum of two integrals.  It is 
equal to $\int_{-\frac{1}{2}}^{-W} \ dx \ + \ \int_W^{\frac{1}{2}} \ dx$.
In Appendix (\ref{app:a}) it is shown that,
\begin{eqnarray}
  \lefteqn{ \left| T_{m,Q}(f) - T^{(i)}_{m,Q}(f) \right| \leq} \hspace{4cm} A_{m,M,Q}(\lambda_{min}, V_{M,*}, \Gamma^{(Q)}_{m,*} ) \ ,\nonumber\\
 \label{eqn:bnd_a}  
\end{eqnarray}
where 
\begin{eqnarray}
  \lefteqn{A_{m,M,Q}(\lambda_{min}, V_{M,*}, \Gamma^{(Q)}_{m,*} ) =}&&\nonumber \\
    && \hspace{1cm}  \left( 1 - \lambda_{min} \right)^{(Q-1)/2} V_{M,*} M^Q \Gamma_{m,*}^{(Q)} \ . \label{Qorderfirstbound}
\end{eqnarray}
Here $V_{M,*}$ and $\Gamma_{m,*}^{(Q)}$ are respectively, bounds on the magnitude of the DPSWF \eqref{eqn:lambda} and on the magnitude of the $Q^{\rm th}$ order
Volterra kernel frequency response \eqref{QorderVolterra}.  These bounds are suprema over the appropriate domains:
\begin{eqnarray}
  \lefteqn{\Gamma^{(Q)}_{m,*} =} && \nonumber \\
    && \sup\limits_{\substack{ {\bf m}_Q \in \left\{1,2,\ldots,M\right\}^Q \\ {\bf f}_{Q-1} \in (-\frac{1}{2},\frac{1}{2})^{Q-1} \\ f \in (-\frac{1}{2},\frac{1}{2})} } \bigg| \Gamma^{(Q)}_{m, {\bf m}_Q}\left( {\bf f}_{Q-1}, f - {\bf f}_{Q-1}^T {\bf 1}_{Q-1} \right) \bigg| \ , \nonumber \\
  \label{eqn:key_gamma1}
\end{eqnarray}
and
\begin{eqnarray}
  V_{M,*} &=& \sup\limits_{\substack{      m_Q \in \left\{1,2,\ldots,M\right\}   \\ {\bf f}_{Q-1} \in (-\frac{1}{2},\frac{1}{2})^{Q-1} \\ f \in (-\frac{1}{2},\frac{1}{2})} } 
  \bigg| V_{m_Q}\left( f - {\bf f}_{Q-1}^T {\bf 1}_{Q-1} \right) \bigg| \ . \nonumber \\
  \label{eqn:key_V*}
\end{eqnarray}
From \eqref{Qorderfirstbound}, when the product $V_{M,*} \Gamma_{m,*}^{(Q)}$ is small relative to $M^Q  \left( 1 - \lambda_{min} \right)^{(Q-1)/2}$ the $Q^{\rm th}$ order
Volterra frequency response is approximated by $T^{(i)}_{m,Q}(f)$.  To offer some perspective, for DPSS with $NW = 5$, 
$N = 200$, $K = 6$, {\color{black} yields $1-\lambda_{min} \approx 7 \times 10^{-5}$}.  Here $W = NW/N$ is equal to $0.025$.

Seeking to approximate the system's $Q^{\rm th}$ order response in terms of it's response at zero input-frequency, 
consider the in-band portion of the system's $Q^{\rm th}$ order frequency response, $T^{(i)}_{m,Q}$.
To this end, let 
\begin{eqnarray}
T^{(i)}_{0,m,Q}(f) &=& \sum_{{\bf m}_Q = {\bf 1}}^M \Gamma_{m,{\bf m}_Q}^{(Q)}({\bf 0}_{Q-1}, f) \ J( W, Q, f, {\bf m}_Q ) \ ,
    \nonumber \\
  \label{eqn:T_approx}
\end{eqnarray}
\comment
{
\begin{eqnarray}
  \lefteqn{ T^{(i)}_{0,m,Q}(f) = } && \nonumber \\
& & 
  \sum_{{\bf m}_Q = {\bf 1}}^M \Gamma_{m,{\bf m}_Q}^{(Q)}({\bf 0}_{Q-1}, f) \ J( W, Q, f, {\bf m}_Q ) \ ,
    \nonumber \\
  \label{eqn:T_approx}
\end{eqnarray}
}
where
\begin{eqnarray}
    \lefteqn{J(W,Q,f,{\bf m_Q}) = }&& \nonumber \\
  && \int_{-W}^W 
    V_{m_Q}(f - {\bf f}_{Q-1}^T {\bf 1}_{Q-1} ) \prod_{j=1}^{Q-1} V_{m_j}(f_j) df_j \ . \nonumber \\
  \label{eqn:JJ}
\end{eqnarray}
It is useful to define the upper-bound, $J_B$:
\begin{eqnarray}
  \lefteqn{\left| J(W,Q,f,{\bf m_Q}) \right| \leq }&& \hspace{2.05cm} J_B(   W,Q,M) \ , \nonumber \\
                                                      &=& \sup_{\substack{{\bf m}_Q \in \left\{1,\ \ldots, \ M\right\}^{Q} \\ f \in (-W,W) }} \nonumber \\
                                                      && \int_{-W}^W \left| \ 
    V_{m_Q}(f - {\bf f}_{Q-1}^T {\bf 1}_{Q-1} ) \prod_{j=1}^{Q-1} V_{m_j}(f_j) \ \right| df_j \ . \nonumber \\
    \label{eqn:abs_JJ_ub}
\end{eqnarray}
Let \begin{eqnarray}
 \lefteqn{\Gamma^{(Q)'}_{m,*} =} && \nonumber \\
  && \sup_{\substack{ {\bf m}_Q \in \left\{1,\ \ldots, \ M\right\}^{Q} \\ f \in (-W,W) \\ {\bf f}_{Q-1} \in \left(-W,W\right)^{Q-1}}} \nonumber \\
  && \int_0^1 \bigg| \nabla_{{\bf f}_{Q-1}}\bigg\{ 
        \Gamma^{(Q)}_{m, {\bf m}_Q}\left( t {\bf f}_{Q-1}, f - t {\bf f}_{Q-1}^T {\bf 1}_{Q-1} \right) \bigg\} \bigg|\ dt \ . \nonumber \\
  \label{eqn:in-band-sys-rem}
\end{eqnarray}
By Taylor expanding (Appendix (\ref{app:b})), it can be seen that,
 \begin{eqnarray}
   \lefteqn{ \left| T^{(i)}_{m,Q}(f) - T^{(i)}_{0,m,Q}(f) \right| \leq} && \hspace{3.5cm} B_{m,M,Q}(W,\Gamma_{m,*}^{(Q)'}) \ , \nonumber \\
     \label{eqn:bnd_b}
 \end{eqnarray}
 where
 \begin{eqnarray}
   \lefteqn{B_{m,M,Q}( W,\Gamma_{m,*}^{(Q)'})  =}&& \nonumber \\
                                                 && \hspace{2.0cm} W^{Q-1} \ \Gamma^{(Q)'}_{m,*} \ M^Q \ J_B(W,Q,M) \ . \nonumber \\
   \label{eqn:B}
 \end{eqnarray}
Combining (\ref{eqn:bnd_a}) and (\ref{eqn:bnd_b}) obtain,
\begin{eqnarray}
  \lefteqn{ \left| T_{m,Q}(f) - T^{(i)}_{0,m,Q}(f) \right| \leq} && \nonumber \\
    && \hspace{.1cm} A_{m,M,Q}(\lambda_{min}, V_{M,*}, \Gamma^{(Q)}_{m,*} ) + B_{m,M,Q}(W,\Gamma_{m,*}^{(Q)'}) \ . \nonumber \\ 
    \label{eqn:Ti_approx}
\end{eqnarray}
Thus, (\ref{eqn:T_approx}) approximates the $Q^{\rm th}$ order Volterra response to DPSS input.
Contributing to (\ref{eqn:T_approx}) is the $Q-1$ dimensional integral over the DPSWFs, $J(W,Q,f,{\bf m}_Q)$
specified in (\ref{eqn:JJ}). Because 
\begin{equation}
  \int_{-W}^W | V_j (f') |^2 \ df' = \lambda_j  \ \approx 1 \ , 
\end{equation}
\begin{eqnarray}
  \left| J(W,Q,f,{\bf m}_Q ) \right| &\leq& J_B(W,Q,M)  \ , \nonumber \\
                                     &\approx& \left( 2 W\right)^{\frac{ Q-2}{2}} \ , \ Q > 1 \ . 
  \label{eqn:JJ_approx}
\end{eqnarray}
{\color{black} Figure (\ref{fig:crude_est}) shows \eqref{eqn:JJ_approx}, $\rm J_B$ and $\max\limits_{f \in (-\frac{1}{2}, \frac{1}{2})} |J(W,Q,f,{\bf m}_Q)|$
for 
 various $Q$ and 
randomly chosen ${\bf m}_Q$.} 
\footnote{Here, $NW = 4$, $N =256$ (black curve) or $N=1000$ (red curve), for $Q = \left\{ 3, \ 4, \ 5,\ 6 \right\}$ and $M = 6$.  Twenty
five random draws (with replacement) from the sequence with elements one through six are made.  The $i^{th}$ draw, ${\bf m}_Q^{(i)}$, is
used to compute $ J(   W,Q,f,{\bf m}_Q^{(i)})$ for $f$ varied over the interval $(-W,W)$.  For each 
$\max\limits_{f \in (-W,W)} \left|  J(   W,Q,f,{\bf m}_Q^{(i)}) \right|$ a mark in Fig. (\ref{fig:crude_est}) is made.}
Eqn. \eqref{eqn:JJ_approx} approximately (exactly in the plot) upper-bounds $\left| J(   W,Q,f, {\bf m}_Q^{(i)}) \right|$ for simulated 
conditions.  
Using \eqref{eqn:JJ_approx}, the in-band $Q^{\rm th}$ order Volterra system response can be bounded:
\begin{figure}
  \begin{center}
  \includegraphics[width=3in]{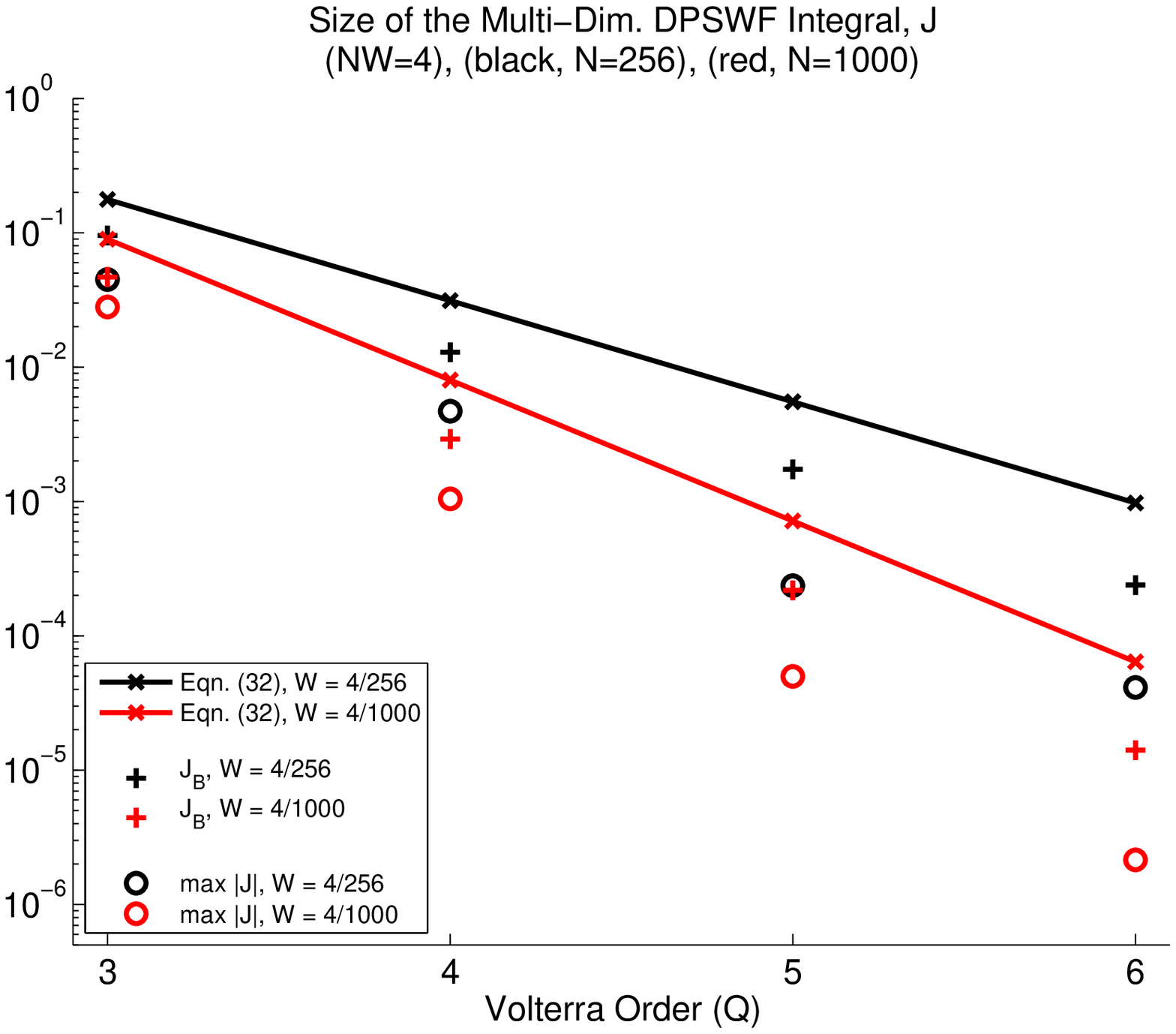}
    \caption{\label{fig:crude_est} {\bf The integral specified in (\ref{eqn:JJ}),  $ J(   W,Q,f,{\bf m}_Q)$, is approximately 
upper-bounded by (\ref{eqn:JJ_approx}).}
      In this simulation $J(   W,Q,f,{\bf m}_Q)$ is numerically computed for $N=256$ (black) and $N=1000$ (red) as a function of Volterra order $Q$.
      The multi-index ${\bf m}_Q$ is picked randomly $25$ times for each order and $N$.
      The frequency $f$ is varied over the interval $(-W,W)$.  Each mark corresponds to the maximum of the absolute value of $J$ (circles) over these
      frequencies, or the value of the bound $J_B$ (pluses).  The solid lines are computed using the approximation (\ref{eqn:JJ_approx}) and usefully approximate an 
    upper-bound for the actual maximum of the absolute integrals.  
    } 
  \end{center}
\end{figure}
\begin{eqnarray}
  \lefteqn{ \left| T^{(i)}_{0,m,Q}(f) \right| \leq } && \nonumber \\
&& 
  \sum_{{\bf m}_Q = {\bf 1}}^M \bigg| \Gamma_{m,{\bf m}_Q}^{(Q)}({\bf 0}, f) \bigg| \bigg| J(W,Q,f,{\bf m}_Q) \bigg| \ , \nonumber \\
&\leq&
   \left( 2W \right)^{\frac{Q-2}{2}} \sum_{{\bf m}_Q = {\bf 1}}^M \bigg| \Gamma_{m,{\bf m}_Q}^{(Q)}({\bf 0}, f) \bigg| + \delta' \ , \nonumber \\
&\leq&
  \left( 2W \right)^{\frac{Q-2}{2}} \ M^Q \ \Gamma_{m,**}^{(Q)}({\bf 0}, f) + \delta' \ . \nonumber \\
  \label{eqn:T_bnd1}
\end{eqnarray}
Here 
\begin{eqnarray}
  \lefteqn{ \Gamma_{m,**}^{(Q)}( {\bf 0}, f ) = }&& \nonumber \\
  && \sup\limits_{{\bf m}_Q \in \left\{1,\ 2, \ldots, \ M\right\}^Q} \ \left| \Gamma_{m,{\bf m}_Q}^{(Q)}({\bf 0}, f ) \right| \ . \nonumber \\
  \label{eqn:gamma_**}
\end{eqnarray}
In the following the case  $\delta' = 0$ is explicitly investigated.
Combining (\ref{eqn:Ti_approx}) and (\ref{eqn:T_bnd1}), 
\begin{eqnarray}
  \lefteqn{ \left| T_{m,Q}(f) \right| \leq } && \hspace{1.2cm} \nonumber \\
    && C_{m,M,Q}(f,\lambda_{min},W,V_{M,*},\Gamma^{(Q)}_{m,*},\Gamma_{m,*}^{(Q)'}, \Gamma^{(Q)}_{m,**}({\bf 0}, f)) \ , \nonumber \\
\end{eqnarray}
\begin{eqnarray}
  \lefteqn{ \left| T_{m,Q}(f) \right| \leq} && \hspace{1.2cm} \nonumber \\
    && \hspace{0.25cm} A_{m,M,Q}(\lambda_{min}, V_{M,*}, \Gamma^{(Q)}_{m,*} )+ 
    B_{m,M,Q}(W,\Gamma_{m,*}^{(Q )'})  \ + \nonumber \\
  && \hspace{1cm} \left( 2 W \right)^\frac{Q-2}{2}\  M^Q \ \Gamma_{m,**}^{(Q)}({\bf 0},f) \ , \nonumber \\
  &=& \left( 1 - \lambda_{min} \right)^{Q/2} V_{M,*} M^Q \Gamma_{m,*}^{(Q)} \ + \nonumber \\
   &&  2^\frac{Q-2}{2}W^{\frac{3Q}{2}-2} M^Q \Gamma_{m,*}^{(Q )'} \ + \nonumber \\
  && \left( 2W \right)^\frac{Q-2}{2} \ M^Q \ \Gamma_{m,**}^{(Q)}({\bf 0},f) \ , \ \ \ Q \geq 2 \ .\nonumber \\
  \label{eqn:tmq_f_bnd}
\end{eqnarray}
Thus the $Q^{\rm th}$ order Volterra kernel system response due to the DPSS inputs is bounded by a sum of 
three positive terms.  The first of these terms is due to the restriction of the integrals to $(-W,W)$.  This term
is largest when the DPSWFs are
poorly concentrated within $(-W,W)$. In this case $\lambda_{min}$ is small. 
The second term in (\ref{eqn:tmq_f_bnd}) results from the
Taylor expansion of the truncated integrals.
 The third term bounds the in-band contribution to the $Q^{\rm th}$ order Volterra kernel response.  
 Both the second and third term feature a further $W$ dependence owing to
the shifted multi-dimensional integral of the product DPSWFs (\ref{eqn:JJ}), independent of system properties.
For a fixed number of system inputs $M$, and a fixed $\lambda_{min}$,
\begin{equation}
  \left| T_{m,Q}(f) \right| = O\left( \left(  W \right)^\frac{Q-2}{2} \right) \ .
\end{equation}
Since $W<\frac{1}{2}$, the higher-order responses are suppressed relative to the linear and quadratic responses when the Volterra system is 
driven with DPSS input.
\comment
{
Recognizing that $D_{1,*}$ is zero, to investigate the effect on $C_{m,M,Q}(N,NW)$ due to the DPSS parameters independent of system effects, consider
the situation where $\Gamma_{m,*}^{(Q)}({\bf 0}, f)$ and $\Gamma_{m,**}({\bf 0},f)$ are equal to one independent of $Q$.  Further, assume that 
the system responses match the bounds.  In this scenario, a plot of 
\begin{eqnarray}
  \lefteqn{C'_{M,Q}(N,NW)  \times \frac{ V_{M,*} \sqrt{1-\lambda_{min}} + V_{M,**}}{M^{Q-1}} = } && \hspace{1.5cm}                 \ , \nonumber \\
    &&  V_{M,*} \left( 1 - \lambda_{min} \right)^{Q/2}  + 
    2^{Q-1}       W^{2Q-2} +     \nonumber \\ 
    && V_{M,**}
    \bigg[ v_0^{(*)} + \sqrt{ 1 - \lambda_{min} } \bigg]^{Q-1}
      \ , \nonumber \\
\label{eqn:cprime}
\end{eqnarray}
displays the relevant contribution to the system output due to the Volterra kernels of differing orders with respect to the system output due
to the linear component ($Q=1$). Here the difference in contribution
is due solely to the response of the system to features of the DPSS inputs.
In Fig. (\ref{fig:c_m_M_Q}), $C'_{M,Q}(N,NW)$ is plotted for various $NW$, $N$, and $M$ as the Volterra kernel order $Q$ is varied.
\begin{figure}
  \begin{center}
    \includegraphics[width=3.5in]{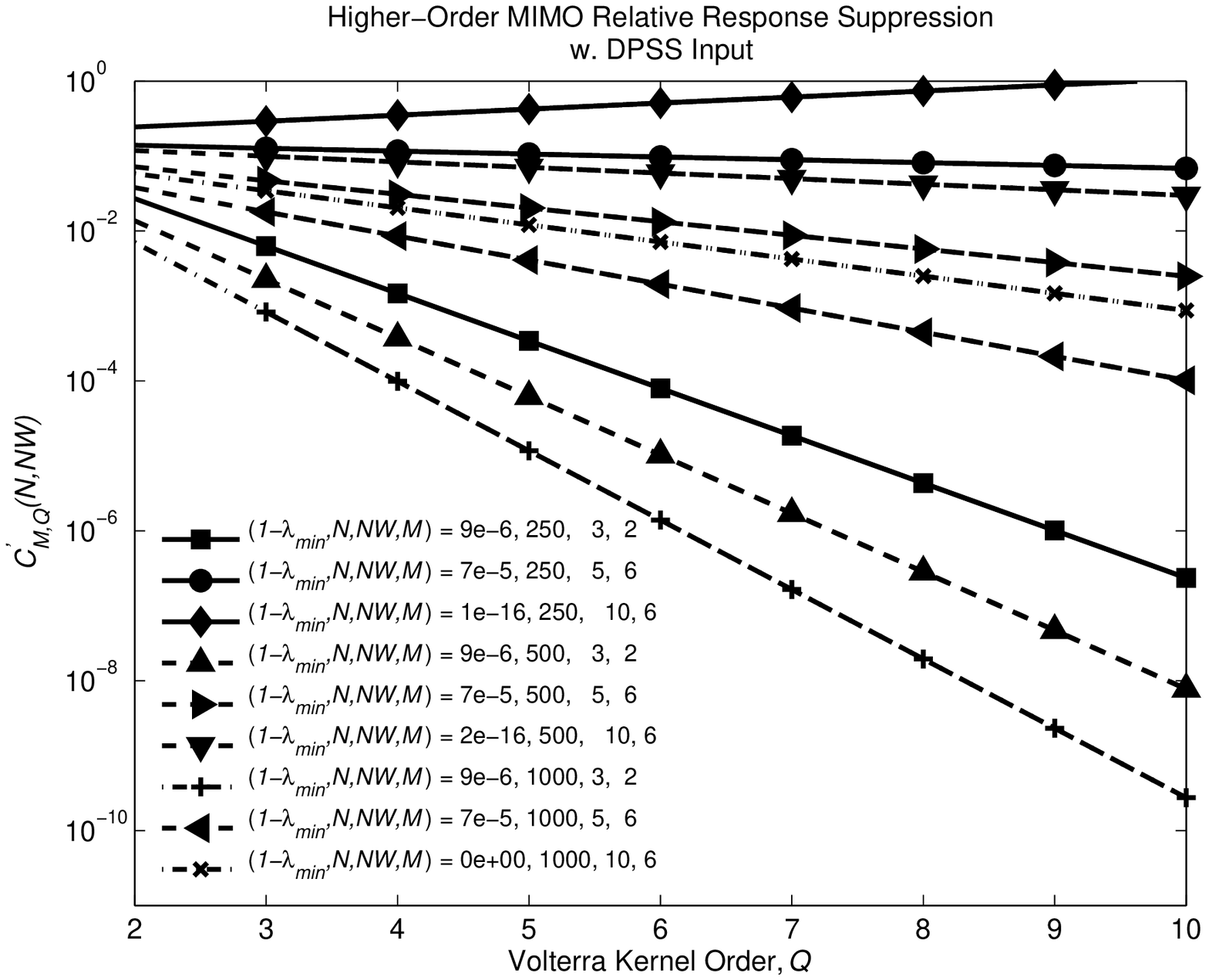}
    \caption{ \label{fig:c_m_M_Q}
      The quantity $C'_{M,Q}(N,NW)$ specified in (\ref{eqn:cprime}) is plotted 
      for various $N$, $NW$, $M$, with changing $Q$.  This quantity summarizes the system independent
contribution to the $Q^{\rm th}$ order Volterra output when the system-dependent quantities are set to unity
and the contribution bounds are assumed to be met.  
      The higher-order system response to DPSS input is an order of magnitude reduced
      relative to the linear system response for certain parameter values.  
      As expected, this suppression increases with increasing input length, $N$, with
      decreasing $W$, decreasing $1-\lambda_{min}$ and decreasing $M$.  Note the saturation effects as 
      different terms dominate.
    } 
  \end{center}
\end{figure}
    The higher-order system response to DPSS input is at least an order of magnitude reduced
      relative to the linear system response (Fig. (\ref{fig:c_m_M_Q})) for suitably chosen $M$.  
As expected, this suppression increases with increasing input length, $N$, with
      decreasing $W$ and with decreasing $1-\lambda_{min}$, though note that there is a saturation effect as different terms dominate.
} 
\comment
{
  Further note that the absolute bound on the $Q^{\rm th}$ order response is related to $C'_{M,Q}(N,NW) \ \times \  \sqrt{1-\lambda_{min}}$.
  Thus the absolue bounds on the system response can be obtained from Fig. (\ref{fig:c_m_M_Q}) by multiplication with $\sqrt{1-\lambda_{min}}$.
}
\comment
{
In fact,
noting that $D_{1,*} = 0$, and ignoring the system specific contributions to the bounds, the ratio of the
$Q = 2$ bound to the $Q=1$ bound is,
\begin{eqnarray}
  \lefteqn{ R_{2,1} = } && \nonumber \\
  && 
    \frac{
  M^2 \left( 1 - \lambda_{min} \right) + 
            2 M^2 W^{2} + 
  M^2  \left(       v_0^{(*)} + \left( 1 - \lambda_{min} \right) \right)}
  { 
      M \left( 1 - \lambda_{min} \right)^{1/2} + 2 M   
  } \ , \nonumber \\
 &&
    \frac{
  M  \bigg[  2 \left( 1 - \lambda_{min} \right) + 
            2 W^{2} + v_0^{(*)} \bigg]
  }{ 
        \left( 1 - \lambda_{min} \right)^{1/2} + 2     
  } \ , \nonumber \\
  &\approx&
  M  \bigg[    \left( 1 - \lambda_{min} \right) + 
              W^{2} + v_0^{(*)}/2 \bigg]
       \ . \nonumber \\
\end{eqnarray}
For $N = 200$, $NW = 4$, and $K=6$, $R_{1,2} \approx $.
} 

\comment
{
Taking the supremum over the system kernel contributions out of the sums and then taking the limit $Q \rightarrow \infty$,
obtain,
{\tiny
\begin{eqnarray}
  \lefteqn{ \lim\limits_{Q\rightarrow \infty} \left| W_m^{(Q)}(f) - T_{m,1}(f) \right| \leq   } && \nonumber \\
  && 
  V_{M,*} \Gamma_{m,*}^{(*)} \frac{1 - \left[ 1 - M \left( 1 - \lambda_{min} \right)^{1/2} \right] - 
      M\sqrt{1-\lambda_{min}}\left(1 - M \sqrt{1 - \lambda_{min}}\right) }
{ 1 - M \left( 1 - \lambda_{min} \right)^{1/2} }              + \nonumber \\
  &&
  D_{*,*} 2^{-1} W^{-2} 
    \frac{ 1 - \left( 1 - 2 MW^2 \right) \left( 1 + 2^2 M^2 W^4 \right)}
    { 1 - 2 M W^2 } + \nonumber \\
  && V_{M,**} \Gamma_{m,**}^{(*)} \bigg[ \nonumber \\
  && \left( v_0^{(*)} \right)^{-1} \frac{ 1 - \left( 1 - M v_0^{(*)} \right) \left( 1 + M^2 \left( v_0^{(*)}\right)^2 \right)}
      { 1 - M  v_0^{(*)} }
             + \nonumber \\
  && ( 1 - \lambda_{min} )^{-1/2} \frac{ 1 - \left( 1 - M \sqrt{ 1 - \lambda_{min}}\right) \left( 1 + M^2 (1-\lambda_{min})\right) }
        { 1 - M \sqrt{1 - \lambda_{min}}} \bigg]
                \ , \nonumber \\
\end{eqnarray}
} 
for values of $W$, $M$, $\lambda_{min}$ such that the geometric series converge.
Simplifying,
{\tiny
\begin{eqnarray}
  \lefteqn{ \lim\limits_{Q\rightarrow \infty} \left| W_m^{(Q)}(f) - T_{m,1}(f) \right| \leq   } && \nonumber \\
  && 
  V_{M,*} \Gamma_{m,*}^{(*)} \frac{ M^2 \left(  1 - \lambda_{min} \right) }
{ 1 - M \left( 1 - \lambda_{min} \right)^{1/2} }              + \nonumber \\
  &&
  D_{*,*} W^{-2} MW^2  
    \frac{   1 - 2   M   W^2 + 4 M^2 W^4 } 
    { 1 - 2 M W^2 } + \nonumber \\
  && V_{M,**} \Gamma_{m,**}^{(*)} \bigg[ \nonumber \\
  && M 
      \frac{ 1 - M v_0^{(*)} + M^2 \left( v_0^{(*)}\right)^2  }
      { 1 - M  v_0^{(*)} }
             + \nonumber \\
  && M \frac{ 1 - M \sqrt{ 1 - \lambda_{min}} + M^2 (1-\lambda_{min})        }
        { 1 - M \sqrt{1 - \lambda_{min}}} \bigg]
                \ , \nonumber \\
  &\approx&
        V_{M,*} \Gamma_{m,*}^{(*)} M^2 \left(  1 - \lambda_{min} \right)  + \nonumber \\
          && D_{*,*}        M + \nonumber \\
        && V_{M,**} \Gamma_{m,**}^{(*)} M \ .
\end{eqnarray}
} 
Thus, linear system response requires conditions on the Volterra kernels as the system order $Q$ tends to infinity.
\end{proof}
} 

\comment
{
\hrule
From Parseval's theorem\footnote{A consequence of the fact that inner products are invariant to unitary transformation.}, the
inner product, $I$, of the $m^{\rm th}$ output with the $m'^{\rm th}$ output is equal to 
\begin{eqnarray}
  I_{m,m'} &=& \int_{-\frac{1}{2}}^\frac{1}{2} W_m(f) W^*_{m'}(f) \ df \ , \\
      &=& \sum_{q,q'=1}^Q \int_{-\frac{1}{2}}^\frac{1}{2} T_{m,q}(f) T^*_{m',q'}(f) \ df \ , \\ 
      &=& \sum_{q,q'=1}^Q J_{m,m',q,q'} \ . 
  \label{eqn:ip}
\end{eqnarray}
Each $J_{m,m',q,q'}$ is the inner-product between the order $q$ volterra system output on channel $m$ with
the order $q'$ volterra system output on channel $m'$. 
{\it
The sense in which $I_{m,m'} \approx 0$ is quantified in terms of $W$ the 
effective support of the discrete prolate spheroidal wave functions (DPSWFs) in the frequency domain,
the extent of history dependence of the volterra kernels, $\gamma$, the order of the volterra expansion, $Q$,
and the number of inputs, $M$.  
}

Eqn. (\ref{eqn:ip}) specifies the inner-product between output channels of a general volterra model of a MIMO system.  
In this paper focus it placed on $I_{m,m'}$ when the system inputs $U_{m_j}$ are specified to equal $V_{m_j}$ the 
order $m_j$ DPSWF.
\begin{eqnarray}
   \lefteqn{J_{m,m',q,q'} =}&& \nonumber \\
       && \int_{-\frac{1}{2}}^\frac{1}{2} \sum_{\substack{ {\bf m}_q ={\bf 1}_q \\ {\bf m}'_{q'} ={\bf 1}_{q' }}}^M 
                              \int_{-\frac{1}{2}}^\frac{1}{2} \int_{-\frac{1}{2}}^\frac{1}{2} \times \nonumber \\
       &&    \Gamma^{(q)}_{m, {\bf m}_q}\left({\bf f}_{q-1}, f - {\bf f}_{q-1}^T {\bf 1}_{q-1} \right) \times \nonumber \\
     &&      \Gamma^{(q')*}_{m',{\bf m}'_{q'-1}}\left( {\bf f'}_{q-1}, f - {\bf f'}_{q-1}^T {\bf 1}_{q'-1}  \right) \times \nonumber \\
       &&                       V_{m_q}\left(  f - {\bf f}_{q-1}^T {\bf 1}_{q-1} \right) \prod_{j=1}^{q-1} V_{m_j}( f_j ) \times \nonumber \\
       &&                       V^*_{m'_{q'}}\left(  f' - {\bf f'}_{q'-1}^T {\bf 1}_{q'-1} \right) \prod_{j'=1}^{q'-1} V_{m'_{j'}}^*( f'_{j'} ) \times \nonumber \\
       &&                       d{\bf f}_{q-1} d{\bf f'}_{q'-1} \ df \ .
        \label{eqn:bigJ}
\end{eqnarray}
Further restrict $m_j \in \mathbb{V}$, the set of DPSWF orders such that $\lambda_{m_j}$ is near one. When the number
of input channels, $M$, exceeds $|\mathbb{V}|$, set the remaining inputs to zero.
With this additional restriction (\ref{eqn:bigJ}) can be bounded. 
Let $Ji_{m,m',q,q',{\bf m}_q,{\bf m'}_{q'}}({\bf f}_q, {\bf f'}_{q'} )$ be the
integral in (\ref{eqn:bigJ}) over $f$,  
\begin{eqnarray}
  \lefteqn{Ji_{m,m',q,q',{\bf m}_q, {\bf m'}_{q'} }( {\bf f}_q, {\bf f'}_{q'} ) = }&& \nonumber \\
    && \int_{-\frac{1}{2}}^\frac{1}{2} 
       \Gamma^{(q)}_{m,{\bf m}_q }\left({ \bf f}_{q-1}, f - {\bf f}_{q-1}^T {\bf 1}_{q-1} \right) \times \nonumber \\
       && \Gamma^{(q')*}_{m',{\bf m'}_{q'} }\left({\bf f'}_{q-1}, f - {\bf f'}_{q'-1}^T {\bf 1}_{q'-1} \right) \times \nonumber \\
            && V_{m_q}( f - {\bf f}_{q-1}^T {\bf 1 } ) V^*_{m'_{q'}}( f - {\bf f'}_{q'-1}^T {\bf 1}_{q'-1} ) \ df \ . \nonumber \\
    \label{eqn:Ji}
\end{eqnarray}
Let 
\begin{eqnarray}
    \lefteqn{Ji^{(W)}_{m,m',q,q',{\bf m}_q, {\bf m'}_{q'}}({\bf f}_q, {\bf f'}_{q'})  = } && \nonumber \\
      &&\kappa_{m,m',q,q',{\bf m}_q, {\bf m'}_{q'} }({\bf f}_q, {\bf f'}_{q'}) \ \times \nonumber \\
      && \int_{-W}^W V_{m_q}(x) V^{*}_{m'_{q'}}(x) \ dx \ ,
  \label{eqn:jiw_kappa}
\end{eqnarray}
where 
\begin{eqnarray}
  \lefteqn{\kappa_{m,m',q,q',{\bf m}_q, {\bf m'}_{q'}} ({\bf f}_q, {\bf f'}_{q'} )  =} && \nonumber \\
    && \Gamma^{(q)}_{m,{\bf m}_q }\left( {\bf f}_{q-1}, -{\bf f}_{q-1}^T {\bf 1}_{q-1} \right) \times \nonumber \\
    && \hspace{.2cm} \Gamma^{(q')*}_{m',{\bf m'}_{q'}  }\left({\bf f'}_{q'-1},  {\bf f}_{q-1}^T{\bf 1}_{q-1} -  {\bf f'}_{q'-1}^T {\bf 1}_{q'-1} \right)  \ .
      \nonumber \\
\end{eqnarray}

We are now in a position to relate the inner-integral in (\ref{eqn:bigJ}) to an inner product between the DPSWFs input on 
channels $m_q$ and $m'_{q'}$. 
\begin{theorem}{$Ji^{(W)}$ approximates $Ji$\newline}
  \label{th:Ji_approx}
  Let $B_{m'_{q'},1}$ be the bound specified in (\ref{app:deriv_bnd}) for the derivative of the magnitude of the $m'_{q'}$ order DPSWF, and
  let 
$C_{m,m', {\bf m}_q, {\bf m'}_{q'}  }({\bf f}_q, {\bf f'}_{q'}, W)$ be as specified in (\ref{eqn:Cbig}).
  Then,
\begin{eqnarray}
    \lefteqn{\left| Ji_{m,m',q,q',{\bf m}_q, {\bf m'}_{q'}}({\bf f}_q,{\bf f'}_{q'})
           - Ji^{(W)}_{m,m',q,q',{\bf m}_q, {\bf m'}_{q'}}({\bf f}_q,{\bf f'}_{q'} )\right| \leq } &  \hspace{2cm}  \nonumber \\
      &&  B_{m'_{q'},1} \ \sqrt{\lambda_{m_q}} \ \left( e - 1 \right) \ \times \nonumber \\
      &&  \hspace{1cm} C_{m,m', {\bf m}_q, {\bf m'}_{q'}}({\bf f}_q, {\bf f'}_{q'}, W) \ . \nonumber \\
\end{eqnarray}
\end{theorem}
The proof of Thm. (\ref{th:Ji_approx}) is provided in 
(\ref{app:ji_w}).
\begin{figure*}[h]
\begin{eqnarray}
  \lefteqn{C_{m,m', {\bf m}_q, {\bf m'}_{q'} }( {\bf f}_q, {\bf f'}_{q'}, W ) = } &&\hspace{2cm}  \hspace{0cm}     \nonumber \\
      &  \bigg[ 
                \left| 
                  \Gamma_{m,{\bf m}_q }\left( {\bf f}_{q-1}, \ -{\bf f}_{q-1}^T {\bf 1}_{q-1}  \right) 
                  \Gamma^*_{m',{\bf m'}_{q'} }\left( {\bf f}_{q'-1}, \ -{\bf f}_{q'-1}^T {\bf 1}_{q'-1}  \right) 
                \right|^2 + \nonumber \\
      &         \left| 
                  \Gamma_{m,{\bf m}_q }\left( {\bf f}_{q-1}, W - {\bf f}_{q-1}^T {\bf 1}_{q-1} \right) 
                  \Gamma^*_{m',{\bf m'}_{q'} }\left( {\bf f'}_{q'-1}, W - {\bf f'}_{q'-1}^T {\bf 1}_{q'-1} \right) 
                \right|^2 - \nonumber \\
      &         2 \cos{ \left( \phi(W) - \phi(0) \right) } \times \nonumber \\
      &         \left| 
                  \Gamma_{m,{\bf m}_q }\left( {\bf f}_{q-1}, \ -{\bf f}_{q-1}^T {\bf 1}_{q-1} \right)
                  \Gamma^*_{m',{\bf m'}_{q'} }\left( {\bf f'}_{q'-1}, \ -{\bf f'}_{q'-1}^T {\bf 1}_{q'-1} \right)
                 \right| \times \nonumber \\
      &         \left| \Gamma_{m,{\bf m}_q }\left( {\bf f}_{q-1}, W-{\bf f}_{q-1}^T {\bf 1}_{q-1} \right) 
                        \Gamma^*_{m',{\bf m'}_{q'} }\left( {\bf f'}_{q'-1}, W - {\bf f}_{q'-1}^T {\bf 1}_{q'-1} \right) 
                 \right| \bigg]^\frac{1}{2} \ .\nonumber \\
\end{eqnarray}
\caption{ \label{eqn:Cbig}
      This is the distance in the complex plane between a vector associated with frequency ${\bf 0}$ and a vector associated with 
frequency ${\bf W}$.  If the product transfer function doesn't change between ${\bf 0}\ {\rm cps}$ and ${\bf W} \ {\rm cps}$, then
$C$ is zero.  Hence, $C$ tends to zero as $W$ tends to zero.  
} 
\end{figure*}
Some comments on the size of this bound are appropriate.  
When $m_q \neq m'_{q'}$, the inner-integral is zero in the sense specified by Thm. (\ref{th:Ji_approx}).
While this prevents the different input channels from mixing via the inner integral, the outer integrals also contribute to $J_{m,m',q,q'}$:
\begin{eqnarray}
   \lefteqn{J_{m,m',q,q'} \approx}&& \nonumber \\
       && \sum_{\substack{ {\bf m}_q ={\bf 1} \\ {\bf m}'_{q'} ={\bf 1}}}^M \int_{-W}^W V_{m_q}(x) V^{*}_{m'_{q'}}(x) \ dx \times \nonumber \\
        && \int_{-\frac{1}{2}}^\frac{1}{2} \int_{-\frac{1}{2}}^\frac{1}{2} 
        \kappa_{m,m',q,q',{\bf m}_q, {\bf m'}_{q'} }({\bf f}_q, {\bf f'}_{q'}) \ \times \nonumber \\
      && \prod_{j=1}^{q-1} V_{m_j}( f_j ) \prod_{j'=1}^{q'-1} V_{m'_{j'}}^*( f'_{j'} ) \ d{\bf f}_{q-1} d{\bf f'}_{q'-1} \ .
        \label{eqn:bigJ2}
\end{eqnarray}
In general, the higher-order Volterra kernels, $q,q' > 1$, involve all of the input channels.  This inclusion allows any of the 
input channels to contribute to $J_{m,m',q,q'}$.  Because of the 
in-band energy concentration of the DPSWFs the integrals in (\ref{eqn:bigJ2}) can be taken over the interval $(-W,W)$ to obtain,
\begin{eqnarray}
   \lefteqn{J_{m,m',q,q'} \approx}&& \nonumber \\
       && \sum_{\substack{ {\bf m}_q ={\bf 1} \\ {\bf m}'_{q'} ={\bf 1}}}^M \int_{-W}^W V_{m_q}(x) V^{*}_{m'_{q'}}(x) \ dx \times \nonumber \\
        && \int_{-W}^W \int_{-W}^W 
        \kappa_{m,m',q,q',{\bf m}_q, {\bf m'}_{q'} }({\bf f}_q, {\bf f'}_{q'}) \ \times \nonumber \\
      && \prod_{j=1}^{q-1} V_{m_j}( f_j ) \prod_{j'=1}^{q'-1} V_{m'_{j'}}^*( f'_{j'} ) \ d{\bf f}_{q-1} d{\bf f'}_{q'-1} \ .
        \label{eqn:bigJ3}
\end{eqnarray}
When $\kappa_{m,m',q,q',{\bf m}_q,{\bf m'}_{q'}}({\bf f}_q, {\bf f'}_{q'} )$ is slowly varying over the hyper-cube, $(-W,W)^{q+q'-2}$, 
$\left| J_{m,m',q,q'} \right|$ is small for $q,q' > 2$, restricting the influence of input channels on $I_{m,m'}$.
\begin{theorem}{Higher-order Volterra Kernel Suppression\newline}
  \label{thm:ho_supp}
Let 
\begin{equation}
   B^{(J)}_{m,m',q,q'} = \alpha_* M^{q+q'} s_{q+q'}(N, NW) \ , \nonumber \\
  \label{eqn:jbnded}
\end{equation}
as defined in Appendix (\ref{app:ho_supp}).
Then
  \begin{equation}
    \left| J_{m,m',q,q'} \right| \leq  B^{(J)}_{m,m',q,q'} \ .
  \end{equation}
\end{theorem}
The proof, provided in Appendix (\ref{app:ho_supp}), follows by Taylor expanding and truncating the integrals to the frequency interval $(-W,W)$.  
Theorem (\ref{thm:ho_supp}) provides a bound on the contribution to the inner product between output channels $m$ and $m'$ by
the interaction between the $q^{\rm th}$ and $q^{'{\rm th}}$ order Volterra kernels.  This bound, $B^{(J)}_{m,m',q,q'}$, depends upon
a kernel specific quantity, $\alpha_*$, the number of input channels $M$, and a quantity, $s_{q,q'}(N,NW)$ related to the DPSWFs.
A plot of  $s_{q,q'}(N,NW)$ as a function of $q+q'$ for various $N$ and $NW$ is provided in Fig. \ref{fig:s}.
\begin{figure*}
  \begin{center}
     \includegraphics[width=3.5in]{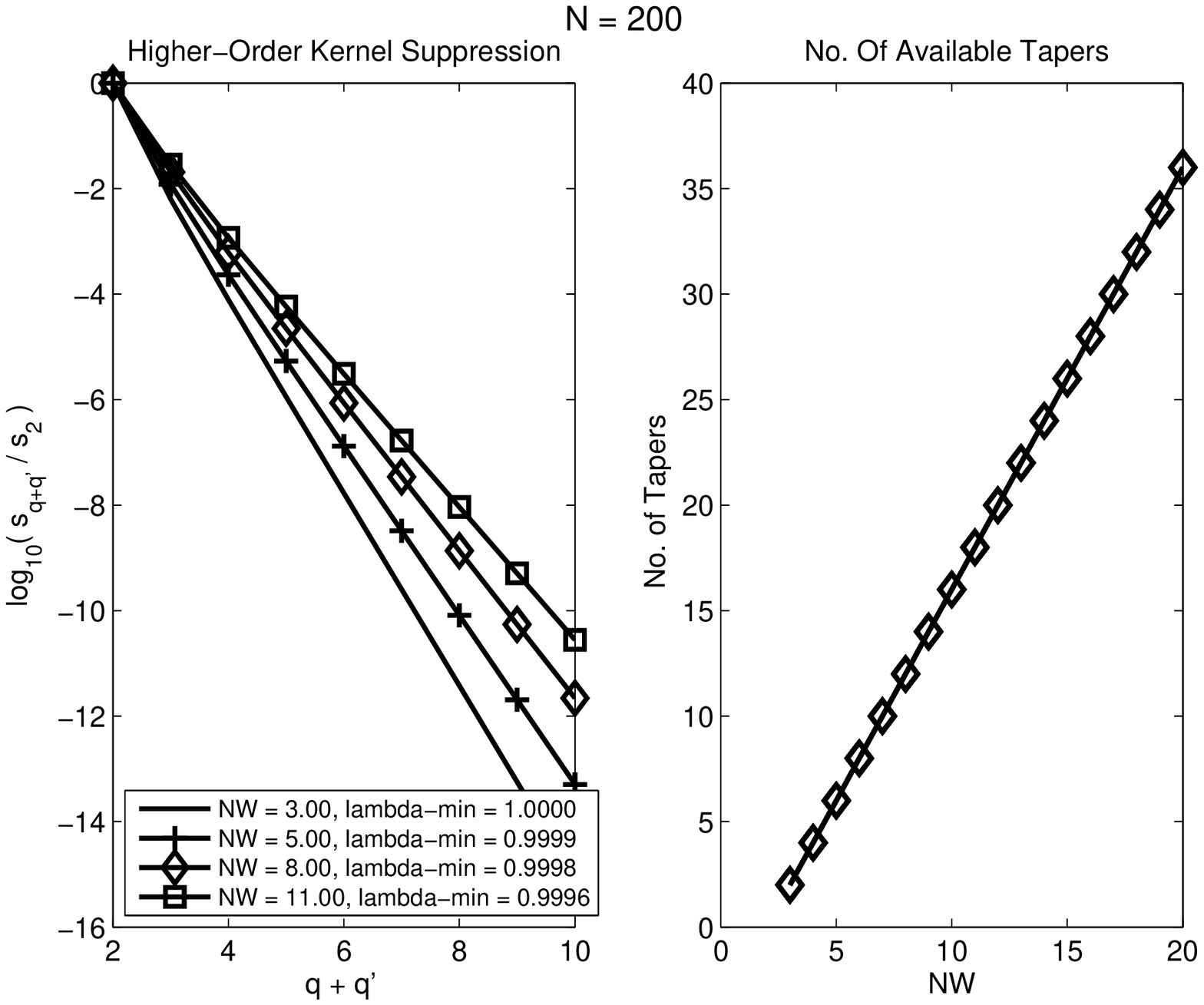}
     \includegraphics[width=3.5in]{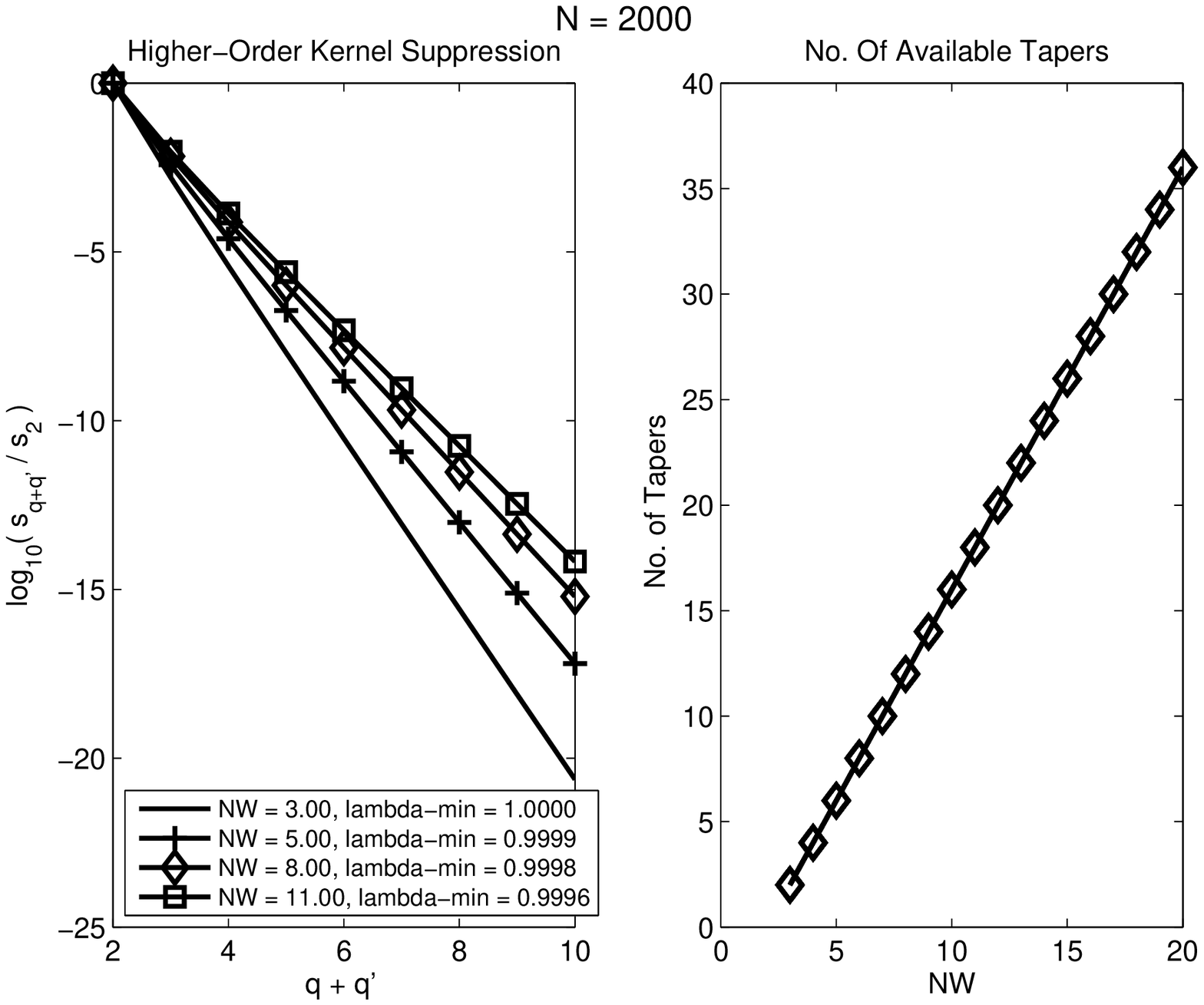}
    \caption{ \label{fig:s}
        The bound, $s_{q+q'}(N,NW)$, specified by (\ref{eqn:dpswf_bnd}) on the contribution, $|J_{m,m',q,q'}|$ to the inner-product, $I_{m,m'}$, 
      normalized by the bound on the 
        linear kernel ($q+q'=2$).  The second order kernel contribution is suppressed by two orders of magnitude due
        to the effects of the DPSWFs relative to the first-order kernel contribution.
    }  
  \end{center}
\end{figure*}
From Fig. (\ref{fig:s}) the $q+q'=3$ contributions to $I_{m,m'}$ are two orders of magnitude less than the $q+q'=2$ term, due
solely to the nature of the DPSS input.  This factor is scaled by the number of input channels, $M$, and the system dependent quantity, $\alpha_*$, specified in 
(\ref{eqn:alpha*}).
Thus the orthogonality of the output channel signals depends upon the number of non-zero DPSS inputs and the nature of the system. 
Using (\ref{eqn:bigJ2}) and focusing upon the dominating $q +q'= 2$ term in (\ref{eqn:ip}),
\begin{eqnarray}
   \lefteqn{J_{m,m',1,1} \approx}&& \nonumber \\
       && \sum_{\substack{ {\bf m}_1 ={\bf 1} \\ {\bf m}'_{1} ={\bf 1}}}^M \int_{-W}^W V_{m_1}(x) V^{*}_{m'_{1 }}(x) \ dx \times \nonumber \\
        && \Gamma_{m,m_1}^{(1)}( 0 ) \Gamma_{m',m'_1}^{(1)*}( 0 ) \ , \nonumber \\
  &=& 
       \sum_{\substack{ {\bf m}_1 ={\bf 1} \\ {\bf m}'_{1} ={\bf 1}}}^M \delta_{m_1,m'_1} \lambda_{m_1} 
        \Gamma_{m,m_1}^{(1)}( 0 ) \Gamma_{m',m'_1}^{(1)*}( 0 )  \ , \nonumber \\
  &=& 
       \sum_{\substack{ {\bf m}_1 ={\bf 1} \\ {\bf m}'_{1} ={\bf 1}}}^M \delta_{m_1,m'_1} \lambda_{m_1} 
        \Gamma_{m,m_1}^{(1)}( 0 ) \Gamma_{m',m'_1}^{(1)*}( 0 ) \ , \nonumber \\
  &=& 
       \sum_{m_1=1}^M \lambda_{m_1} \Gamma_{m,m_1}^{(1)}( 0 ) \Gamma_{m',m_1}^{(1)*}( 0 ) \ , \nonumber \\
        \label{eqn:bigJ_11}
\end{eqnarray}
Thus the inner-product, $I_{m,m'}$, between output channels $m$ and $m'$ tends to be dominated by the first order
Volterra kernels when the average of
all average system impulse responses linking the input channels to the output channels $m$ and $m'$ is also non-zero.
In applications such as sounding, one ones to use the inner-product to determine which system connections connect a given
input channel, $i_*$, to which output channels.  
Here
the inner product between the output channels is used to probe system connections, consistent with taking the inner-product between
 copies of the inputs, with the system outputs and monitoring the system responses.  This task can be accomodated in the
developed theory by specifying system connections between input channels $m_1$ and the output channel $m'$ to be zero unless they match.
Specifically, consider the situation where $\Gamma_{m',m_1}^{(1)*} = \delta_{m', m_1}$, where $\delta_{m',m_1}$ is the Kronecher delta function,
zero when $m' \neq m_1$ and one otherwise.  With this specification $m'$ selects the input channel to study and 
$m$ selects the ouput channel to study and the inner-product provides the relevant information, 
\begin{equation}
   J_{m,m',1,1} \approx
       \lambda_{m'} \Gamma_{m,m'}^{(1)}( 0 ) \ . 
        \label{eqn:bigJ_121}
\end{equation}
To recap, the probing operation whereby DPSSs is applied to the system inputs and the system ouputs are queried by resolving the
output waveforms onto individual DPSS, each matching one of the system inputs, is equivalent to $I_{m,m'}$ in the augmented
system where the $m'$ output channels each pass a single input DPSS without modification and block the other input channels.  In that
scenario, (\ref{eqn:bigJ_121}) is approximately equal to the inner-product $I_{m,m'}$.  The development up to now can be used to
provide the quantitative sense in which this favourable system property is valid.   


\hrule
\comment
{
In (\ref{eqn:bigJ1}) the contribution of the inner integral, $Ji$, to $J_{m,m',q,q'}$ differs if the input channel, $m_q$ equals the
input channel $m'_{q'}$, somewhat independent of the remaining input channels, $m_1,m'_1,\ldots,m_q,m'_{q'}$. To explore this
property of the inner-product of the system output, re-write (\ref{eqn:bigJ1}), as,
\begin{eqnarray}
   \lefteqn{J_{m,m',q,q'} =}&& \nonumber \\
       && \sum_{\substack{m_1,\ldots,m_q =1 \\ m'_1,\ldots, m'_{q'} =1  \\ m_q = m'_{q'} }}^M  J^{(if)}_{m,m',q,q',m_1,m'_1,\ldots,m_q,m'_{q'}}
        + \nonumber \\
       && \sum_{\substack{m_1,\ldots,m_q =1 \\ m'_1,\ldots, m'_{q'} =1  \\ m_q \neq m'_{q'} }}^M  J^{(if)}_{m,m',q,q',m_1,m'_1,\ldots,m_q,m'_{q'}}                                \ , \nonumber \\
        &=& J^{( \ell )}_{m,m',q,q'} + J^{(r)}_{m,m',q,q'} \ ,
        \label{eqn:bigJ2}
\end{eqnarray}
where
\begin{eqnarray}
   \lefteqn{ J^{(if)}_{m,m',q,q',m_1,m'_1,\ldots,m_q,m'_{q'}} = } && \hspace{4cm} \nonumber \\
      && \int_{-\frac{1}{2}}^\frac{1}{2} \ldots \int_{-\frac{1}{2}}^\frac{1}{2} \times \nonumber \\
       && \hspace{.5cm} Ji_{m,m',q,q',m_1,m'_1,\ldots,m_q,m'_{q'}}(  f_1,f'_1,\ldots,f_q,f'_{q'}) \times \nonumber \\
        && \hspace{.5cm} df_1 df'_1 \ldots df_q df'_{q'} \ .
\end{eqnarray}

For $m_q \neq m'_{q'}$, $Ji^{(W)} = 0$ and the contribution of $J^{(r)}_{m,m',q,q'}$, the second term in (\ref{eqn:bigJ2}), is bounded.
For integer $q,q' > 0$, and the half-bandwidth $W$, let
\begin{eqnarray}
  \lefteqn{C^{(r)}_{*,m,m'}(W) =} & \hspace{2cm} \nonumber \\
    &  \sup\limits_{\substack{m_1,m'_1,\ldots,m_q,m'_{q'} \ \in \ [1,M]^q \ \cap \ \mathbb{N}^q \ /  \ \left\{ m_q \ : \ m_q \ = \ m'_{q'}\right\} 
                      \\ f_1,f'_1,\ldots,f_{q},f'_{q'} \ \in \ (-W,W)^q \times (-W,W)^{q'}}}
        \hspace{1cm}
        \nonumber \\
    & \hspace{1cm} C_{m,m', m_1, m'_1, \ldots,m_q,m'_q}(f_1,f'_1,\ldots,f_{q},f'_{q'},W) \ . \nonumber \\
\end{eqnarray}
Since in (\ref{eqn:bnd22}) $q$ and $q'$ are interchangeable set $J^{(r)}_{*,m,m',q,q'}$ to the tighter bound:
\begin{theorem}{Bounded $\big| J^{(r)}_{m,m',q,q'} \big|$\newline} 
  \begin{eqnarray}
    \lefteqn{\left| J^{(r)}_{m,m',q,q'} \right| \leq}&& \hspace{1.5cm} \nonumber \\
      && C^{(r)}_{*,m,m'}(W) \ \left( M^{q+q'} - M^{q+q'-1}\right) \ \times \nonumber \\
      && \hspace{1.5cm} \min\left\{ B_{m'_{q'},1} \sqrt{\lambda_{m'_{q'}}} , \ B_{m_q,1} \sqrt{\lambda_{m_q}} \right\} \ ,
            \nonumber \\
      & &= \hspace{.1cm}  J^{(r)}_{*,m, m', q,q'} \ .
  \end{eqnarray}
\end{theorem}
\begin{proof}
The proof is established by direct calculation.
\begin{eqnarray}
  \lefteqn{\left| J^{(r)}_{m,m',q,q'} \right| \leq } &\hspace{2.00cm} B_{m'_{q'},1} \ \sqrt{\lambda_{m_q}} \ C^{(r)}_{*,m,m'}(W) \ \times \hspace{3.5cm} \nonumber \\
        & \sum\limits_{\substack{m_1,\ldots,m_q =1 \\ m'_1,\ldots, m'_{q'} =1 \\ m_q \neq m'_{q'} }}^M 
                              \int_{-\frac{1}{2}}^\frac{1}{2} \ldots \int_{-\frac{1}{2}}^\frac{1}{2} \ \times \nonumber \\
        & \left| \prod\limits_{j=1}^{q-1} V_{m_j}( f_j ) \prod\limits_{j'=1}^{q'-1} V^*_{m'_{q'}}( f'_{j'} ) \right| \times \nonumber \\
       &   \hspace{1cm}                     df_1 \ldots df_{q-1} \ 
                                            df'_1 \ldots df'_{q'-1} \ , \nonumber \\
       &=   B_{m'_{q'},1} \ \sqrt{\lambda_{m_q}} \ C^{(r)}_{*,m,m'}(W) \ \left( M^{q+q'} - M^{q+q'-1}\right) \times \hspace{3.5cm} \nonumber \\
       &                       \int_{-\frac{1}{2}}^\frac{1}{2} \ldots \int_{-\frac{1}{2}}^\frac{1}{2} 
              \left| \prod\limits_{j=1}^{q-1} V_{m_q}( f_j ) \prod\limits_{j'=1}^{q'-1} V^*_{m'_{q'}}( f'_{j'} ) \right| \ \times \nonumber \\
       &            \hspace{1cm} df_1 \ldots df_{q-1} \ 
                                            df'_1 \ldots df'_{q'-1} \ , \nonumber \\
       &\leq   B_{m'_{q'},1} \ \sqrt{\lambda_{m_q}} \ C^{(r)}_{*,m,m'}(W) \ \left( M^{q+q'} - M^{q+q'-1}\right) \times \hspace{3.5cm} \nonumber \\
       &                       \int_{-\frac{1}{2}}^\frac{1}{2} \ldots \int_{-\frac{1}{2}}^\frac{1}{2} 
              \left| \prod\limits_{j=1}^{q-1} V_{m_q}( f_j ) \prod\limits_{j'=1}^{q'-1} V^*_{m'_{q'}}( f'_{j'} ) \right| \ \times \nonumber \\
       &            \hspace{1cm} df_1 \ldots df_{q-1} \ 
                                            df'_1 \ldots df'_{q'-1} \ , \nonumber \\
       &\leq   B_{m'_{q'},1} \ \sqrt{\lambda_{m_q}} \ C^{(r)}_{*,m,m'}(W) \ \left( M^{q+q'} - M^{q+q'-1}\right) \ , \hspace{3.6cm}
  \label{eqn:bnd22}
\end{eqnarray}
since 
  \begin{eqnarray}
    \lefteqn{ \int_{-\frac{1}{2}}^\frac{1}{2} \left| V_k(f) \right| \ df \leq} && \nonumber \\
      && \sqrt{ \int_{-\frac{1}{2}}^\frac{1}{2} \left| V_k(f') \right|^2 df' \ \int_{-\frac{1}{2}}^\frac{1}{2} df''} \ , \nonumber \\
      &=& 1 \ , \ k = 1, \ 2, \ \ldots \ .
  \end{eqnarray}
\end{proof}
\hrulefill
\newline
The contribution to the inner-product, $I_{m,m'}$, due to terms where $m_q \neq m'_{q'}$ can now be bounded. 
\begin{eqnarray}
  I_{m,m'} &=& I^{(x)}_{m,m'} + I^{(o)}_{m,m'} \ , \nonumber \\
      &=& \sum_{q,q'=1}^Q J^{(\ell)}_{m,m',q,q'}  + \sum_{q,q'=1}^Q J^{(r)}_{m,m',q,q'} \ .
\end{eqnarray}
Now the contribution to $I_{m,m'}$ due to the cross-orders, $q \neq q'$, can be bounded:
\begin{theorem}{Contribution of $I^{(x)}_{m,m'}$ to $I_{m,m'}$\newline}
Let 
\begin{equation}
  J^{(x)}_{*,m,m'} = \sup\limits_{\substack{ q,q' \ \in \ \left\{1, \ 2, \ldots, Q\right\}^2 \\ q \neq q'}} \ J_{*, m,m',q,q'} \ .
\end{equation}
Then,
\begin{equation}
   \left| I^{(x)}_{m,m'} \right| \leq \left( Q^2-Q \right)   \ J_{*,m,m'} \ .
\end{equation}
\end{theorem}
\begin{proof}
 The omitted proof follows from direct calculation.
\end{proof}
but for $q = q'$,
\begin{equation}
  \left| J_{m,m',q,q} - \lambda_q \right| \leq  B_{q,1} \ \sqrt{\lambda_q} \ C_* \ M^{2q} \ suspect.
\end{equation}
Then,
\begin{equation}
   \left| I_{m,m'} \right| \leq
    \left( Q^2 - Q \right) \ \times \ \max\limits_{q \neq  q'} J_{*,q,q'} + + Q \lambda_q \ .
\end{equation}

By developing an approximation for
$Ji$ 
\begin{theorem}{J-bound\newline}
  \begin{eqnarray}
    \lefteqn{ \left| J_{m,m',q,q'} - \tilde{J}_{m,m',q,q'} \right|}
  \end{eqnarray}
\end{theorem}
\begin{proof}\newline
  There exist functions, $\tilde{\Gamma}^{(q)}$, and $g_q$, such that 
  \begin{eqnarray}
   \lefteqn{\Gamma^{(q)}(f_1, \ldots, f_{q-1}, f - \sum_{j=1}^{q-1} f_j ) =} && \nonumber \\ 
            && \tilde{\Gamma}^{(q)}(f_1,\ldots,f_{q-1}) g_q\left(f - \sum_{j=1}^{q-1} f_j \right) \ .
  \end{eqnarray}
  Let 
  \begin{eqnarray}
   \lefteqn{\tilde{J}_{m,m',q,q'} =}&& \nonumber \\
       && \int_{-\frac{1}{2}}^\frac{1}{2} \sum_{\substack{m_1,\ldots,m_q =1 \\ m'_1,\ldots, m'_{q'} =1}}^M 
                              \int_{-\frac{1}{2}}^\frac{1}{2} \ldots \int_{-\frac{1}{2}}^\frac{1}{2} \times \nonumber \\
       &&                       \Gamma^{(q)}_{m,m_1, \ldots,m_q}\left(f_1,\ldots,f_{q-1}, f - \sum_{j=1}^{q-1} f_j \right) \times \nonumber \\
       &&                       \Gamma^{(q')*}_{m',m'_1, \ldots,m'_{q'}}\left(f'_1,\ldots,f'_{q-1}, f - \sum_{j'=1}^{q'-1} f'_{j'} \right) \times \nonumber \\
       &&                       U_q\left(  f - \sum_{j=1}^{q-1} f_j \right) \prod_{j=1}^{q-1} U( f_j ) \times \nonumber \\
       &&                       U^*_{q'}\left(  f' - \sum_{j'=1}^{q'-1} f'_{j'} \right) \prod_{j'=1}^{q'-1} U^*( f'_{j'} ) \times \nonumber \\
       &&                       df_1 \ldots df_{q-1} \ \ df \ .
  \end{eqnarray}

\end{proof}
} 
\comment
{
\begin{definition}{$W$ -- Limited Transfer}
\end{definition}

\begin{definition}{$\epsilon$ -- DPSS Conformal System\newline}
  Let $S$ be the $Q^{\rm th}$ order system specified by (\ref{eqn:volterra_td}) 
with $M'$ outputs, and parameterized by $\gamma^{(q)}$, $q = 1, \ldots, Q$.
  Let $K$ of the $M$ channel inputs at time-index $t$ be $v^{(j)}_t$, $j = 1, \ldots, K$. For the $M-K$ remaining channels, the input
  is set to zero for all $t$.
  Further, let $\epsilon > 0$, such that $\left| I_{m,m'} \right| < \epsilon \left| I_{m,m} \right|$, for all $m,m' = 1, \ldots, M'$, $m \neq m'$.
  Then $S$ is an $\epsilon$ -- DPSS conformal system.
\end{definition}
\begin{definition}{$q^{\rm th}$ order, $W, \eta$ Constant Volterra Kernel}
  Let
  \begin{equation} 
  \end{equation}
\end{definition}
\begin{theorem}{Single $\epsilon$ -- DPSS Conformal System Equivalence\newline}
  Sets of $\epsilon$ -- DPSS conformal systems are equivalent to a single $\epsilon$ -- DPSS conformal system. 
\end{theorem}
\begin{definition}{$\epsilon$ -- DPSS orthogonal}
\end{definition}
\begin{theorem}{An $\epsilon$ -- DPSS Conformal System is $\epsilon$ -- DPSS Orthogonal\newline}
  Output is orthogonal to DPSSs to level $\epsilon$.
\end{theorem}
} 

} 

%% file: ortho2EH.tex
\label{sect:ortho}
The system output $Y_m(f)$ is approximated by the contributions from 
$T_{m,1}(f)$ and $T_{m,2}(f)$,
\begin{eqnarray}
  \lefteqn{ \left| Y_m(f) - T_{m,1}(f) - T_{m,2}(f) \right| \leq   } && \hspace{4.25cm} y_m^{(0)} + \sum_{j=3}^Q \left| T_{m,j}(f) \right| \ . \nonumber \\
  \label{eqn:azaz}
\end{eqnarray}
As in \ref{sec:content2}, consider the situation where 
the DC response $y_m^{(0)}$ is set to zero and the bound \eqref{eqn:JJ_approx} is exact (i.e. $\rm \delta' = 0$).
As a consequence of the higher order Volterra kernel suppression demonstrated 
in \ref{sec:content2}, the system output $Y_m(f)$ can be further bounded:
\begin{eqnarray}
  \lefteqn{ \left| Y_m(f) - T_{m,1}(f) - T_{m,2}(f) \right| \leq   } && \hspace{4.0cm} \sum_{j=3}^Q \left| T_{m,j}(f) \right| \ , \nonumber \\
  &\leq& \sum_{j=3}^Q C_{m,M,j}(f,\lambda_{min},W,V_{M,*},\Gamma^{(j)}_{m,*},\Gamma_{m,*}^{(j )'}, \Gamma^{(j)}_{m,**}({\bf 0},f)) \ . \nonumber \\
\end{eqnarray}

In general the Volterra expansion of $\mathcal{H}$ may have a finite or infinite number of terms. Any Volterra
expansion, however, may be represented as an infinite series by adding kernels of value zero if needed.
This representation is referred to as the infinite Volterra expansion.
\begin{definition}{Exponential DPSS System\newline}
  \label{def:exp_dpss_sys}
  Let $\mathcal{H}$ be a nonlinear, time-invariant MIMO system equal to the Volterra expansion \eqref{eqn:volterra_td0} in the limit as
  $Q \rightarrow \infty$. 
  Then the generalized frequency response functions 
of $\mathcal{H}$ define functions $\Gamma_{m,*}^{(j)}$ \eqref{eqn:key_gamma1}, $\Gamma_{m,**}^{(j)}$ \eqref{eqn:gamma_**}, 
and $\Gamma_{m,*}^{(j)'}$  \eqref{eqn:in-band-sys-rem}. $\mathcal{H}$ is an {\bf exponential DPSS system} if there exist 
constants $\alpha$, $\beta$ and $\gamma$ such that 
  \begin{eqnarray}
    \Gamma_{m,*}^{(j)} &\leq& \alpha^j / j! \ , \nonumber \\
    \Gamma_{m,**}^{(j)}({\bf 0},f) &\leq& \beta^j / j! \ , \nonumber \\
    \Gamma_{m,*}^{(j )'}       &\leq & \gamma^j / j! \ , \nonumber \\
    \label{eqn:alpha-exp-constraints}
  \end{eqnarray}
  for all $j \in \mathbb{Z}^+$ and all $m \in \{ 1, \ \ldots, \ M' \}$.
\end{definition}
\begin{definition}{$\epsilon$-Quadratic System\newline}
Fix an $\epsilon>0$ and a Volterra system $\mathcal{H}$. Let $T_{m,1}(f)$ and $T_{m,2}(f)$ be the first two Volterra 
kernel responses contributing to the $m^{th}$ output channel of $\mathcal{H}$, as defined in Equation \eqref{eqn:volterra_fd}. 
Then $\mathcal{H}$ is an {\bf $\bm \epsilon$-Quadratic System} if for every output channel $m\in \{1, ..., M'\}$ the channel output $Y_m(f)$ satisfies
  \begin{equation}
     \left| Y_m(f) - T_{m,1}(f) - T_{m,2}(f) \right| \leq \epsilon. \nonumber
  \end{equation}
\end{definition}
\begin{theorem}{Exponential DPSS System, $\epsilon$-Quadratic Equivalence\newline}
  Let $\mathcal{H}$ be an exponential DPSS System. Then there exists an $\epsilon>0$ such that $\mathcal{H}$ is an $\epsilon$-quadratic system.
  \label{lemma:exp_dpss_eps_quad}
\end{theorem}
\begin{proof}
The proof follows from direct calcuation.  Consider,
\begin{eqnarray}
  \lefteqn{ \left| Y_m(f) - T_{m,1}(f) - T_{m,2} \right| \leq   } && \nonumber \\
  && 
  V_{M,*} 
    \sum_{j=3}^\infty \Gamma_{m,*}^{(j)} M^j \left( 1 - \lambda_{min} \right)^{j/2}  \ + \nonumber \\
    && W^{-2} \sum_{j=3}^\infty \Gamma_{m,*}^{(j )'} \ 2^\frac{j-1}{2} \ M^j \ W^\frac{3j}{2}    \ + \nonumber \\
  &&          \sum_{j=3}^\infty \Gamma_{m,**}^{(j)}({\bf 0},f) \ M^j \ \left( 2 W \right)^{\frac{j-2}{2}} \ , \nonumber \\
  &\leq&
  V_{M,*} 
    \sum_{j=3}^\infty 
        \alpha^j / j! M^j \left( 1 - \lambda_{min} \right)^{j/2}  \ + \nonumber \\
    && 2^{-\frac{1}{2}} W^{-2} \sum_{j=3}^\infty \gamma^j/j! 2^\frac{j}{2} \ M^j \ W^\frac{3j}{2}    \ + \nonumber \\
  &&    (2 W)^{-1}  \sum_{j=3}^\infty \beta^j / j! M^j \left( 2 W \right)^{\frac{j}{2}} \ . \nonumber \\
\end{eqnarray}
Continuing,
\begin{eqnarray}
  \lefteqn{ \left| Y_m(f) - T_{m,1}(f) - T_{m,2}(f) \right| \leq } && \nonumber \\
  && V_{M,*} \bigg[ e^{ \alpha M \sqrt{ 1 - \lambda_{min}}} - 
              \left( 1 + \alpha M \sqrt{ 1 - \lambda_{min}} + \right.\nonumber \\
        && \left. \alpha^2 M^2 (1-\lambda_{min}) \right) \bigg] + \nonumber \\
  && W^{-2}2^{-\frac{1}{2}} \bigg[ e^{ \sqrt{2} \gamma M W^\frac{3}{2}} \ - \nonumber \\
  && \left( 1 + \sqrt{2} \gamma M W^\frac{3}{2} + 2 \gamma^2 M^2 W^3 \right) \bigg] \ + \nonumber \\
  && (2W)^{-1} \bigg[ e^{ \sqrt{2W} \beta M } \ - \nonumber \\
  && \left( 1 + \sqrt{2W} \beta M + 2W \beta^2 M^2 \right) \bigg] \ , \nonumber \\
  &\equiv& \epsilon \ .
\end{eqnarray}
\end{proof}
\begin{remark}{Linear System Response\newline}
  An $\epsilon$-Quadratic System driven by DPSSs approximates a linear system for $W \ll 1$.  The
  system response in the time-domain due to the first-order Volterra kernel is,
  \begin{eqnarray}
      \lefteqn{\left| \mathcal{F}^{-1}\big\{ T_{m,1}(f) \big\}_t \right| =} && \nonumber \\
        && \bigg| \sum_{m_1=1}^M 
              \int_{-\frac{1}{2}}^\frac{1}{2} 
                \Gamma^{(1)}_{m,m_1}( f ) V_{m_1}(f ) e^{i 2 \pi f t} \ df \bigg| \ , \nonumber \\
        &\approx& \bigg| \sum_{m_1=1}^M \Gamma^{(1)}_{m,m_1  }( 0 ) \int_{-W}^W  V_{m_1}(  f  ) e^{i 2 \pi f t} \ df \bigg| \ , \nonumber \\
        &\leq&  \sum_{m_1=1}^M \bigg| \Gamma^{(1)}_{m,m_1  }( 0 )\bigg| \ \int_{-W}^W  \ \bigg| V_{m_1}(  f  ) \bigg|  \ df \ , \nonumber \\
        &=& O\left( \sqrt{W} \right) \ ,
  \end{eqnarray}
  since
  \begin{eqnarray}
    \int_{-W}^W \left| V_{m_1}(f) \right|^2 \ df  &\approx& 1 \ , \nonumber \\
    \left|V_{m_1}(f)  \right|^2 &\approx& \frac{1}{2W} \ , \nonumber \\
    \left|V_{m_1}(f) \right|    &\approx& \left( 2W \right)^{-\frac{1}{2}} \ , \nonumber \\
  \end{eqnarray}
  and
  \begin{eqnarray}
\int_{-W}^W \bigg| V_{m_1}( f) \bigg| \ df &\approx&  2W \big( 2W\big)^{-\frac{1}{2}} \ , \nonumber \\
                                              &=& \sqrt{2W} \ .
  \end{eqnarray}
  Similarly,
  \begin{eqnarray}
      \lefteqn{\left| \mathcal{F}^{-1}\big\{ T_{m,2}(f) \big\}_t \right| =} && \nonumber \\
        && \bigg| \sum_{m_1,m_2=1}^M 
              \int_{-\frac{1}{2}}^\frac{1}{2} \int_{-\frac{1}{2}}^\frac{1}{2} 
                \Gamma^{(2)}_{m,m_1,m_2}( f_1, f-f_1) \times \nonumber \\
            && \hspace{.2cm} V_{m_1}(f-f_1) V_{m_2}(f_1) e^{i 2 \pi f t} \ df_1 df \bigg| \ , \nonumber \\
        &\approx& \bigg| \sum_{m_1,m_2=1}^M \Gamma^{(2)}_{m,m_1,m_2}( 0, 0 ) \times \nonumber \\
          && \int_{-W}^W \int_{-W}^W  V_{m_1}(f-f_1) V_{m_2}(f_1)  e^{i 2 \pi f t} \ df_1 df \bigg| \ , \nonumber \\
        &\leq& \sum_{m_1,m_2=1}^M \bigg| \Gamma^{(2)}_{m,m_1,m_2}( 0, 0 ) \bigg| \times \nonumber \\
          && \int_{-W}^W \int_{-W}^W  \ \bigg| V_{m_1}(f-f_1) V_{m_2}(f_1)  \bigg| \ df_1 df \ , \nonumber \\
         &=& O\left( W \right) \ .
  \end{eqnarray}   
 \qed
\end{remark}

\comment
{
An Exponential DPSS 
\begin{definition}{$\epsilon$-Time Domain Linear System\newline}
Fix an $\epsilon>0$ and a Volterra system $\mathcal{H}$. Let $\gamma_{m,1}(t)$ 
be the first Volterra 
kernel response  contributing to the $m^{th}$ output channel of $\mathcal{H}$, as defined in Equation \eqref{eqn:volterra_td0}. 
Then $\mathcal{H}$ is an {\bf $\bm \epsilon$-Time Domain Linear System} if for every output channel $m\in \{1, ..., M'\}$ the channel output $y_{m,t}$ satisfies
  \begin{equation}
     \left| y_{m,t}  - \sum_{m'=1}^M \sum_{t'=1}^\infty \gamma^{(1)}_{m,m',t'} u_{m',t-t'} \right| \leq \epsilon. \nonumber
  \end{equation}
\end{definition}
\begin{theorem}{$\epsilon$-Quadratic System, $\epsilon'$ Time-Domain Linear System Equivalence} 
  Let $\mathcal{H}$ be an $\epsilon$-Quadratic System. Then there exists an $\epsilon>0$ such that $\mathcal{H}$ is an $\epsilon$-quadratic system.
  \label{lemma:quad_linear_equiv}
\end{theorem}
\begin{proof}
\end{proof}
} 

{\color{black}As motivated in the introduction, }
in parameter identification applications aimed at determining the existence of connections between MIMO system inputs and outputs, 
the focus is upon the inner-product, $I_{m,m'}$, 
between the $m$ channel ouput, $Y_m(f)$, of the $\epsilon$--Exponential DPSS system and the $m'$-order test DPSS.
Let,
\begin{eqnarray}
  I_{m,m'} = \int_{-\frac{1}{2}}^\frac{1}{2} Y_m(f) V_{m'}^*(f) \ df \ .
\end{eqnarray}
Once again, split the integral into in-band and out-of-band components:
\begin{eqnarray}
  \lefteqn{I_{m,m'} =}&& \hspace{1.0cm} \int_{-\frac{1}{2}}^\frac{1}{2} Y_{m  }(f) V_{m'}^*(f) \ df \ , \nonumber \\
    &=& 
        \int_{-W}^W Y_m(f) V_{m'}^*(f) \ df \ + 
    \tildeint_{-\frac{1}{2}}^\frac{1}{2} Y_m(f) V_{m'}^*(f) \ df \ , \nonumber \\ 
    &=& I^{(i)}_{m,m'} + I^{(o)}_{m,m'} \ .
\end{eqnarray}
 Because $T_{m,1}(f)$ + $T_{m,2}(f)$ approximates the output of an $\epsilon$-exponential DPSS system, we can now state the following Lemma.
\begin{lemma}
The inner-product, $I_{m,m'}$, can be approximated by the inner product of $V_{m'}$ with the first and second order Volterra responses of $\mathcal{H}$.
\end{lemma}
\begin{proof}
Let $X_{m,m'}$ be the inner-product of the  first two Volterra system responses with the input $V_{m'}(f)$.  
The out-of-band component 
$X^{(o)}_{m,m'}$ of the inner product, $X_{m,m'}$ is approximated by
\begin{equation}
  \left| X^{(o)}_{m,m'} \right| = \left| X^{(o,1)}_{m,m'} \right| + \left| X^{(o,2)}_{m,m'} \right| \ , 
\end{equation} 
where
\begin{eqnarray}
  \lefteqn{ \left| X^{(o,1)}_{m,m'} \right|  =  } && \nonumber \\
    &&\left| \sum_{m''=1}^M \tildeint_{-\frac{1}{2}}^\frac{1}{2} \Gamma^{(1)}_{m,m''}( f ) V_{m''}(f) V_{m'}^*(f) \ df \right| \ , \nonumber \\
    &\leq& \Gamma^{(1)}_{*,m} \sum_{m''=1}^M \sqrt{ 1- \lambda_{m''} } \sqrt{ 1 - \lambda_{m'}} \ , \nonumber \\
    &\leq& M \Gamma^{(1)}_{*,m} \left( 1 - \lambda_{min} \right)   \ .
\end{eqnarray}
The second-order term is bounded in a similar fashion,
\begin{eqnarray}
  \lefteqn{ \left| X^{(o,2)}_{m,m'} \right|  =  } && \nonumber \\
    &&\bigg| \sum_{m_1, m_2=1}^M \tildeint_{-\frac{1}{2}}^\frac{1}{2} \bigg[ \int_{-W}^W \Gamma^{(2)}_{m,m_1, m_2}( f_1, f-f_1 ) \times \nonumber \\
      && \hspace{2cm} V_{m_2}(f-f_1) V_{m_1}( f_1 ) df_1 \bigg] V_{m'}^*(f) \ df \bigg| \ , \nonumber \\
    &\leq& M^2 \Gamma^{(2)}_{*,m} \sqrt{ 1 - \lambda_{m'}} \ . 
\end{eqnarray}
Then, the in-band contribution $X^{(i)}_{m,m'}$ to $X_{m,m'}$ is approximated as:
\begin{equation}
  \left| X^{(i)}_{m,m'} \right| = \left| X^{(i,1)}_{m,m'} \right| + \left| X^{(i,2)}_{m,m'} \right| \ .
\end{equation}
By Taylor expanding the linear transfer function $\Gamma^{(1)}_{m,m''}(f)$, the in-band contribution $X^{(i,1)}_{m,m'}$ 
can be expressed as
\begin{eqnarray}
  \lefteqn{X^{(i,1)}_{m,m'} = } && \nonumber \\
        && \int_{-W}^W \sum_{m''=1}^M \Gamma^{(1)}_{m,m''}(f) V_{m''}(f) V_{m'}^*(f) \ df \ ,  \nonumber \\
        &=&   
           \bigg( \sum_{m''=1}^M \Gamma^{(1)}_{m,m''}(0) \lambda_{m'} \delta_{m',m''} \bigg) + R \ , \nonumber \\
        &=&   
           \Gamma^{(1)}_{m,m'}(0) \lambda_{m'}  + R     \ . \nonumber \\
\end{eqnarray}
%
%
The remainder $R$, due to Taylor's theorem is,
\begin{equation}
  R    = \sum_{m''=1}^M        
          \frac{d\Gamma^{(1)}_{m,m''}(\zeta) }{d\zeta} \int_{-W}^W  f V_{m''}(f) V_{m'}^*(f) \ df \ , 
  \label{eqn:r}
\end{equation}
for some $\zeta \in (0,f)$.  
Next, the second order contribution can be upper bounded as follows:
\begin{eqnarray}
  \lefteqn{\left| X^{(i,2)}_{m,m'} \right| = } && \nonumber \\
        && \bigg| \int_{-W}^W \sum_{m_1,m_2=1}^M \Gamma^{(2)}_{m,m_1,m_2}(f_1, f - f_1 ) \times \nonumber \\
        && \hspace{2cm} V_{m_2}(f-f_1) V_{m_1}(f_1) V_{m'}^*(f) \ df_1 df \bigg| \ ,  \nonumber \\
        &\leq&   \sqrt{2W} M^2 \Gamma^{(2)}_{m,*} \ .
  \label{eqn:aaax}
\end{eqnarray}
Then
\begin{eqnarray}
  \lefteqn{ \left| X^{(i)}_{m,m'} - \Gamma^{(1)}_{m,m'}(0) \lambda_{m'} \right| -  
                      \sqrt{2W} M^2 \Gamma^{(2)}_{m,*} \leq \ } && \hspace{5.5cm} \ \left| R \right| 
           \ , \nonumber \\
  &\leq& \sum_{m''=1}^M        
          \left| \frac{d\Gamma^{(1)}_{m,m''}(\zeta)}{d\zeta}\right| \int_{-W}^W \left| f V_{m''}(f) V_{m'}^*(f) \right| \ df \ , \nonumber \\
  &\leq& W \sum_{m''=1}^M        
          \left| \frac{d\Gamma^{(1)}_{m,m''}(\zeta)}{d\zeta}\right| \int_{-W}^W \left| V_{m''}(f) V_{m'}^*(f) \right| \ df \ , \nonumber \\
  &\leq& W \sum_{m''=1}^M        
          \left| \Gamma^{(1)'}_{m,m''}(\zeta)\right| \sqrt{ \lambda_{m''} \lambda_{m'}} \ , \nonumber \\
  &\leq& W \sum_{m''=1}^M \left| \frac{d\Gamma^{(1) }_{m,m''}(\zeta)}{d\zeta} \right| \ .
          \label{eqn:above197}
\end{eqnarray}
Moving $\sqrt{2W} M^2 \Gamma^{(2)}_{m,*}$ to the right-hand side of \eqref{eqn:above197} yields
\begin{eqnarray}
  \label{eqn:ip_bnd1J}
  \lefteqn{\left| X_{m,m'} - \Gamma^{(1)}_{m,m'}(0) \lambda_{m'} \right|  \leq} && \nonumber \\
    && |R| +  \sqrt{2W} M^2 \Gamma^{(2)}_{m,*} + \nonumber \\
    && \hspace{0cm}  M \Gamma^{(1)}_{m,*} \left(1-\lambda_{min}\right) 
                + M^2 \Gamma^{(2)}_{m,*} \sqrt{ 1 - \lambda_{m'}} \ , \nonumber \\
    &\leq& W\sum_{m''=1}^M \left| \Gamma^{(1)'}_{m,m''}( \zeta ) \right| + \nonumber \\
    && \left( \sqrt{2W} +  \sqrt{1-\lambda_{m'}}\right) M^2 \Gamma^{(2)}_{m,*} \ +   \nonumber \\
    && M \Gamma^{(1)}_{m,*} \left( 1 - \lambda_{min} \right) \ .
\end{eqnarray}
Finally,
\begin{eqnarray}
  \label{eqn:ii_jj__bnd}
  \lefteqn{\left| I_{m,m'} - X_{m,m'} \right| \leq} && \nonumber \\
    && \int_{-\frac{1}{2}}^\frac{1}{2} \left| Y_m(f') - T_{m,1}(f') - T_{m,2}(f') \right| \ \times \nonumber \\
    && \hspace{2cm} \left| V_{m'}^*(f') \right| \ df' \ , \nonumber \\
    &\leq& \epsilon \sqrt{ \lambda_{m'}} \ ,
\end{eqnarray}
so that
\begin{eqnarray}
  \lefteqn{\left| I_{m,m'} - \Gamma^{(1)}_{m,m'}(0) \lambda_{m'} \right| \leq} && \nonumber \\
  && \epsilon \sqrt{\lambda_{m'}} + W\sum_{m''=1}^M \left| \frac{d\Gamma^{(1) }_{m,m''}( \zeta )}{d\zeta} \right| + \nonumber \\
 && M \Gamma^{(1)}_{*,m,m''} \left( 1 - \lambda_{min} \right) \ , \nonumber \\
  &\leq& \epsilon \sqrt{\lambda_{min}} + W M        \Gamma^{(1)'}_{m,**}          + M \Gamma^{(1)}_{*,m,m''} \left( 1 - \lambda_{min} \right) + \nonumber \\
  &&  \left( \sqrt{2W} +  \sqrt{1-\lambda_{m'}}\right) M^2 \Gamma^{(2)}_{m,*} \ .
  \label{eqn:ip_bnd1c}
\end{eqnarray}
Here
\begin{equation}
   \Gamma^{(1)'}_{m,**} = \sup\limits_{\substack{ \zeta \in (0,f) \\ m'' \in {1,2,\ldots,M}}} 
      \left| \frac{d\Gamma^{(1) }_{m,m''}( \zeta )}{d\zeta} \right| \ .
  \label{eqn:gamma_**_prime}
\end{equation}
\end{proof}
Eqn. (\ref{eqn:ip_bnd1c}) establishes an upper bound on the difference between the 
inner-product of the system output on channel $m$ with the DPSS input on channel $m'$ for the case $\delta' = 0$ (see \eqref{eqn:T_bnd1}).  
This situation is only approximately valid (see \eqref{eqn:JJ_approx}, and Eqns. \eqref{eqn:aaax}-\eqref{eqn:ip_bnd1c} are more generally ($\delta' \neq 0$) approximate.
The bound \eqref{eqn:ip_bnd1c} provides conditions under which the DPSS remain orthogonal, even after passing through $\mathcal{H}$.  
As described in the Introduction, and as will be demonstrated in \ref{sect:sys_id} below, this result provides the basis for a detector capable of separating the relative influences of inputs on the system outputs.

\comment
{
While it is possible to use (\ref{eqn:ip_bnd1}) to approximate the inner product using the previously 
established bounds, a tighter bound is available that makes further use of the DPSWF orthogonality to 
that,
\begin{eqnarray}
  \lefteqn{\left| I_{m,m'} - \Gamma^{(1)}_{m,m'}(0) \lambda_{m'} \right|  \leq} && \nonumber \\
    && . \nonumber \\
\end{eqnarray}
} 
\comment
{
\begin{definition}{$\eta$--Linear, Memoryless DPSS System\newline}
  \label{def:exp_dpss_sys}
  Let $\eta > 0$.  Let $m$ be an output channel of a system responding to DPSS system inputs, 
  $v_t^{(j)}$, $j = 1, \ 2, \ldots, M$, with Volterra response $w_{m,t}$ specified in (\ref{eqn:volterra_id}),
  such that 
  \begin{equation}
    \lim\limits_{Q\rightarrow \infty} \left| Y_m^{(Q)}(f) - \Gamma^{(1)}_{m,m'}(0) \sum_{m'=1}^M V_{m'}(f) \right| \leq \eta  \ .
  \end{equation}
  Then $w_{m,t}$ is the response of a $\eta$-linear, memoryless DPSS system.
\end{definition}
\begin{theorem}{$\epsilon$-Exponential DPSS \& $\eta$--Linear, Memoryless System Equivalence\newline}
  Let $Y_m(f)$ be the response on channel $m$ of an exponential DPSS system.  Then there exists an $\eta$, 
  \begin{eqnarray}
    \lefteqn{\epsilon =} && \nonumber \\
  && V_{M,*} \bigg[ e^{ \alpha M \sqrt{ 1 - \lambda_{min}}} - \left( 1 + \alpha M \sqrt{ 1 - \lambda_{min}} \right) \bigg] + \nonumber \\
  && W^{-2}2^{-1} \bigg[ e^{ 2 \gamma M W^2} - \left( 1 + 2 \gamma M W^2 \right) \bigg] + \nonumber \\
  && V_{M,**} \bigg[ \left( v_0^{(*)}\right)^{-1} \bigg( e^{\beta M v_0^{(*)}} - \left( 1 + \beta M v_0^{(*)} \right) \bigg) + \nonumber \\
  &&        \left( 1 - \lambda_{min}\right)^{-1/2} \bigg( e^{\beta M \sqrt{ 1 - \lambda_{min} }} - \left( 1 + \beta M \sqrt{1 - \lambda_{min}}\right) \bigg)
    \bigg] \ , \nonumber \\
  \end{eqnarray}
  such that $W_m(f)$ is the output of an $\eta$-linear, memoryless system.
\end{theorem}
\begin{proof}
  Once gain the proof follows from direct calculation.  Consider,
  \begin{eqnarray}
   \lefteqn{\left| T_{m,1}(f) -  \Gamma^{(1)}_{m,m'}(0) \sum_{m'=1}^M V_{m'}(f) \right| =} && \nonumber \\
    && \left| \sum_{m'=1}^M \left( \Gamma_{m,m'}^{(1)}(f) - \Gamma^{(1)}_{m,m'}(0) \right) V_{m'}(f) \right| \ , \nonumber \\
  \end{eqnarray}
\end{proof}
} 

%% file: nb_sys_id.tex
\label{sect:sys_id}
The theory in Sections \ref{sec:content2}-\ref{sect:ortho} can be applied for the purposes of 
identifying the existence of linear narrowband connections between inputs and outputs. 
\comment
{
\begin{figure}
	\centering
		\includegraphics[width=0.45\textwidth]{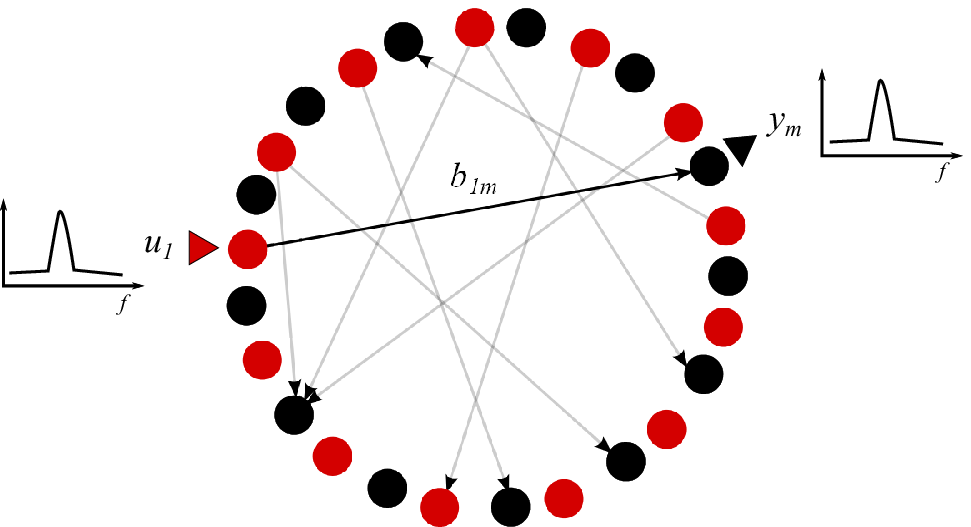}
	\caption{{\bf MIMO interconnection identification.}  
    The narrowband excitation $u_{i,t}$ evokes the system response $y_{j,t}$.  
    The narrowband coupling parameter, $b_{i,j}$, relating $y_{j,t}$ to $u_{i,t}$ is compared against a null
    reponse via the inner-product detector \eqref{eqn:ip_bnd1c}.  It is important to note that
  in this context, full Volterra kernel identification is not the goal.
} 
	\label{fig:mimo_ex}
\end{figure}
} 
Specifically, {\color{black} for an $\epsilon$-quadratic MIMO}, the DPSS can be applied to the input to determine the existence of a {\color{black} a linear narrowband connection }
from input channel $m'$ to output channel $m$ through the use of the inner-product \eqref{eqn:ip_bnd1c}.
Because the inner product of the $\epsilon$-quadratic MIMO system with the DPSS supplied to input $m'$ is approximately
equal to the average, linear impulse response connecting $m'$ with $m$, the inner product determines the existence of in-band linear responses
that do not average to zero (but see Remark 2).
Eqn. (\ref{eqn:ip_bnd1c}) determines the accuracy of this approximation in terms of the quantities listed in
Table \ref{tbl:sys_funcs} and the number of input channels $M$.  

\begin{remark}{Generalization to Non-zero Frequency}\newline
\label{rem_dc_offset}
The results presented in \ref{sec:content1} and \ref{sec:content2} generalize to 
frequency shifted, or modulated DPSS input.
In \ref{sec:content1} and \ref{sec:content2},
instead of Taylor expanding about zero frequency, the Taylor expansions are carried out about
a carrier frequency $f_0$.  That is, the DPSS are multipled by the phase factor  $e^{i 2 \pi f_0 t}$, and the truncated frequency interval changes
from $(-W,W)$ to $(f_0-W, f_0+W)$ (ignoring the contribution from the negative frequencies).  
Thus, the results of this work allow for linear narrowband response at frequency intervals differing from baseband.
\qed
\end{remark}

While it is outside the scope of this paper, note that these conditions are expected to approximate those required for the inner-product detector to achieve the performance of a 
matched-filter when the system response is added to measurement noise prior to observation \cite{scharf1994matched}. This 
scenario is studied in simulation in \ref{sect:sim}, where the DPSS based inner-product detector is found effective.
\begin{remark}{}
\label{rem:DC}
In the theory discussed so far, the zeroth-order Volterra system response, is assumed equal to zero.  Such a response contributes to the system DC offset.  
In situations where this is not guaranteed, restriction to the odd-ordered DPSS ensures that the inner-product will not respond to this offset. \qed
\end{remark}

%% file: nb_sys_id_sim.tex
\label{sect:sim}
To explore the utility of the proposed method  
two $3^{\rm rd}$ order Volterra SISO systems are simulated. The first system is a null system containing a white, or constant response as a function of frequency, and
the second, or 
alternate, system possesses an elevated response about $2 \ {\rm Hz}$.  In this simulation four methods of detecting a narrowband system response are compared.
Specifically,
the detector responses resulting from null system excitation are compared with detector responses resulting from alternate system excitation.

Each simulation involves $240$ measurements ($n = 240$) and the time-index is an element  in the 
set of time-indices: $t \in  \left\{ 0, \ \ldots, \ n-1\right\}$. 
This set is used in all simulations, each of which involves only a single trial of simulated data. 
The sample period $\Delta$ is equal to $1/30 \ {\rm s}$,  Nyquist frequency $f_N$ is equal to $15 \ {\rm Hz}$, 
the duration of observation is $8 \ {\rm s}$ and the Rayleigh resolution  $f_R$, is equal to $3/8 \ {\rm Hz}$. 
 
Let $\gamma^{(j,n)}_{t}$ be the $j^{\rm th}$ order kernel for the null system
evaluated at time-index $t$, and let $\gamma^{(j,a)}_{t}$ be the $j^{\rm th}$ order kernel for the alternate system.  
The  null system is  specified by
inverse discrete Fourier transforming $\Gamma^{(1,n)}$, set to a constant function of
frequency equal to $3/4$. 
The alternate system is specified by inverse Fourier transforming $\Gamma^{(1,a)}$ specified as:
\begin{equation}
  \Gamma^{(1,a)}(f) = \left\{ \Gamma^{(1,n)}(f) + \begin{array}{ccc} 
      10^{-3} / f^{-3} & , &  3 f_R \leq |f| \leq f_N \\
      10^{-3} / f_R^{-3} & , & |f|  < 3 f_R 
        \end{array} \right.  \ .
\end{equation}
To illustrate the merits of the proposed method
narrowband response detection is compared against the least-squares
kernel identification procedure presented in  
\cite{laguerre_marmarelis1993identification,marmarelis1997modeling}.  
This procedure makes use of a Laguerre polynomial basis.  
For comparison, and to facilitate the specification of the higher-order response functions, 
let $c_k^{(n,1)}$ ($c_k^{(a,1)}$) be the 
$k^{\rm th}$ Laguerre expansion coefficient for the null (alternate) system, multiplying
\begin{eqnarray}
 g_{k,t}  = P_{100(k-1)+1,t} 
\end{eqnarray}
in the representation:
\begin{equation}
  \gamma^{(1,n)}_t = \sum_{k=1}^{50} c^{(1,n)}_k g_{k,t} \ .
  \label{eqn:gamma_lag}
\end{equation}
Here $P_{k,t}$ is the $k^{\rm th}$ order discretized Laguerre polynomial, 
  $P_{k,t}  =  L_k( \Delta t)$, where $L_k$ is the $k^{\rm th}$ Laguerre polynomial:
\begin{equation}
  L_k(x) = \sum_{j=0}^k \frac{(-1)^j }{j!} \binom{k}{j} x^j \ .
\end{equation}
The expansion coefficients, $c_k^{(1,n)}$ ($c_k^{(1,a)}$), $k = 1, \ldots, 50$,  are computed 
as a least-squares solution to \eqref{eqn:gamma_lag} after specifying $\gamma^{(1,n)}$ ($\gamma^{(1,a)}$).
They are plotted in Fig. (\ref{fig:ck2}) (top row).  
The first order kernel is depicted in the time and frequency domains in Fig. (\ref{fig:tf}) (top row).

The second-order kernels, $\gamma^{(2,n)}$,  $\gamma^{(2,a)}$ are specified in a fashion akin to that
for the first-order kernels, save that after computing the expansion coefficients $c_k^{(2,n)}$, $c_k^{(2,a)}$ as above,  
\begin{equation}
  \gamma^{(n,2)}_{t_1,t_2}  = \sum_{k=1}^{50} c_k^{(2,n)} g_{k,t_1} g_{k,t_2} \ .
    \label{eqn:gamma_k2}
\end{equation}
In this procedure $c_k^{(2,a)}$ (Fig. (\ref{fig:ck2})) is chosen such that the linear combination
\begin{equation}
  \sum_{k=1}^{50} c_k^{(2,a)} g_{k,t} \ , 
\end{equation}
plotted in the time (Fig. (\ref{fig:laguerre_k2}), middle row, left) and in the frequency (Fig. (\ref{fig:laguerre_k2}), middle-row, right) domains,
result in the second order kernel shown in Fig. (\ref{fig:tf2}). 
%
\begin{figure}
  \begin{center}
    \includegraphics[width=3.000in]{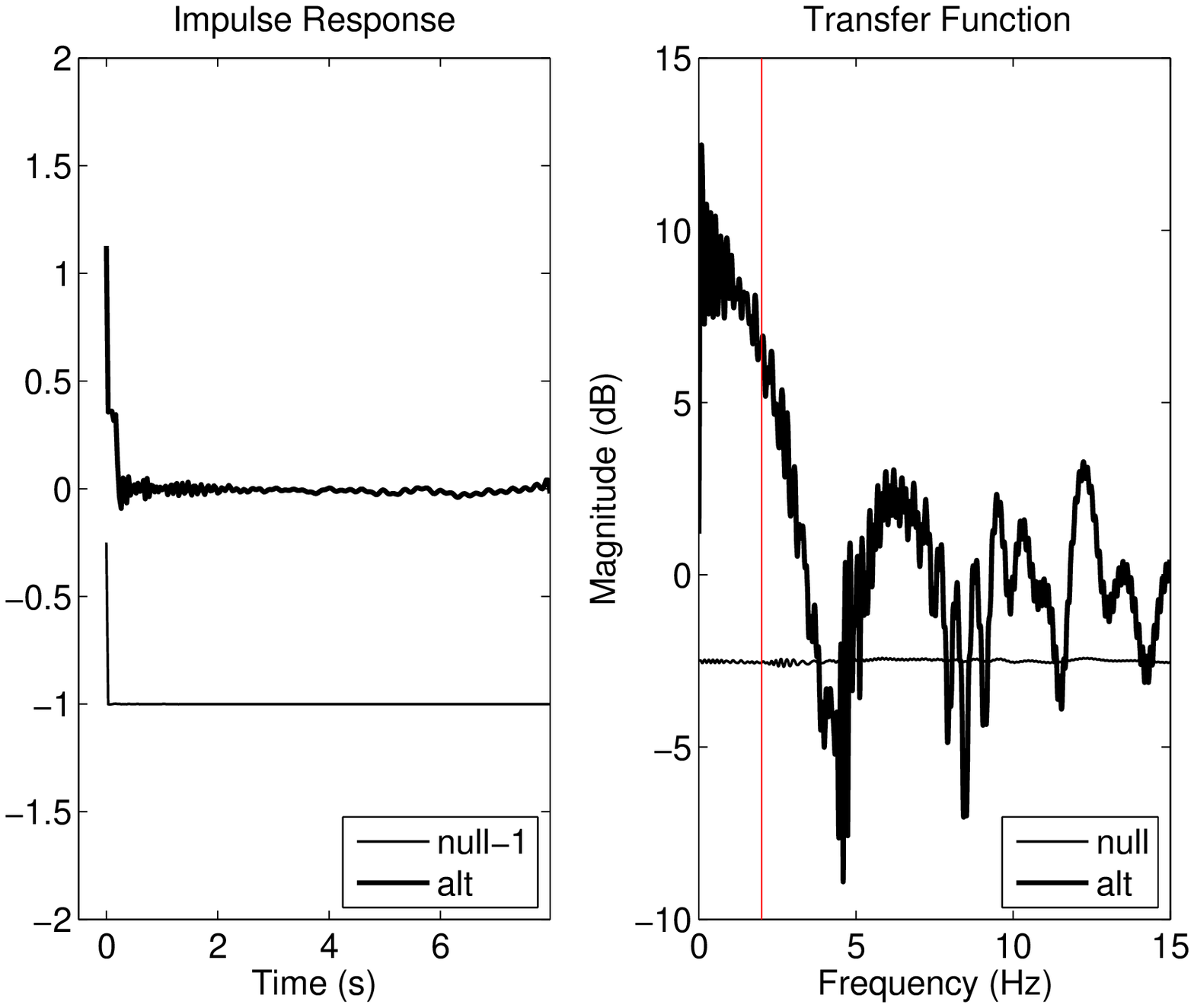}
    \includegraphics[width=3.00in]{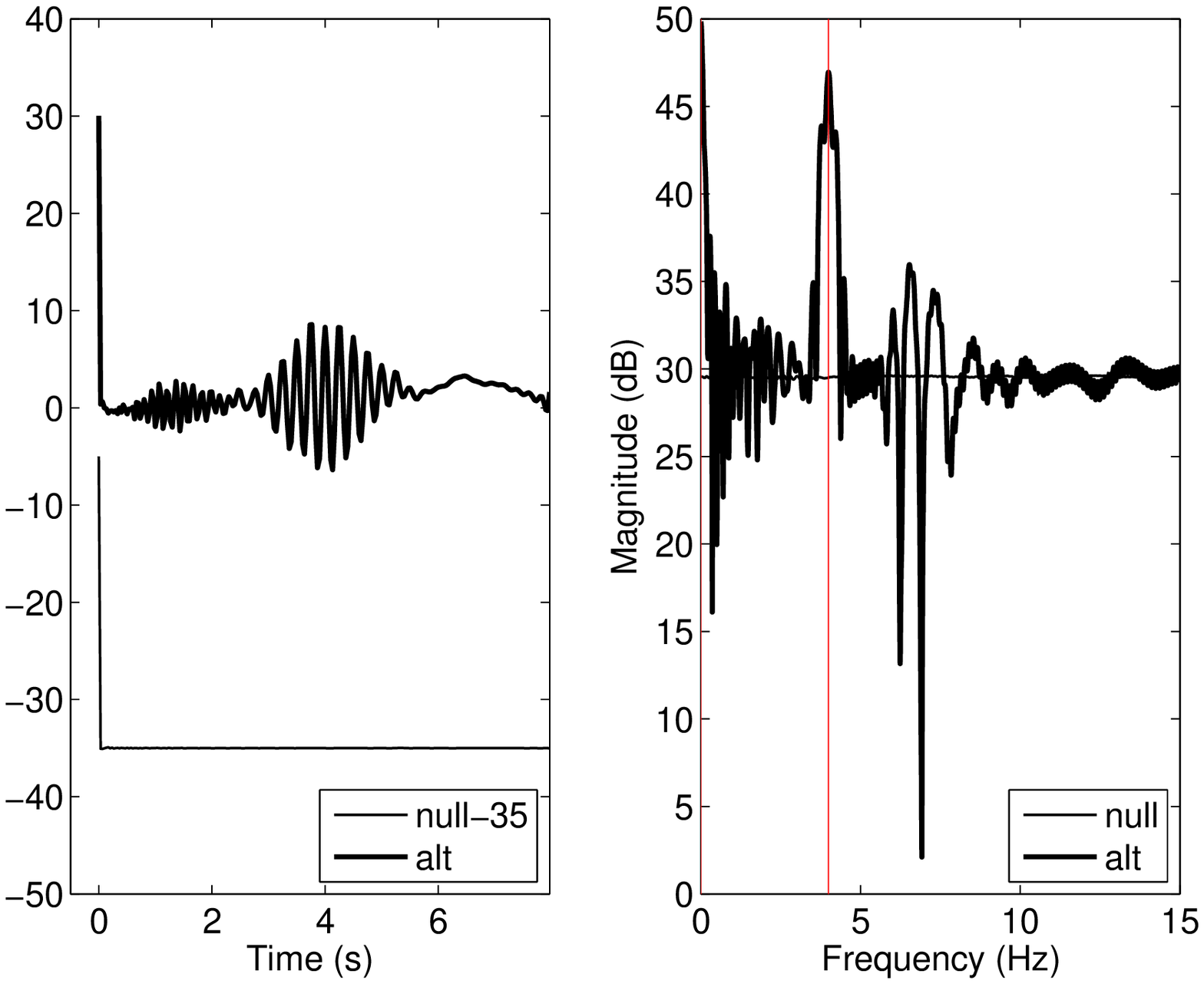} 
    \includegraphics[width=3.00in]{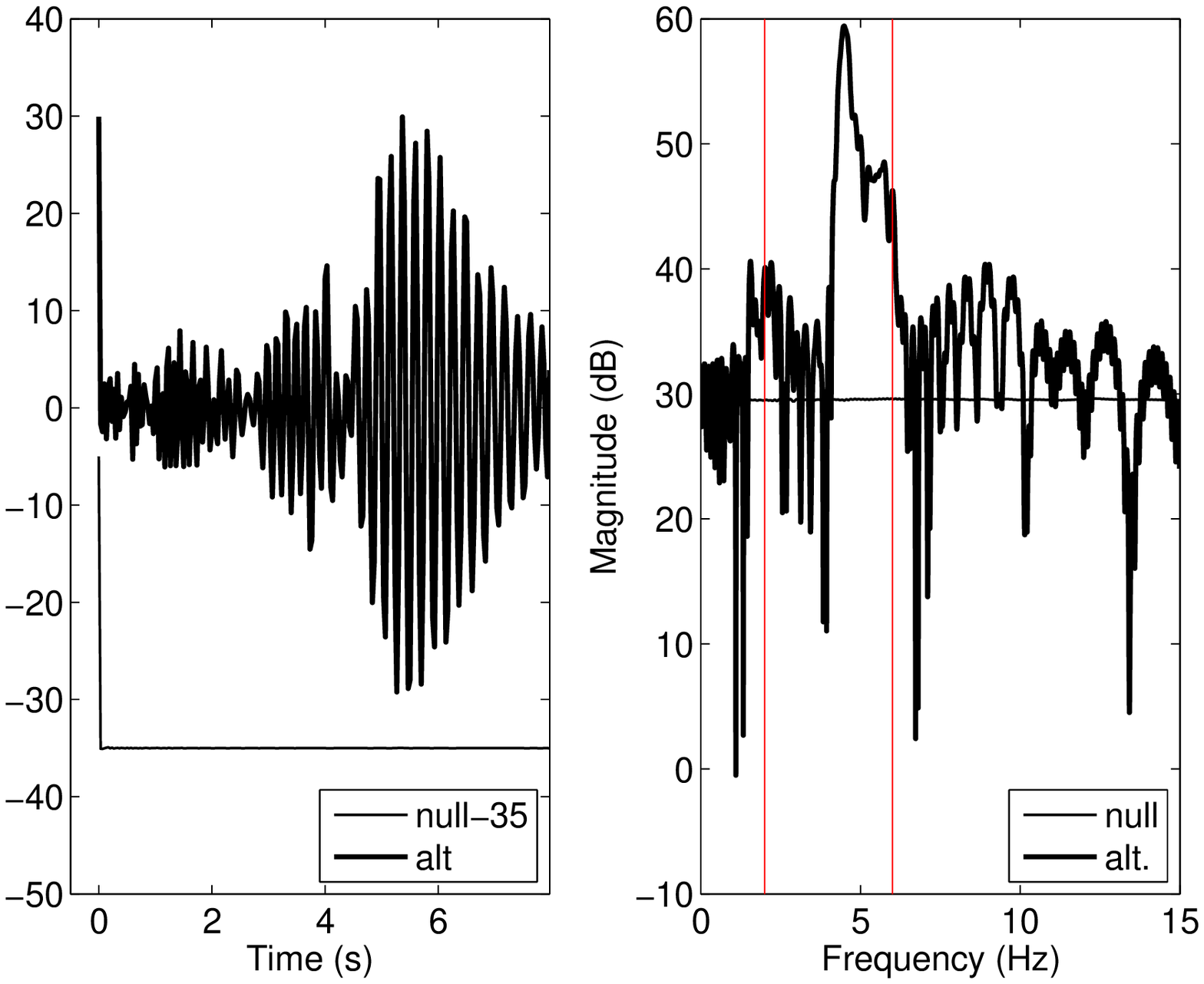} 
  \caption{\label{fig:tf} \label{fig:laguerre_k2} \label{fig_gamma3} {\bf 1D kernel representations.}\newline
    Time domain (left column). Frequency domain (right column).\vspace{.15cm}\newline
    {\bf Top row, left:}   $\gamma^{(1,n)}_t$ (thin), $\gamma^{(1,a)}_t$ (thick). \vspace{.13cm}\newline\vspace{0cm}
      {\bf Top row, right:}   $\Gamma^{(1,n)}$ (thin), $\Gamma^{(1,a)}(f)$ (thick). \newline\vspace{.12cm}
      {\bf Middle row, left:}   $\sum\limits_{k=0}^{50} c^{(2,n)}_k g_{k,t}$ (thin), $\sum\limits_{k=0}^{50} c^{(2,a)}_k g_{k,t}$ (thick). \newline\vspace{.12cm} 
      {\bf Middle row, right:}   Magnitude of DFT of curves plotted on left. \newline\vspace{.12cm} 
      {\bf Bottom row, left:}   $\sum\limits_{k=0}^{50} c^{(3,n)}_k g_{k,t}$ (thin), $\sum\limits_{k=0}^{50} c^{(3,a)}_k g_{k,t}$ (thick). \newline\vspace{.12cm}
      {\bf Bottom row, right:}   Magnitude of DFT of curves plotted on left.  
  } 
  \end{center}
\end{figure}
\begin{figure}
  \begin{center}
  \includegraphics[width=3.20in]{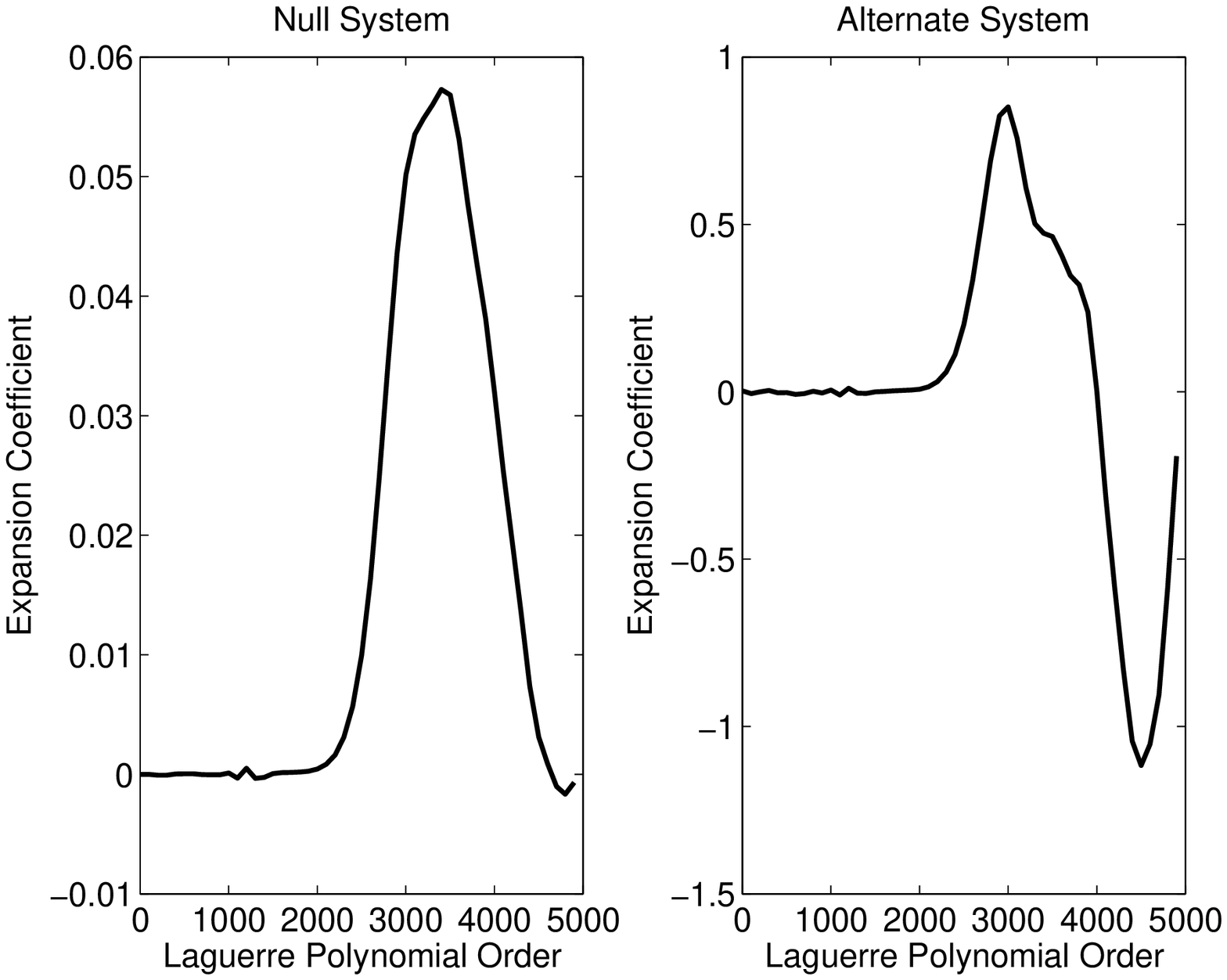}
  \includegraphics[width=3.2in]{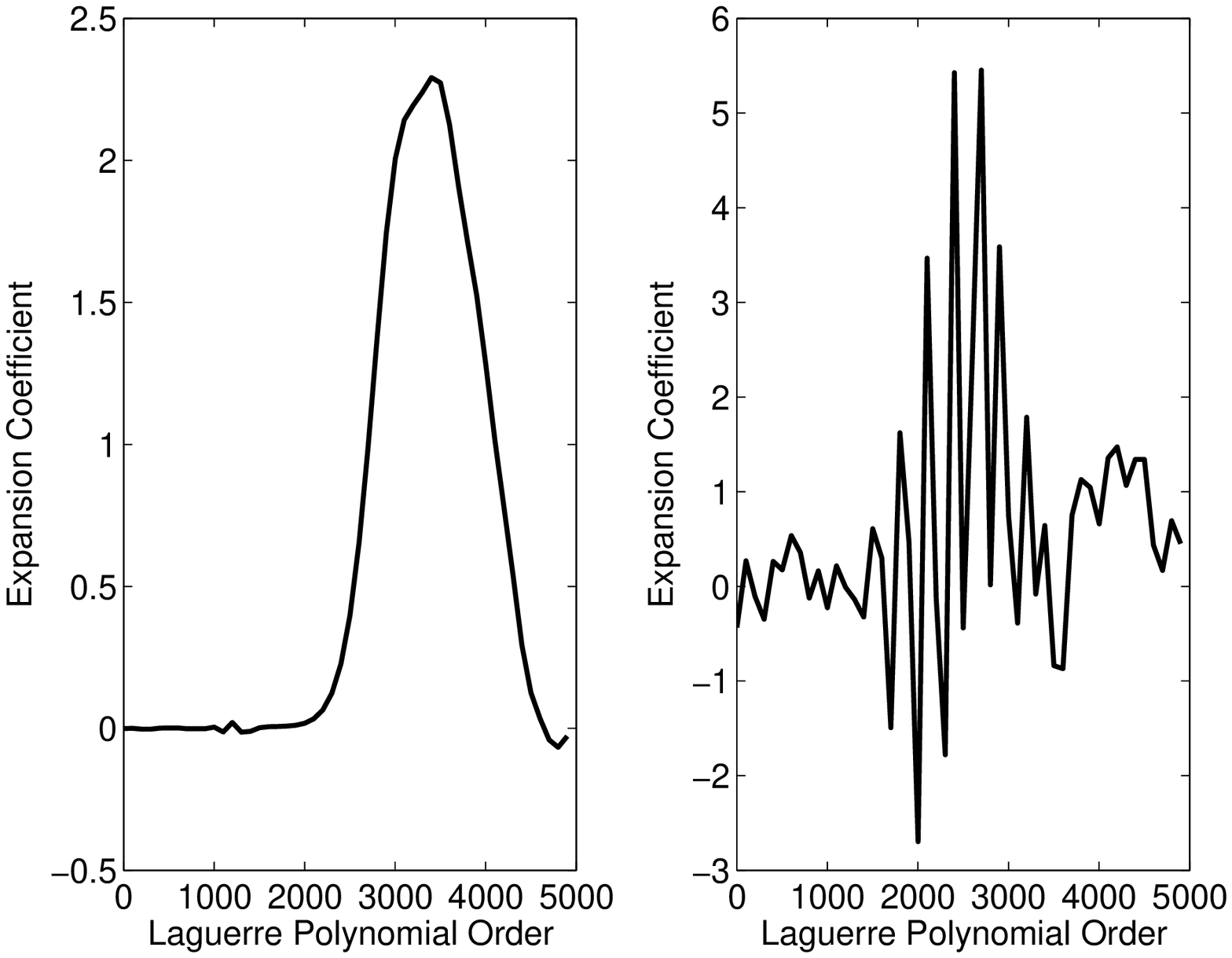} 
  \includegraphics[width=3.20in]{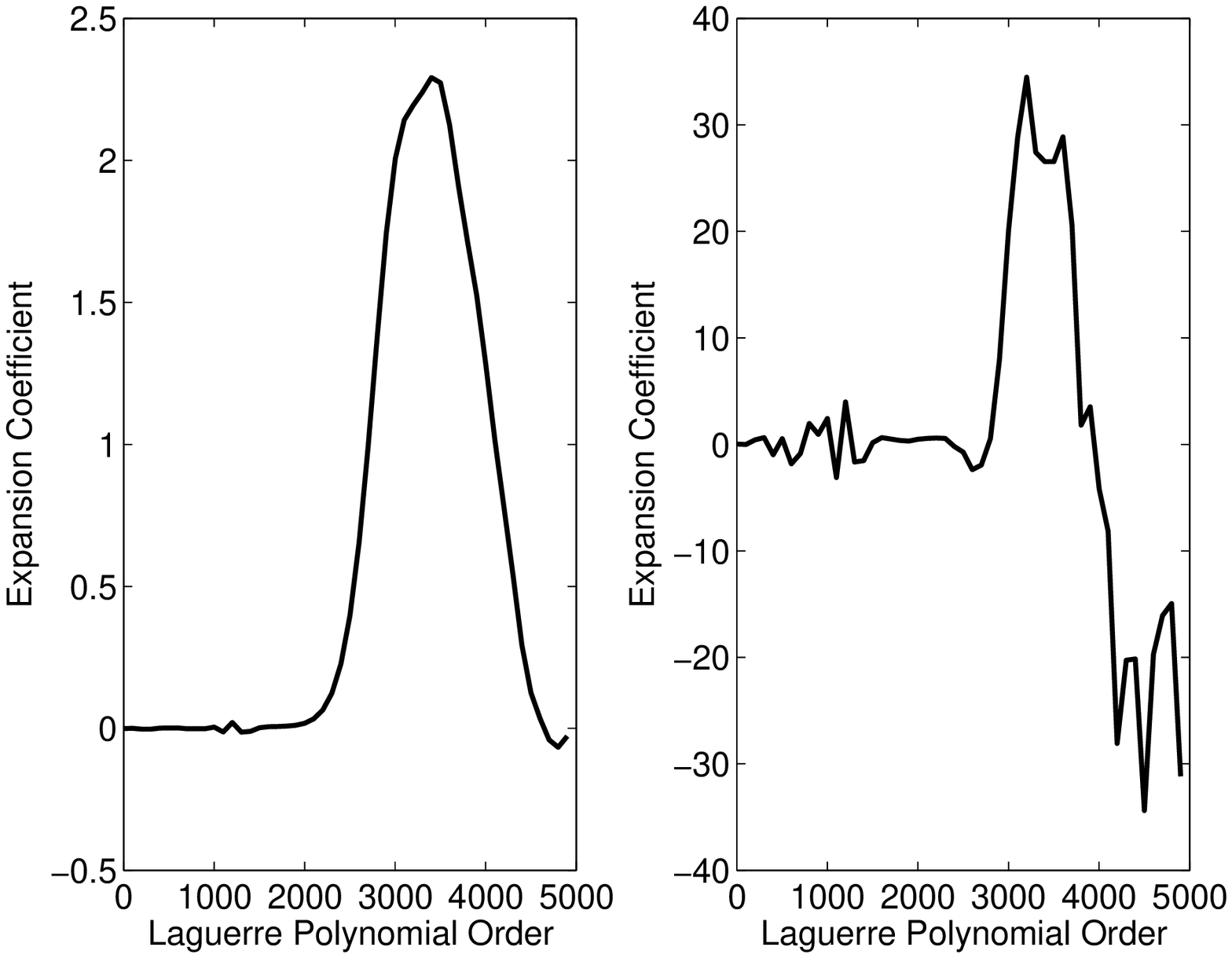} 
  \caption{{\bf Expansion coefficients.} For $k = 1, \ldots, \ 50$:\vspace{.1cm} \newline     
    {\bf Top left:} $c_k^{(1,n)}$. {\bf Top right:} $c_k^{(1,a)}$.\newline \vspace{.1cm}
    {\bf Middle left:} $c_k^{(2,n)}$. {\bf Middle right:} $c_k^{(2,a)}$.\newline \vspace{.1cm}
    {\bf Bottom left:} $c_k^{(3,n)}$. {\bf Bottom right:} $c_k^{(3,a)}$.\newline \vspace{.1cm}
        \label{fig:ck2}
        \label{fig:ck3}
\vspace{2cm}
  }
  \end{center}
\end{figure}
\begin{figure}
  \begin{center}
  \includegraphics[width=3in]{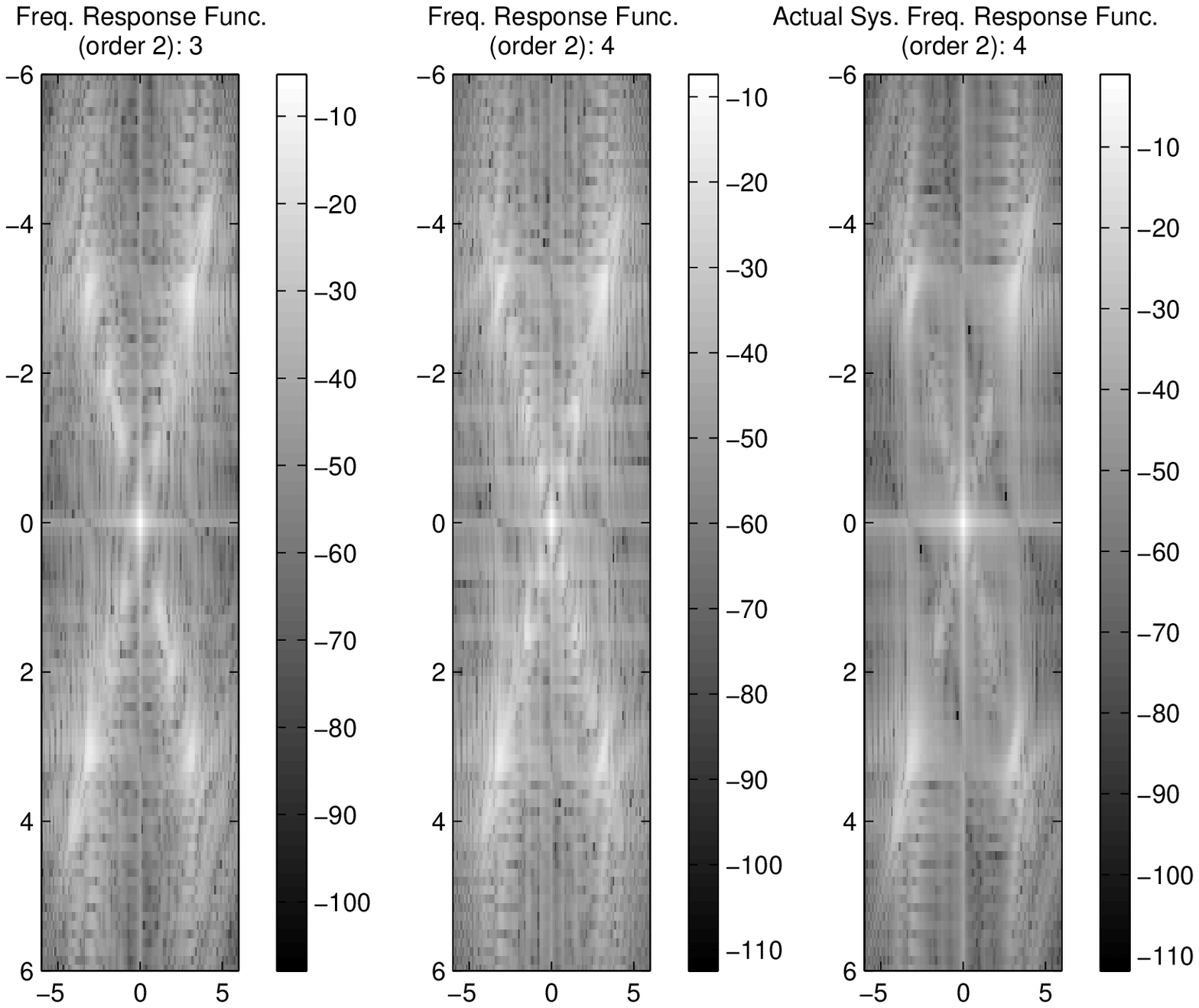}
  \caption{\label{fig:tf2} {\bf The magnitude of the second order general frequency response function, $\left| \Gamma^{(2)}\right|$.} 
  Estimates corresponding to Gaussian white input (left), and M-sequence input (middle). The  
    input signal-to-output-noise ratio (SNRI) is equal to $4 \times 10^6$.  The system response due
    to the second and third order kernels is scaled to have an output sample standard deviation equal to $1.4 \times 10^2$ (variance
    equal to $2 \times 10^4$).  Right: The magnitude of the actual second-order frequency response funciton.
  } 
\end{center}
\end{figure}
The third-order Volterra kernel is specified in a similar fashion:  
\begin{equation}
  \gamma^{(n,3)}_{t_1,t_2,t_3} = \sum_{k=1}^{50} c_k^{(n,3)} g_{k,t_1} g_{k,t_2} g_{k,t_3} \ .
  \label{eqn:gamma_k3}
\end{equation}
The resulting coefficients $c_k^{(n,3)}$, $c_k^{(a,3)}$ are plotted in Fig. (\ref{fig:ck3}) and 
the time and Fourier representations of the associated linear combination of Laguerre polynomials is
plotted in Fig. (\ref{fig:ck3}).  The $3^{\rm rd}$ order kernel for the alternate system is chosen such that 
there is an elevated response at $2 \ {\rm Hz}$ and at $6 \ {\rm Hz}$ such that narrowband system excitation centered
upon $2 \ {\rm Hz}$ excites a strong $3^{\rm rd}$-order Volterra response.

For both the null and alternate systems, narrowband response detection is performed using four methods.  These 
are direct kernel identification (or estimation) using (i) white Gaussian input, and (ii) white M-sequence input.  
Narrowband response excitation results from (iii) an input sequence equal to a sum of sinusoids (SSR) with in-band frequencies, 
and (iv) discrete prolate spheroidal stimulation.  In cases (i)-(iv) the output is added to Gaussian white noise with a variance of
$1$,  and the input signal energy (sum of square sequence elements) is set to one of $4 \times 10^3$, $4 \times 10^4$, or $4 \times 10^6$.  
The higher-order response,
consisting of the response of the second and third order Volterra components, is scaled by one of the factors $2 \times 10^{-6}$, $4 \times 10^{-6}$ or $6 \times 10^{-6}$.  
Ten repetitions of each
input, signal energy and higher-order response scale are performed.  

For input types (i) and (ii) kernel identification is performed assuming a third-order Volterra system using the least-squares
identification procedure specified in \cite{laguerre_marmarelis1993identification}. Unlike in \cite{laguerre_marmarelis1993identification}, 
instead of optimizing the $\alpha$ parameter parameterizing the Laguerre basis (here it is set to zero), the Lagurre expansion is restricted to
include Laguerre polynomials with an order equal to an integer multiple of $100$.\footnote{More accurately, the Associated Laguerre polynomials are
parameterized by $\alpha$.  The Laguerre polynomials result when $\alpha$ equals zero.}  
  For input (i) and (ii), linear narrowband response is taken to be the collection
of values $|\Gamma^{(1,n)}(f)|$ ($|\Gamma^{(1,a)}(f)|$), $f = 2-W, 2-W+df\ \ldots,\ 2+W-df, \ 2+W$. Here $W$ is varied over the values 
$.5 \ {\rm Hz}$, $.75 \ {\rm Hz}$, and $1\ {\rm Hz}$, corresponding respectively to $NW$ equal to $4$, $6$, and $8$.
Due to sampling, input types (i), (ii) depend on $W$; however, for simplicity this dependence is ignored and the results
of simulations involving input (i), (ii) using different values of $W$ are combined.  
In total, $3 \times 3 \times 3 \times 10$ simulations are performed for each of the system types (null \& alternate) resulting in $270$
simulations involving input types (i) and (ii).   For input (iii), (iv) this number is $810$, as every value of $W$ is simulated.  
The SSR input (iii) is comprised of a sum of cosines.  Each
cosine is at a frequency, $f \in (2-W,2+W)$ and the frequencies are spaced by $f_R$.
Input type (iv) is defined in Section \ref{sect:dpss}.  System response to input (iv) is studied and characterized in 
Sections (\ref{sect:result}), (\ref{sect:ortho}).  The narrowband responses to these input are taken
to be the inner-product of the output with the scaled input (scaled to possess an energy equal to $1$).  
\begin{figure}
  \begin{center}
  \includegraphics[width=3in]{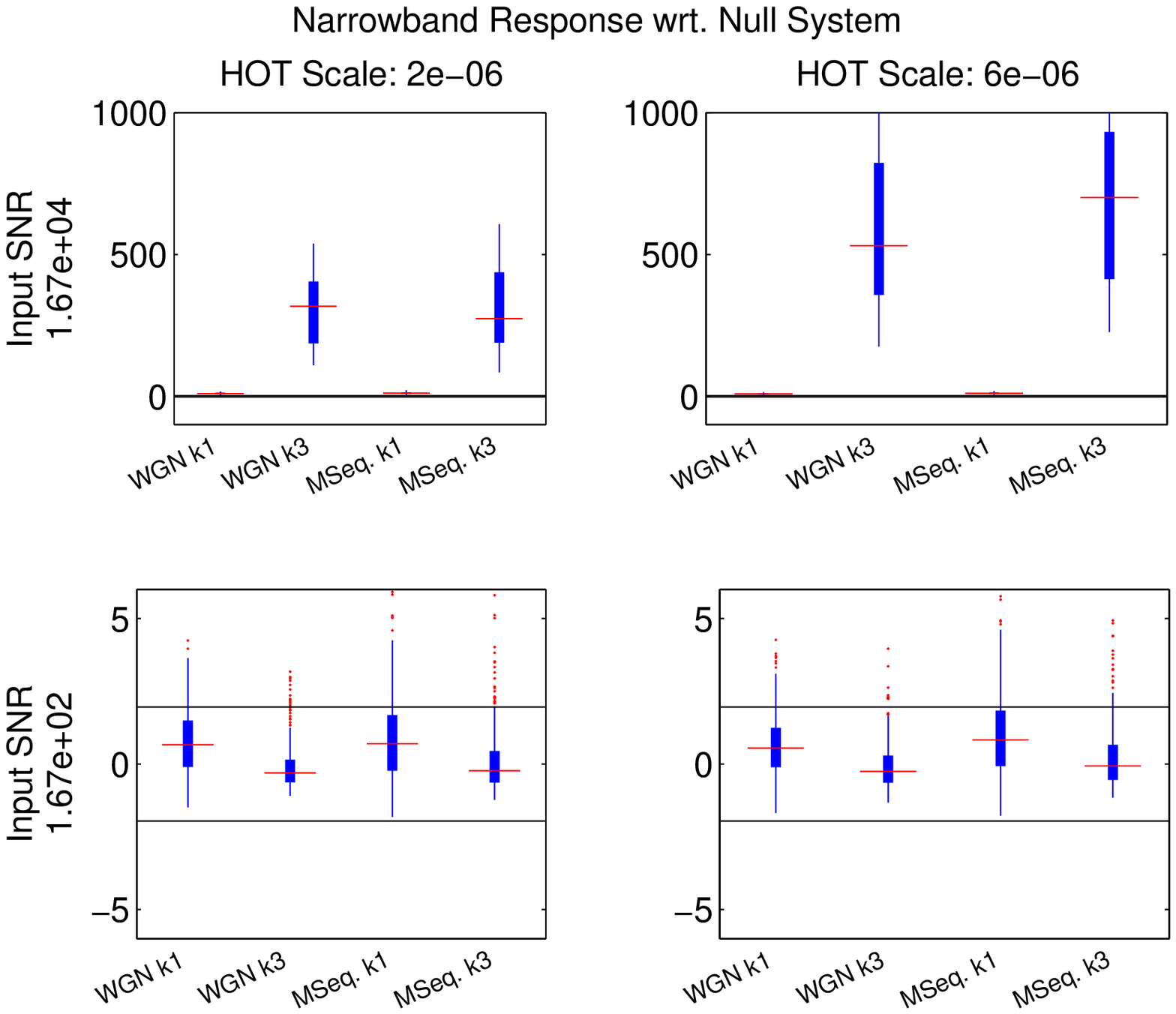}
  \caption{ \label{fig:nb-detect-kernel}
    {\bf Null system normalized magnitude of the collection of in-band generalized frequency responses.}
    Alternate system is detected for large input energy, but not for low-input energy.  Compare to Fig. (\ref{fig:nb-detect}).
    The dependence on higher-order system response is evident.
  } 
\end{center}
For the DPSS stimulation, the DPSSs are modulated to $2 \ {\rm Hz}$ by multiplication with a cosine oscillating with a frequency of $2 \ {\rm Hz}$.\footnote{The DPSSs
in this work are computed using the dpss() MATLAB function.  Literature discussing the generation of the DPSSs includes \cite{Percival&Walden,Thomson1982}.} 
As discussed in Section (\ref{sect:sys_id}), only the odd-ordered DPSS are input to the system, up to a maximum order of $2NW-1$.  The responses associated with
the $j^{\rm th}$ order modulated DPSS are associated, in Fig. (\ref{fig:nb-detect}),  with $v_{(j-1)/2}$, $j \in \left\{0,2,\ldots, 2 \lfloor (2NW-1)/2 \rfloor \right\}$. 
\end{figure}
\begin{figure}
  \begin{center}
  \includegraphics[width=3in]{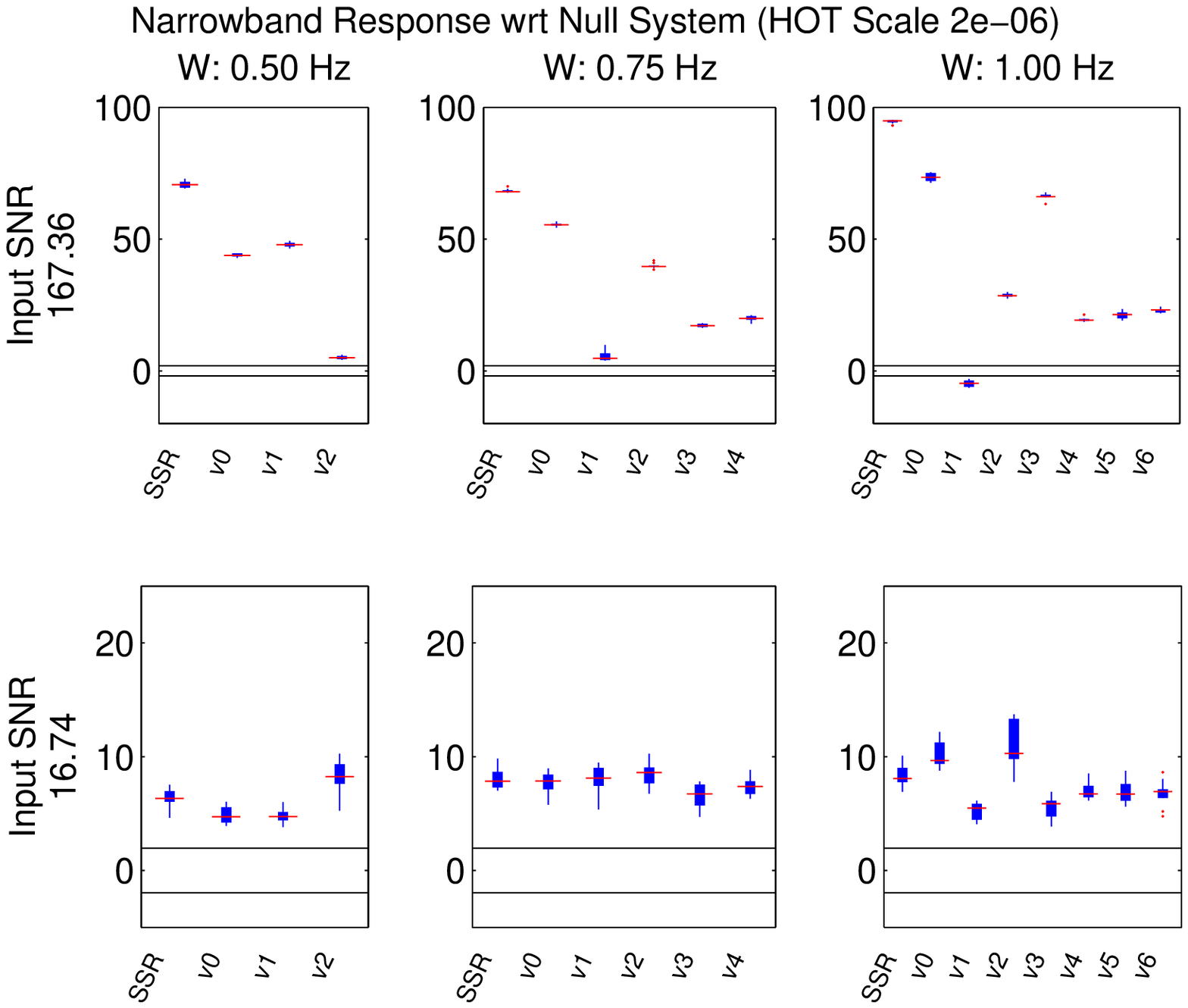}
  \caption{ \label{fig:nb-detect}
    {\bf Null system normalized magnitude of the collection of in-band inner-product detector responses.}
    Alternate system is detected in all cases.  There is a clear dependence on input energy and on $W$.
  } 
\end{center}
\end{figure}
\begin{figure}
  \begin{center}
  \includegraphics[width=3in]{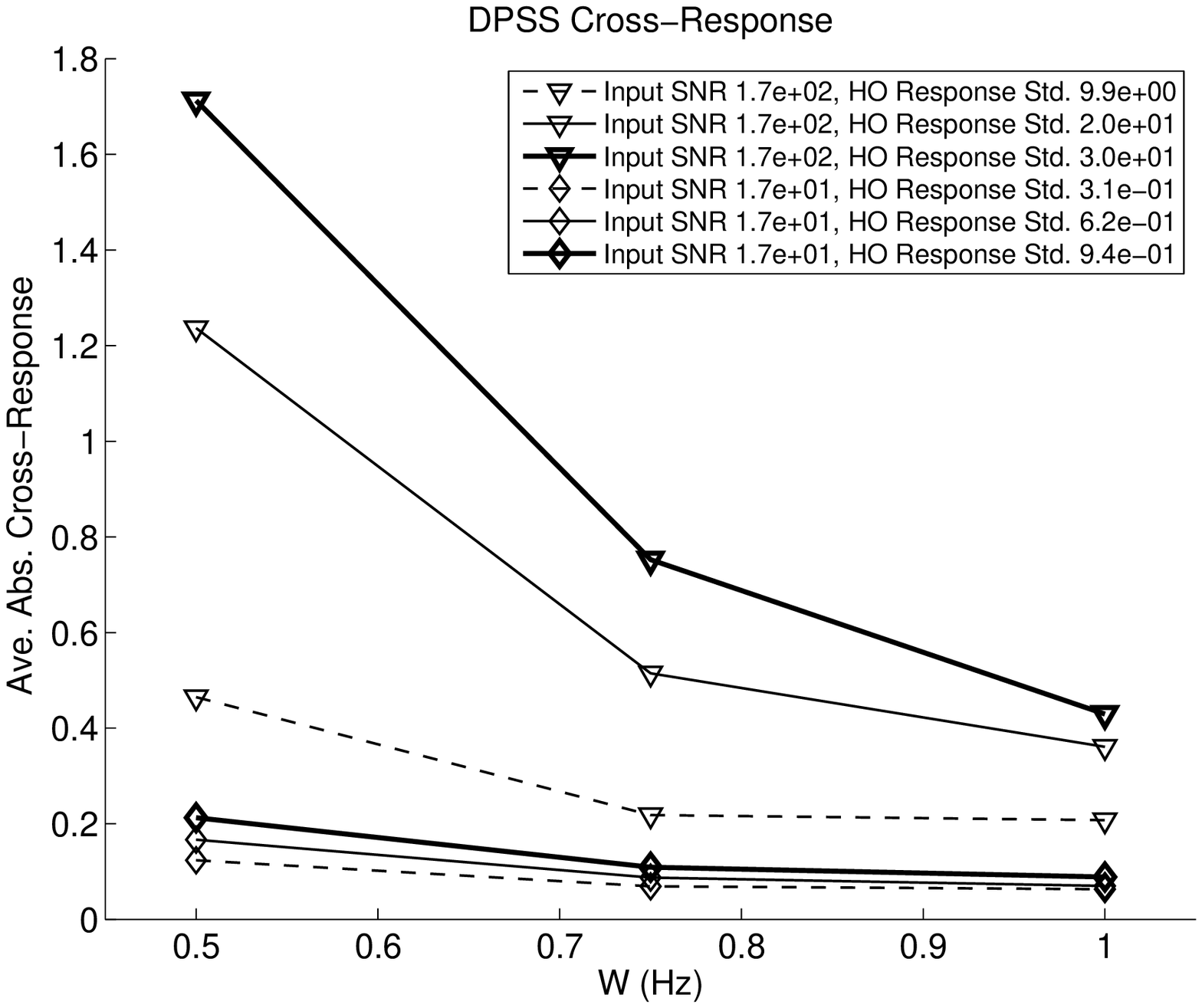}
  \caption{ {\bf DPSS orthogonality depends on $W$, input energy, and on higher-order response.}
    Vertical axis is the sample average of the magnitude of the cross-DPSS products, normalized by the 
    square root of the product of the self-DPSS products.  Here $1-\lambda_j$, for a given DPSS order $j$, decreases with 
    increasing $W$, decreasing out-of-band contribution.
    \label{fig:x}
  } 
\end{center}
\end{figure}

When testing for narrowband response, some performance is due to the choice of detector (choice of hypothesis test).  
To attribute performance to the method
of system identification, for each input method, null system responses are used to normalize the responses of the alternate system.
Let, respectively, $a$, and $\sigma$ be the sample average and the sample standard deviation of the null system responses. 
The alternate system responses
are reduced by $a$ and then divided by $\sigma$ to obtain normalized alternate system responses.
Under the null hypothesis that there is no difference in narrowband response between the null and the alternate systems,
the resulting normalized alternate system responses are realizations of independent standard normal random variables.  Deviations of the
normalized observations from that expected of a standard Gaussian random variable provides evidence in favour of rejecting the hypothesis 
that the observations are the result of the null system.  

The normalized responses for inputs (i) and (ii) are plotted in Fig. (\ref{fig:nb-detect-kernel}).  The responses due to inputs with an 
energy of $4 \times 10^4$ (bottom row) are consistent
with those expected from the null system \footnote{For all plots containing horizontal black lines,  
the probability of lying within the horizontal lines is $.95$ (under the hypothesis that the responses are due to the null system).}, while the responses (top row) associated
with an input energy of $4 \times 10^6$ are not.  The normalized responses for input types (iii) and (iv) are plotted in Fig. (\ref{fig:nb-detect}) (top row: input energy equal
to $4 \times 10^4$, bottom row: input energy equal to $4 \times 10^3$). The normalized inner-product
responses due to the SSR and DPSS input are significantly different from that expected due to the null system. 
Narrowband system response is successfully detected with the inner-product detector using 
input types (iii) and (iv) that is not detected by the kernel identification approach associated with white input.  
The DPSS input yields detections comparable to that of the SSR input, while possessing a larger fraction of in-band energy for $W$ equal to $.5$
or $.75$ (Fig. (\ref{fig:nb-detect})). When $W$ is equal to $1$ the DPSS input yields superior detection performance to that of the SSR input.

Fig. (\ref{fig:x}) depicts the average of the normalized cross-product responses.  Here a cross-product response is
defined to be the absolute value of the inner-product between the response to DPSS stimulation $v_j$ with DPSS $v_{j'}$, $j' \neq j$.
The normalization is performed by dividing the cross-response by the square-root of the absolute value of the product of the self-responses.
The cross-responses become smaller with increasing $W$, with decreasing input signal energy and with decreasing higher-order system response scale.
\comment
{
Each of the
different methods is associated with a different type of input signal.
varying the input signal energy over three values: $.4$, $4$, and $40$.  Similarly, for each of these signal input energies,
the nonlinear contribution to the system response is scaled by $2$, $4$ or $6$.  
}  

\comment
{
Fig. (\ref{fig:tf}) depicts the first order response
The higher-order kernels in the alternate system 
are specified such that 

The three symmetric Volterra kernels characterizing each system are 
specified in terms of $50$ Laguerre polynomials.  The orders of the Laguerre polynomials 
The results of this paper characterize the sense in which the output of a Volterra MIMO system excited by DPSS input
are orthogonal, and further, 
} 

%% file: disc3.tex
In this work, conditions are provided under which nonlinear system output due to discrete prolate spheroidal sequence (DPSS) input
remains approximately orthogonal to other DPSS input.
\comment
{
These conditions, expressed in terms of the Volterra expansion of the nonlinear system,
relate to the degree of system nonlinearity, the extent of system filtering, and the DPSS bandwidth.
The bounds derived in \ref{sec:content1} and \ref{sec:content2} are quantitative
statements resulting from two phenomena: (i) the DPSS are energy $1$ sequences concentrated
within the frequency interval $(-W,W)$. As the DPSS are taken to higher powers, the resulting 
sequence consists of terms that tend to zero and; (ii) due to the in-band energy concentration
of the {\color{black} DPSSs}, filtering effects are limited in the 
frequency domain to multiplicative phase factors within the frequency interval $(-W,W)$.
When the higher order system Volterra 
kernels are small relative to {\color{black} the effect in condition} (i) and are approximately constant over $(-W,W)$, then 
the DPSS remain approximately orthogonal at the system output.\footnote{Because, as in \cite{Thomson:1982}, the product of the DPSWFs 
between even/even, or odd/odd orders is even, when restricting to oddly symmetric DPSSs, this latter condition can
be relaxed to linear within $(-W,W)$, as demonstrated in \ref{sect:sim}.}
Stated differently, under conditions (i) and (ii),
 the in-band dynamics are approximately linear with 
minimal phase perturbation.  In the limiting case of 
a purely static, linear system (as described in \ref{sect:intro}) the output orthogonality is exact.

{\color{red} {\it move to discussion: As discussed in 
\cite{laguerre_marmarelis1993identification,marmarelis1997modeling}, and demonstrated in Fig. (\ref{fig:ck2}), 
for a long-memory system, such as one producing a narrowband response, a Laguerre basis requires a 
large number of non-zero coefficients.}}

This work implies two immediate generalizations which are omitted for brevity. 
Secondly, it follows from the results in this work that linear combinations of DPSSs at a single input can also be 
unmixed at {\color{black} any} output via the inner product operator.
} 
These conditions are developed under two assumptions. Specifically,
that the system admits the Volterra MIMO system representation Eqn. (9) with $y_m^{(0)}$ equal to zero (but see Remark 3), and 
(ii) that Eqn. (32) is an exact bound.  The conditions relate the DPSS bandwidth parameter $W$, the DPSS eigenvalues (through $\lambda_{min}$),
and suprema over the Volterra kernels to higher-order system response.

These properties facilitate  linear narrowband detection in network settings using multiple simultaneous DPSS input, at
both baseband, and higher frequencies.  The performance of these detectors, owing to higher-order nonlinear suppression,
are expected to  approximate matched-filter type detectors with optimal performance where system output is added to noise
prior to observation.  The further development of these notions is left for future study.  

Important limitations to this work are (i) 
due to higher-order nonlinear system response suppression, the utility of the results of this work 
for understanding the full nonlinear system dynamics is expected to be limited and (ii) the results presented in this work require 
  that the Volterra model accurately describes system dynamics.  This precludes, for example, sub-harmonic generation. 

\comment
{
As suggested in \ref{sect:intro}, the results derived herein 
may be of immediate use in domain specific applications in neural engineering. 
Here, a brain network is stimulated to learn its interconnections.  
{\color{black} Other candidate applications } exist in network science, where it is desired to 
attribute certain output effects to one of multiple sources exciting an interconnected system.  
In general, {\color{black}suitable} applications are expected to include situations where system knowledge is limited and it is desirable
 to pass an input through a nonlinear MIMO system relatively unmodified, with maximal energy concentration within a specific frequency interval.
}

In summary, this work provides a mathematical quantification of DPSS stimulation of {\color{black}MIMO Volterra nonlinear systems}.  It 
provides a {\color{black}mathematically principled} foundation from which to further develop nonlinear 
system identification methodology and from which to build applications in engineering and network science. 

%% file: app_gfrf.tex
\label{app:gfrf}
On output channel $m$, the $Q^{\rm th}$-order system response $T_{m,Q}$ as a function of frequency $f$ is
the discrete Fourier transform of the $Q^{\rm th}$ order Volterra response:\footnote{Here the Volterra kernels are indexed at negative lags.  To avoid
an acausal system response, a necessary condition is that these kernels are set equal to zero for negative lag index. With this further stipulation 
\eqref{eqn:T1} is equivalent to \eqref{eqn:volterra_td0}.}
\begin{eqnarray}
  T_{m,Q}(f) 
  &=& \mathcal{F} \bigg\{ 
            \sum_{ {\bf m}_Q = {\bf 1} }^M \sum_{\tau_1,\ldots \tau_Q=-\infty}^\infty \gamma^{(Q)}_{m, {\bf m}_Q, \tau_1, \ldots, \tau_Q} \times \nonumber  \\
        && \hspace{.1cm} u_{m_1, t - \tau_1}  u_{m_2, t - \tau_2} \ \ldots\ u_{m_Q, t - \tau_Q} \bigg\} \ ,  \nonumber \\ 
  &=& \sum_{t=-\infty}^\infty \bigg( 
            \sum_{ {\bf m}_Q = {\bf 1} }^M \sum_{\tau_1,\ldots, \tau_Q=-\infty}^\infty \gamma^{(Q)}_{m, {\bf m}_Q, \tau_1, \ldots, \tau_Q} \times \nonumber  \\
        && \hspace{.1cm} u_{m_1, t - \tau_1}  u_{m_2, t - \tau_2} \ \ldots\ u_{m_Q, t - \tau_Q} \bigg) \nonumber \\ 
        && \hspace{.1cm} e^{-i 2 \pi f t } \ . \nonumber \\
        \label{eqn:T1}
\end{eqnarray}
Substituting the Fourier representation
  \begin{equation}
    u_{m_j,t} = \int_{-\frac{1}{2}}^\frac{1}{2} U_{m_j}(f') \ e^{i 2 \pi f' t} \ df' \ , 
  \end{equation}
  into \eqref{eqn:T1} results in:
  \begin{eqnarray}
        \lefteqn{T_{m,Q}(f)  =}&& \nonumber \\
  && \sum_{t=-\infty}^\infty \bigg( 
            \sum_{ {\bf m}_Q = {\bf 1} }^M 
                  \sum_{\tau_1,\ldots,\tau_Q=-\infty}^\infty 
            \gamma^{(Q)}_{m, {\bf m}_Q, \tau_1, \ldots, \tau_Q} 
                  \times \nonumber  \\
        && \hspace{.1cm} \prod_{j=1}^Q \int_{-\frac{1}{2}}^\frac{1}{2} U_{m_j}(f_j) e^{i 2 \pi f_j (t-\tau_j)} df_j \bigg) e^{-i 2 \pi f t } \ , \nonumber \\
  &=& \sum_{ {\bf m}_Q = {\bf 1} }^M  
                  \sum_{\tau_1,\ldots,\tau_Q=-\infty}^\infty \gamma^{(Q)}_{m, {\bf m}_Q, \tau_1, \ldots, \tau_Q}\times \nonumber  \\
&& \hspace{.1cm}
        \prod_{j=1}^Q \int_{-\frac{1}{2}}^\frac{1}{2} U_{m_j}(f_j) e^{-i 2 \pi f_j \tau_j } \times \nonumber \\
      && \hspace{.1cm} \sum_{t=-\infty}^\infty e^{ -i 2 \pi t \left( f - \sum\limits_{j=1}^Q f_j \right)} \ df_j \ . \nonumber \\
  \end{eqnarray}
  Since $\sum\limits_{t=-\infty}^\infty e^{-i 2 \pi f t}$ is equal to the dirac-delta function $\delta(f)$:
  \begin{eqnarray}
        \lefteqn{T_{m,Q}(f)  =}&& \nonumber \\
  && \sum_{ {\bf m}_Q = {\bf 1} }^M \sum_{\tau_1,\ldots,\tau_Q=-\infty}^\infty \gamma^{(Q)}_{m, {\bf m}_Q, \tau_1, \ldots, \tau_Q} \nonumber \\
&& \hspace{.1cm}
        \prod_{j=1}^Q \int_{-\frac{1}{2}}^\frac{1}{2} U_{m_j}(f_j) e^{-i 2 \pi f_j \tau_j } \ \delta\bigg( f - \sum\limits_{j=1}^Q f_j \bigg) \ df_j \ , \nonumber \\
  &=& \sum_{ {\bf m}_Q = {\bf 1} }^M \int_{-\frac{1}{2}}^\frac{1}{2} U_{m_1}(f_1) \times \ldots \times U_{m_Q}(f_Q)  \ \delta\bigg( f - \sum\limits_{j=1}^Q f_j \bigg)  \nonumber \\
  && \hspace{.1cm} \sum_{\tau_1,\ldots,\tau_Q=-\infty}^\infty \gamma^{(Q)}_{m, {\bf m}_Q, \tau_1, \ldots, \tau_Q} 
        e^{-i 2 \pi \sum\limits_{j=1}^Q f_j \tau_j } \ df_1 \times \ldots \times df_Q \ , \nonumber \\
  &=& \sum_{ {\bf m}_Q = {\bf 1} }^M  
         \int_{ -\frac{1}{2}}^\frac{1}{2} \Gamma^{(Q)}_{m, {\bf m}_Q}(f_1,\ldots,f_{Q-1}, f-\sum\limits_{j=1}^{Q-1} f_j ) \times \nonumber \\
         && \hspace{.1cm} U_{{\bf m}_Q}\bigg(f - \sum_{j=1}^{Q-1} f_j \bigg) 
            \prod_{j=1}^{Q-1} U_{m_j}( f_j ) \ d{\bf f}_{Q-1} \ . \nonumber \\
  \label{eqn:gfrf_end}
\end{eqnarray}
After notational changes, \eqref{eqn:gfrf_end} is identical to \eqref{eqn:tmqf}.

%% file: app_a.tex
\label{app:a}
\begin{eqnarray}
  \lefteqn{\left| T_{m,Q}(f) - T^{(i)}_{m,Q}(f) \right| =}&& \hspace{3cm} \left| T^{(o)}_{m,Q}(f) \right| \nonumber \\
 &=& \bigg| \sum_{ {\bf m}_Q = {\bf 1}_Q }^M \tildeint_{-\frac{1}{2} }^\frac{1}{2}
                        \Gamma^{(Q)}_{m, {\bf m}_Q}\left( {\bf f}_{Q-1}, f - {\bf f}_{Q-1}^T {\bf 1}_{Q-1} \right) \nonumber \\
        && \hspace{.1cm} V_{m_Q}\left( f - {\bf f}_{Q-1}^T {\bf 1}_{Q-1} \right) 
              \prod_{j=1}^{Q-1} V_{m_j}\left( f_j \right)  \ d{\bf f}_{Q-1} \bigg| \ , \nonumber \\
  &\leq&
    \sum_{ {\bf m}_Q = {\bf 1} }^M \tildeint_{-\frac{1}{2} }^\frac{1}{2}
                        \bigg| \Gamma^{(Q)}_{m, {\bf m}_Q}\left( {\bf f}_{Q-1}, f - {\bf f}_{Q-1}^T {\bf 1}_{Q-1} \right) \nonumber \\
        && \hspace{.1cm} V_{m_Q}\left( f - {\bf f}_{Q-1}^T {\bf 1}_{Q-1} \right) \bigg|
              \prod_{j=1}^{Q-1} \bigg| V_{m_j}\left( f_j \right) \bigg|  \ d{\bf f}_{Q-1} \ , \nonumber \\
  &\leq&
    \Gamma^{(Q)}_{m,*} V_{M,*}
    \sum_{ {\bf m}_Q = {\bf 1}_Q }^M \tildeint_{-\frac{1}{2} }^\frac{1}{2}
              \prod_{j=1}^{Q-1} \bigg| V_{m_j}\left( f_j \right) \bigg|  \ d{\bf f}_{Q-1} \ , \nonumber \\
  &\leq&
    \Gamma^{(Q)}_{m,*} V_{M,*}
    \sum_{ {\bf m}_Q = {\bf 1}_Q }^M 
              \prod_{j=1}^{Q-1} \sqrt{ 1 - \lambda_{m_j} } \ , \nonumber \\
  &\leq&
    \Gamma^{(Q)}_{m,*} V_{M,*}
              M^Q \left( 1 - \lambda_{min} \right)^{(Q-1)/2} \ , \nonumber \\
  &=& A_{m,M,Q}(\lambda_{min}, V_{M,*}, \Gamma^{(Q)}_{m,*} )\ .
              \label{eqn:Alambdamin}
\end{eqnarray}
Here
\begin{eqnarray}
  \lefteqn{\Gamma^{(Q)}_{m,*} =} && \nonumber \\
    && \sup\limits_{\substack{ {\bf m}_Q \in \left\{1,2,\ldots,M\right\}^Q \\ {\bf f}_{Q-1} \in (-\frac{1}{2},\frac{1}{2})^{Q-1} \\ f \in (-\frac{1}{2},\frac{1}{2})} } \bigg| \Gamma^{(Q)}_{m, {\bf m}_Q}\left( {\bf f}_{Q-1}, f - {\bf f}_{Q-1}^T {\bf 1}_{Q-1} \right) \bigg| \ , \nonumber \\
\end{eqnarray}
and
\begin{eqnarray}
  \lefteqn{V_{M,*} =} && \nonumber \\
    && \sup\limits_{\substack{      m_Q \in \left\{1,2,\ldots,M\right\}   \\ {\bf f}_{Q-1} \in (-\frac{1}{2},\frac{1}{2})^{Q-1} \\ f \in (-\frac{1}{2},\frac{1}{2})} } 
  \bigg| V_{m_Q}\left( f - {\bf f}_{Q-1}^T {\bf 1}_{Q-1} \right) \bigg| \ . \nonumber \\
\end{eqnarray}
Note that both $\lambda_{min}$, and $V_{M,*}$ depend upon $W$ implicitly.

%% file: app_b3.tex
\label{app:b}
\comment
{
Beginning with the $m^{th}$ iterated integral, with $f$ shifted by $\sum\limits_{\substack{j=1 \\ j\neq m}}^{Q-1} f_j$ consider:
\begin{eqnarray}
   \lefteqn{ \int_{-W}^W V_{m_Q}( f -   f_m ) V_m(f_m)  \ df_m = }&& \nonumber \\
  && \sum_{n,n'=0}^{N-1} v_n^{(m_Q)} v_{n'}^{(m)} \int_{-W}^W e^{-i 2 \pi n \left( f - f_m \right)} e^{-i 2 \pi n' f_m } df_m \ , \nonumber \\
  && \sum_{n,n'=0}^{N-1} v_n^{(m_Q)} v_{n'}^{(m)} e^{-i 2 \pi n f} \int_{-W}^W e^{ i 2 \pi f_m \left( n - n' \right)} df_m \ , \nonumber \\
\end{eqnarray}
\hrule
Beginning with the $m^{th}$ iterated integral, with $f$ shifted by $\sum\limits_{\substack{j=1 \\ j\neq m}}^{Q-1} f_j$ consider:
\begin{eqnarray}
   \lefteqn{ \int_{-W}^W V_{m_Q}( f -   f_m ) V_m(f_m)  \ df_m = }&& \nonumber \\
  && \sum_{n,n'=0}^{N-1} v_n^{(m_Q)} v_{n'}^{(m)} \int_{-W}^W e^{-i 2 \pi n \left( f - f_m \right)} e^{-i 2 \pi n' f_m } df_m \ , \nonumber \\
  &=& \sum_{n,n'=0}^{N-1} v_n^{(m_Q)} v_{n'}^{(m)} e^{-i 2 \pi f n } \sinc\big( 2 \pi W ( n - n' ) \big)  \ , \nonumber \\
  &=& \sum_{n =0}^{N-1} v_n^{(m_Q)} e^{-i 2 \pi f n }  \sum_{n'=0}^{N-1} v_{n'}^{(m)}  \sinc\big( 2 \pi W ( n - n' ) \big)  \ . \nonumber \\
  &=& \sum_{n =0}^{N-1} v_n^{(m_Q)} e^{-i 2 \pi f n }  \sum_{n'=0}^{N-1} v_{n'}^{(m)} \sum_{j=0}^{N-1} \lambda_j v_n^{(j)} v_{n'}^{(j)} 
            \ , \nonumber \\
  &=& \sum_{n'=0}^{N-1} v_{n'}^{(m)} \sum_{j=0}^{N-1} \lambda_j v_{n'}^{(j)} \sum_{n =0}^{N-1} v_n^{(m_Q)}  v_n^{(j)} e^{-i 2 \pi f n }  \ , \nonumber \\
  &=& \sum_{n'=0}^{N-1} v_{n'}^{(m)} \sum_{j=0}^{N-1} \lambda_j v_{n'}^{(j)} \int_{-W}^W V_{(m_Q)}(f-f')  V_j(f') df' \ , \nonumber \\
  &=& \lambda_m \sum_{n =0}^{N-1} v_n^{(m_Q)} v_n^{(m)} e^{-i 2 \pi f n }  \ , \nonumber \\
\label{eqn:aab}
  \end{eqnarray}
Substituting the spectral representation of a matrix (A17 in \cite{djt:paleo}),
\begin{equation}
  \sinc\big( 2 \pi W (n - n') \big) = \sum_{j=0}^{N-1} \lambda_j v_n^{(j)} v_{n'}^{(j)} \ ,
\end{equation}
into (\ref{eqn:aab}), obtain
\begin{eqnarray}
   \lefteqn{ \int_{-W}^W V_{m_Q}( f -   f_m ) V_m(f_m)  \ df_m = }&& \nonumber \\
  &=& \sum_{n =0}^{N-1} v_n^{(m_Q)} e^{-i 2 \pi f n }  \sum_{n'=0}^{N-1} v_{n'}^{(m)}  \sum_{j=0}^{N-1} \lambda_j v_n^{(j)} v_{-n'}^{(j)} \ , \nonumber \\
  &=& \sum_{n =0}^{N-1} v_n^{(m_Q)} e^{-i 2 \pi f n }  \sum_{j=0}^{N-1} \lambda_j v_n^{(j)}  \sum_{n'=0}^{N-1} v_{n'}^{(m)} v_{-n'}^{(j)} \ . \nonumber \\
\end{eqnarray}
Because the DPSS are zero outside of the indices $0$ to $N-1$, 
\begin{eqnarray}
   \lefteqn{ \int_{-W}^W V_{m_Q}( f -   f_m ) V_m(f_m)  \ df_m = }&& \nonumber \\
  &=& \sum_{n =0}^{N-1} v_n^{(m_Q)} e^{-i 2 \pi f n }  v_0^{(m)} \sum_{j=0}^{N-1} \lambda_j v_n^{(j)} v_{0}^{(j)} \ , \nonumber \\
  &=& v_0^{(m)} \sum_{n =0}^{N-1} v_n^{(m_Q)}  \sinc\big( 2 \pi W n \big) e^{-i 2 \pi f n } \ , \nonumber \\
  &=& v_0^{(m)} \int_{-W}^W V_{m_Q}( f - f' ) df' \ . \nonumber \\
\end{eqnarray}
Now the $Q-1$ multidimensional integral
\begin{eqnarray}
  \lefteqn{\int_{-W}^{W} V_{m_Q}( f - {\bf f}^T_{Q-1} {\bf  1}_{Q-1} ) \prod_{j=1}^{Q-1} V_{m_j}( f_j ) \ {\bf df}_{Q-1} =} && \nonumber \\
  && v_0^{m_{Q-1}} \int_{-W}^W  \int_{-W}^{W} V_{m_Q}( f - {\bf f}^T_{Q-2} {\bf  1}_{Q-2} - g_1 ) \ dg_1 \times \nonumber \\
      &&  \hspace{3cm} \prod_{j=1}^{Q-2} V_{m_j}( f_j ) \ {\bf df}_{Q-2}   \ , \nonumber          \\
  &=& \left( \prod_{j=1}^{Q-1} v_0^{m_{j}} \right) \int_{-W}^W V_{m_Q}( f - {\bf g}^T_{Q-1} {\bf  1}_{Q-1} ) \ {\bf dg}_{Q-1} \ . \nonumber \\
  \label{dpswf_multidim_integral}
\end{eqnarray}
When $Q=1$ (\ref{dpswf_multidim_integral}) is equal to $V_{m_Q}(f)$.

\hrule
Note that (for e.g., \cite{djt:paleo})
so that
  \begin{equation}
    \lambda_m v^{(m)}_0 =  \left. \sum_{n'=0}^{N-1} v_{n'}^{(m)}  \sinc\big( 2 \pi W ( n + n' ) \big) \right|_{n=0} \ . \nonumber \\
  \end{equation}
The global maximum of $\sinc\big( 2 \pi W ( n + n' ) \big)$ occurs at $n+n'$ equal to zero.
\begin{eqnarray}
  \lefteqn{ 2 \pi W ( n + n' ) = }&& \hspace{2cm} \frac{\ell \pi}{2} \ , \ell = 0, \ \pm 1, \ \ldots \nonumber \\
      n + n' &=& \frac{ \ell  }{ 4 W } \ , \nonumber \\
\end{eqnarray}
the maximum of $ \sinc\big( 2 \pi W( n+n' )\big)$ is 
\hrule
Taylor expanding about $c = 0$, 
  \begin{eqnarray}
    e^{ i 2 \pi   f_m\left( n - n' \right) } = 1 + i 2 \pi f_m n c e^{ i 2 \pi f_m \zeta n  } \ , 
  \end{eqnarray}
  with $\zeta \in (0,c)$.
Now,
  \begin{eqnarray}
   \lefteqn{ \int_{-W}^W V_{m_Q}( f - c f_m ) V_m( f_m)  \ df_m = }&& \nonumber \\
  &=& \sum_{n =0}^{N-1} v_n^{(m_Q)} e^{-i 2 \pi f n} \ \int_{-W}^W V(f_m) df_m  + \nonumber \\
  & & i 2 \pi c \sum_{n   =0}^{N-1} n v_n^{(m_Q)} e^{-i 2 \pi f n}  
      \int_{-W}^W f_m  V_m(f_m) e^{ i 2 \pi f_m \zeta n } df_m  \ . \nonumber
  \end{eqnarray}
Continuing,
  \begin{eqnarray}
  \lefteqn{       \int_{-W}^W V_{m_Q}( f - c f_m ) V_m( f_m)  \ df_m  - V_{m_Q}(f) v^{(m)}_0            = }&& \nonumber \\
  & & i 2 \pi c  \int_{-W}^W f_m  V_m(f_m) \sum_{n   =0}^{N-1} n v_n^{(m_Q)} e^{-i 2 \pi n \left( f - f_m \zeta\right)  }  \ , \nonumber \\
\end{eqnarray}
  \begin{eqnarray}
  \lefteqn{\left| \int_{-W}^W V_{m_Q}( f - c f_m ) V_m( f_m)  \ df_m  - V_{m_Q}(f) v^{(m)}_0 \right| \leq }&& \nonumber \\
  && 2 \pi c \sum_{n   =0}^{N-1} n v_n^{(m_Q)} \left| e^{-i 2 \pi f n} \right| \times \nonumber \\
      && \left| \int_{-W}^W f_m  V_m(f_m) e^{ i 2 \pi f_m \zeta n } df_m \right| \ , \nonumber \\
  &\leq& 2 \pi c \sum_{n=0}^{N-1} n v_n^{(m_Q)} \left| e^{-i 2 \pi f n} \right| \times \nonumber \\
      && \left| \int_{-W}^W f_m  V_m(f_m) e^{ i 2 \pi f_m \zeta n } df_m \right| \ , \nonumber \\
  &\leq& 2 \pi c W \sqrt{\lambda_m} \sum_{n=0}^{N-1} n v_n^{(m_Q)} \left| e^{-i 2 \pi f n} \right| \ . \nonumber \\
  \end{eqnarray}
\begin{lemma}{Approximation of the Shifted DPSWF Inner Product\newline}
  \begin{eqnarray}
    \left| \int_{-W}^W V_{m_Q}( f - {\bf f}_{Q-1}^T {\bf 1}_{Q-1} ) \prod_{j=1}^{Q-1} V_{f_j}(f_j)  \ d{\bf f}_{Q-1} \right| \leq 
      \left[ v^{(*)}_0 + \sqrt{2} W \right]^{Q-1} \ . \nonumber
  \end{eqnarray}
  \label{lemma:V_integral_bnd}
\end{lemma}
Let $II$ 
\begin{eqnarray}
  \lefteqn{ II = } && \nonumber \\
  &=& \int_{-W}^W V_{m_Q}\left( f - {\bf f}_{Q-1}^T {\bf 1}_{Q-1} \right)  \prod_{j=1}^{Q-1} V_{m_j}( f_j ) \ df_j \ , \nonumber \\
\end{eqnarray}
and consider the first iterated integral, $II_1$, over $(-W,W)$:
\begin{eqnarray}
  \lefteqn{ II_1 =}& & \nonumber \\
    && \int_{-W}^W   V_{m_Q}\left( f - c f_1 \right) V_{m_1}( f_1 ) \ df_1 \ ,
\end{eqnarray}
where $c = \sum\limits_{j=2}^{Q-1} f_j$.  Now, substitute the DFT's to obtain
\begin{eqnarray}
  \lefteqn{ II_1 = } && \nonumber \\
    && \int_{-W}^W   \sum_{n=0}^{N-1} v^{(m_Q)}_n e^{-i 2 \pi n \left( f -  c f_1 \right)} \ \sum_{n'=0}^{N-1} v^{m_1}_{n'} e^{-i 2 \pi n' f_1} \ df_1 \ , \nonumber \\
    &=&   \sum_{n,n'=0}^{N-1} v^{(m_Q)}_n v^{(m_1)}_{n'} e^{-i 2 \pi n f}  \int_{-W}^{W} e^{-i 2 \pi f_1 \left( n' - c n \right)} \ df_1 \ , \nonumber \\
    &=&   \sum_{n,n'=0}^{N-1} v^{(m_Q)}_n v^{(m_1)}_{n'} e^{-i 2 \pi n f}  sinc\left( 2 \pi W (n' - c n) \right) \ , \nonumber \\
\end{eqnarray}
\hrule
\vspace{1cm}
} 
The bound in (\ref{eqn:bnd_b}) is established using the multidimensional version of Taylor's remainder theorem applied to the Taylor expansion of
a complex valued function of several real valued variables. Specifically,
\begin{eqnarray}
  \lefteqn{\left| T^{(i)}_{m,Q}(f) - T^{(i)}_{0,m,Q}(f) \right| \leq} && \hspace{3cm}                                 \nonumber \\
 && \bigg| \sum_{ {\bf m}_Q = {\bf 1} }^M \int_{-W}^W
                        {\bf f}_{Q-1}^T \times \nonumber \\
    && \int_0^1 \nabla_{{\bf f}_{Q-1}}
        \Gamma^{(Q)}_{m, {\bf m}_Q}\left( t {\bf f}_{Q-1}, f - t {\bf f}_{Q-1}^T {\bf 1}_{Q-1} \right) dt \times \nonumber \\
&& V_{m_Q}\left( f - {\bf f}_{Q-1}^T {\bf 1}_{Q-1} \right) \ \prod_{j=1}^{Q-1} V_{m_j}( f_j ) \ d{\bf f}_{Q-1} \bigg| \ . \nonumber \\
\end{eqnarray}
More manipulations yield,
\begin{eqnarray}
  \lefteqn{\left| T^{(i)}_{m,Q}(f) - T^{(i)}_{0,m,Q}(f) \right| \leq} && \hspace{3cm}                                 \nonumber \\
&&
      \int_{-W}^W \left| {\bf f}_{Q-1}^T \right| \sum_{ {\bf m}_Q = {\bf 1} }^M \times \nonumber \\
        && \int_0^1 \bigg| \nabla_{{\bf f}_{Q-1}}
        \Gamma^{(Q)}_{m, {\bf m}_Q}\left( t {\bf f}_{Q-1}, f - t {\bf f}_{Q-1}^T {\bf 1}_{Q-1} \right) \bigg|\ dt \times \nonumber \\
&& \bigg| V_{m_Q}\left( f - {\bf f}_{Q-1}^T {\bf 1}_{Q-1} \right) \prod_{j=1}^{Q-1} V_{m_j}(f_j) \bigg| \ \ d{\bf f}_{Q-1} \ , \nonumber \\
&\leq&
    W^{Q-1} \sum_{ {\bf m}_Q = {\bf 1} }^M \int_{-W}^W \times \nonumber \\
  && \int_0^1 \bigg| \nabla_{{\bf f}_{Q-1}}
        \Gamma^{(Q)}_{m, {\bf m}_Q}\left( t {\bf f}_{Q-1}, f - t {\bf f}_{Q-1}^T {\bf 1}_{Q-1} \right) \bigg|\ dt \times \nonumber \\
&& \bigg| V_{m_Q}\left( f - {\bf f}_{Q-1}^T {\bf 1}_{Q-1} \right) \prod_{j=1}^{Q-1} V_{m_j}(f_j) \bigg| \ \ d{\bf f}_{Q-1} \ . \nonumber \\
\label{eqn:tmp2}
\end{eqnarray}
Let 
\begin{eqnarray}
  \lefteqn{\Gamma^{(Q)'}_{m,*} =} && \nonumber \\
  && \sup_{\substack{ {\bf m}_Q \in \left\{1,\ \ldots, \ M\right\}^{Q} \\ f \in (-W,W) \\ {\bf f}_{Q-1} \in \left(-W,W\right)^{Q-1}}} \nonumber \\
  && \int_0^1 \bigg| \nabla_{{\bf f}_{Q-1}}\bigg\{ 
        \Gamma^{(Q)}_{m, {\bf m}_Q}\left( t {\bf f}_{Q-1}, f - t {\bf f}_{Q-1}^T {\bf 1}_{Q-1} \right) \bigg\} \bigg|\ dt \ . \nonumber \\
  \label{eqn:gamma_deriv_max}
\end{eqnarray}
Then, using \eqref{eqn:abs_JJ_ub}, 
 \begin{eqnarray}
   \lefteqn{\left| T^{(i)}_{m,Q}(f) - T^{(i)}_{0,m,Q}(f) \right|  \leq} && \nonumber \\  
      && W^{Q-1} \ \Gamma^{(Q)'}_{m,*} \sum_{ {\bf m}_Q = {\bf 1} }^M  \times \nonumber \\
      && \int_{-W}^W \bigg| V_{m_Q}\left( f - {\bf f}_{Q-1}^T {\bf 1}_{Q-1} \right) \prod_{j=1}^{Q-1} V_{m_j}(f_j) \bigg| \ \ d{\bf f}_{Q-1} \ , \nonumber \\
      &=& W^{Q-1} \ \Gamma^{(Q)'}_{m,*} M^Q \ J_B(W,Q,M) \ , \nonumber \\
 &=&  B_{m,M,Q}(W,\Gamma_{m,*}^{(Q)'}) \ .
 \end{eqnarray}
\comment
{
Here
\begin{eqnarray}
  \lefteqn{D_{Q,*}    =} && \nonumber \\
&& \sup\limits_{\substack{{\bf f}_{Q-1} \in (-W,W)^{Q-1} \\ f \in (-W,W) \\ {\bf m}_Q \in \left\{ 1, 2, \ldots, M\right\}^Q} } \nonumber \\
&& \int_0^1 \bigg| \nabla_{{\bf f}_{Q-1}}\bigg\{ 
        \Gamma^{(Q)}_{m, {\bf m}_Q}\left( t {\bf f}_{Q-1}, f - t {\bf f}_{Q-1}^T {\bf 1}_{Q-1} \right) \times \nonumber \\
&& V_{m_Q}\left( f - t {\bf f}_{Q-1}^T {\bf 1}_{Q-1} \right) \bigg\} \bigg| \ dt . \nonumber \\
\end{eqnarray}
Note that $D_{1,*} = 0$.
} 